\documentclass[10pt, a4paper, reqno]{amsart}
\usepackage{amssymb,graphicx}
\usepackage{mathrsfs, bbm}
\usepackage{color}
\usepackage{enumerate}
\usepackage{a4wide}
\usepackage[colorlinks=true,citecolor=blue,linkcolor=blue]{hyperref}

\makeatletter
\@namedef{subjclassname@2020}{\textup{2020} Mathematics Subject Classification}
\makeatother

\newcommand{\R}{\mathbb{R}}

\newcommand{\C}{\mathbb{C}}
\newcommand{\Itgr}{\mathbb{Z}}
\newcommand{\Z}{\mathbb{Z}}

\newcommand{\N}{\mathbb{N}}
\newcommand{\T}{\mathbb{T}}

\newcommand{\bS}{\mathbb{S}}

\newcommand{\sF}{\mathscr{F}}
\newcommand{\sL}{\mathscr{L}}
\newcommand{\sH}{\mathscr{H}}
\newcommand{\sK}{\mathscr{K}}

\newcommand{\cS}{\mathscr{S}}

\def\XXint#1#2#3{{\setbox0=\hbox{$#1{#2#3}{\int}$}
\vcenter{\hbox{$#2#3$}}\kern-.5\wd0}}

\newcommand{\op}{\operatorname} 

\newcommand{\Sp}{\op{Sp}}

\newcommand{\Tr}{\op{Tr}}

\newcommand{\ran}{\op{ran}}

\newcommand{\rk}{\op{rk}}
\newcommand{\car}{\mathbbm{1}}

\newcommand{\hell}{\hat{\ell}}

\newcommand{\scal}[2]{\ensuremath{\left\langle #1 | #2 \right\rangle}} 
\newcommand{\acou}[2]{\ensuremath{\left\langle #1 , #2 \right\rangle}} 
\newcommand{\bigscal}[2]{\ensuremath{\big\langle #1 | #2 \big\rangle}} 
\newcommand{\bigacou}[2]{\ensuremath{\big\langle #1 , #2 \big\rangle}} 

\newcommand{\nlambda}{\lambda^{-}}

\newcommand{\dom}{\op{dom}}

\def\XXint#1#2#3{{\setbox0=\hbox{$#1{#2#3}{\int}$ }
\vcenter{\hbox{$#2#3$ }}\kern-.6\wd0}}

\numberwithin{equation}{section}

\newtheorem{theorem}{Theorem}[section]
\newtheorem{proposition}[theorem]{Proposition}
\newtheorem{corollary}[theorem]{Corollary}

\newtheorem{lemma}[theorem]{Lemma}

\newtheorem{conjecture}[theorem]{Conjecture}

\theoremstyle{definition}

\theoremstyle{remark}
\newtheorem{example}[theorem]{Example}
\newtheorem{remark}[theorem]{Remark}
 
\newtheorem*{claim*}{Claim} 

\newcommand{\LLogL}{L\!\log\!L}

\title{Cwikel Estimates and Negative Eigenvalues of Schr\"odinger Operators on Noncommutative Tori}
\date{\today}

\author{Edward McDonald}
 \address{School of Mathematics and Statistics, University of New South Wales, Sydney, Australia}
 \email{edward.mcdonald@unsw.edu.au}

\author{Rapha\"el Ponge}
 \address{School of Mathematics, Sichuan University, Chengdu, China}
 \email{ponge.math@icloud.com}



\begin{document}
\begin{abstract}
 In this paper, we establish Cwikel-type estimates for noncommutative tori for any dimension~$n\geq 2$. We use them to derive Cwikel-Lieb-Rozenblum inequalities and Lieb-Thirring inequalities for negative eigenvalues of fractional Schr\"{o}dinger operators on noncommutative tori. The latter leads to a Sobolev inequality for noncommutative tori.  On the way we establish a ``borderline version" of the abstract Birman-Schwinger principle for the number of negative eigenvalues of relatively compact form perturbations of a non-negative semi-bounded operator with isolated 0-eigenvalue.
 \end{abstract}
\maketitle

\section{Introduction}       
 The celebrated estimates of Cwikel~\cite{Cw:AM77} are a landmark application of trace-ideal techniques in mathematical physics. They assert that if  $f\in L_p(\R^n)$ and $g$ is in the weak $L_p$-space $L_{p,\infty}(\R^n)$ with $p>2$, then the operator $f(X)g(-i\nabla)$ on $L_2(\R^n)$ is in the weak Schatten class $\sL_{p,\infty}$, and we have
\begin{equation}
 \big\| f(X)g(-i\nabla)\big\|_{\sL_{p,\infty}}\leq c_{np}\big\|f\big\|_{L_p} \big\|g\big\|_{L_{p,\infty}},
 \label{eq:Intro.CwikelAM} 
\end{equation}
where the constant $c_{np}$ depends only on $n$ and $p$. Here $f(X)$ is the multiplication by $f$ in position space and $g(-i\nabla)$ is the multiplication by $g$ in momentum space. We refer to Section~\ref{sec:Cwikel} for background on Schatten classes and weak Schatten classes. 
Cwikel's estimates were extended to $p\in (0,2)$ by Birman-Solomyak~\cite{BS:1977} and to $p=2$ by Birman-Karadzhov-Solomyak~\cite{BKS:1991} (who also dealt with more general Lorentz ideals $\sL_{p,q}$). Solomyak~\cite{So:PLMS95} substantially improved the understanding of the $p=2$ case in even dimension. 

Cwikel's estimates in the form stated above were conjectured by Simon~\cite{Si:TAMS76}. He pointed out that combining them with the Birman-Schwinger principle~\cite{Bi:AMST66, Sc:PNAS61} would give an upper-bound for the number of negative eigenvalues (i.e., the number of bound states) of Schr\"odinger operators $\Delta+V$ with $V\in L_{\frac{n}2}(\R^n)$, $n\geq 3$. Namely, there exists a constant $C_n>0$ depending only on $n$ such that
\begin{equation}\label{eq:Intro.CLR}
N^{-}(\Delta+V) \leq C_n\int |V_{-}(x)|^{\frac{n}2}dx,
\end{equation}
where $N^{-}(\Delta+V)$ is the number of negative eigenvalues of $\Delta+V$ counted with multiplicity and $V_{-}=\frac{1}{2}(|V|-V)$ is the negative part of $V$. 

The main motivation for the inequality~(\ref{eq:Intro.CLR}) is establishing a Weyl's law for Schr\"odinger operators $\hbar^2\Delta+V$ with a non-smooth potential $V$ under the semi-classical limit $h\rightarrow 0$ (see~\cite{BS:JFAA70, BS:AMST80, Si:TAMS76}). As it turns out, the inequality~(\ref{eq:Intro.CLR}) was already established by Rozenblum~\cite{Roz:1972, Roz:1976}. Another independent proof was provided by Lieb~\cite{Li:BAMS76, Li:1980} (see~\cite{Co:RMJM85, Fe:BAMS83, Fr:JST14, HKRV:arXiv18, LY:CMP83} for further alternative proofs). The equality~(\ref{eq:Intro.CLR}) is now known as the Cwikel-Lieb-Rozenblum (CLR) inequality. 

A related theme, initiated by Lieb-Thirring~\cite{LT:PRL75,LT:SMP76} concerns bounds for $\gamma$-moments of negative eigenvalues 
$\nlambda_j(\Delta+V)$, $j\geq 0$. Namely, if $\gamma>0$, then there is a constant 
$L_{\gamma,n}>0$ depending only $n$ and  $\gamma$ such that, for any real-valued potential $V\in L_{n+\gamma}(\R^n)$, we have
\begin{equation}\label{eq:Intro.LT}
    \sum_{j} |\nlambda_j(\Delta+V)|^\gamma \leq L_{\gamma,n} \int |V_-(x)|^{\frac{n}{2}+\gamma}\,dx.
\end{equation}

The above inequality~\eqref{eq:Intro.LT} is  known as the Lieb--Thirring (LT) inequality. It was first established for $\gamma=1$ and $n=3$ by Lieb-Thirring~\cite{LT:PRL75} as a key step toward proving the stability of matter in quantum mechanics. They subsequently extended the inequality in~\cite{LT:SMP76} 
to all $n\geq 2$ and $\gamma>0$ and for $n=1$ and $\gamma>1/2$. The case $n=1$ and $\gamma=1/2$ was settled by Weidl~\cite{We:CMP96}. 

 A fundamental fact about the LT inequality for $\gamma=1$ is its equivalence with Sobolev inequality (see~\cite{LT:PRL75}). More generally, the LT inequalities have widespread applications in mathematical physics (see, e.g.,~\cite{Li:RMP76, LS:CUP10}). They also play an important role in the study of Navier-Stokes equations (see, e.g., \cite{Li:CMP84, Te:Book97}). 


There are versions of the CLR inequality and the LT inequalities for fractional Schr\"odinger operators $\Delta^{n/2p}+V$ (see, e.g., \cite{Da:CMP83, LS:JAM97, Le:MRL97, Roz:1972, Roz:1976, RS:SPMJ98, So:PLMS95}).  Such operators naturally appear in the framework of fractional quantum mechanics~\cite{La:PL00}; for $p=n$ we get the ultra-relativistic Schr\"odinger operator $|\nabla|+V$ (see~\cite{Da:CMP83}). We refer to the recent survey~\cite{Fr:Survey20} and the book~\cite{LFW:Book} (and the numerous references therein) for a more thorough account on the CLR and LT inequalities and their various applications and generalizations.  


Recently, Cwikel estimates have been experiencing a revival of interest. A general abstract operator theoretic framework for Cwikel-type estimates has been emerging (see~\cite{Fr:JST14, HKRV:arXiv18, LeSZ:2020}). In a recent article~\cite{HKRV:arXiv18} Hundertmark \emph{et al}.\ used a refinement of Cwikel's estimates to get some of the best CLR bounds to date. Furthermore, applications of Cwikel-type estimates to Connes's noncommutative geometry program were found~\cite{GGBISV:CMP04, LeSVZ:IEOT18, LeSZ:JGP19, LeSZ:2020, LSZ:JFA20, MSX:CMP19, SZ:CMP18, SZ:arXiv20}. In particular, Cwikel-type estimates for noncommutative Euclidean spaces were established by Levitina-Sukochev-Zanin~\cite{LeSZ:2020}. 

The aim of this paper is to obtain general Cwikel-type estimates and establish CLR inequalities on noncommutative tori (a.k.a.\ quantum tori). Noncommutative tori  are arguably the most well known examples of noncommutative spaces. In particular, they naturally appear in the noncommutative geometry approach to the quantum Hall effect~\cite{BES:JMP94} and to topological insulators~\cite{BCR:RMP16, PS:Springer16}. In addition, noncommutative tori have been considered in the context of string theory (see, e.g., \cite{CDS:JHEP98, SW:JHEP99}). Noncommutative 2-tori naturally arise from actions of $\Z$ on the circle $\bS^1$ by irrational rotations. More generally, a noncommutative $n$-tori $\T^n_\theta$ is generated by unitaries $U_1, \ldots, U_n$ subject to the relations, 
\begin{equation*}
 U_lU_j = e^{2i\pi \theta_{jl}}U_jU_l, \qquad j,l=1,\ldots, n, 
\end{equation*}
where $\theta=(\theta_{jl})$ is a given real anti-symmetric matrix. We refer to Section~\ref{sec:NCtori} for more background on noncommutative tori. 

We establish Cwikel-type estimates on  $\T^n_\theta$ for operators $\lambda(x)g(-i\nabla)$, where $x$ is in $L_p(\T^n_\theta)$ with $p\geq 2$. Here $\nabla=(\partial_1,\ldots,\partial_n)$, where $\partial_1,\ldots,\partial_n$ are the canonical derivations of $\T^n_\theta$, and $\lambda$ is the extension to $L_p$-spaces of the left-regular representation of $L_\infty(\T^n_\theta)$ in $L_2(\T^n_\theta)$ (see Section~\ref{sec:boundedness}). More precisely, we show that if $x\in L_q(\T^n_\theta)$ and $g\in \ell_{p,\infty}(\Z^n)$, we have
\begin{equation}
  \|\lambda(x)g(-i\nabla)\|_{\sL_{p,\infty}} \leq c(p,q)\|x\|_{L_q}\|g\|_{\ell_{p,\infty}}, 
  \label{eq:Intro.CwikelNCT} 
\end{equation}
where $c(p,q)$ is some \emph{explicit} constant independent of $g$ and $x$. We establish those Cwikel-type estimates for $q=p>2$ (Theorem~\ref{Cwikel2+}, part (2)), 
$q=2<p$ (Theorem~\ref{Cwikel2-}, part (2)), and for $p=2<q$ (Theorem~\ref{L_2_infty_cwikel}). In each case, we obtain explicit bounds for the best constants in those inequalities. 

We also have $\sL_p$-estimates for $p<2$. Namely, if $x\in L_q(\T^n_\theta)$ and $g\in \ell_p(\Z^n)$, then 
\begin{equation*}
  \|\lambda(x)g(-i\nabla)\|_{\sL_{p}} \leq \|x\|_{L_2}\|g\|_{\ell_{p}}. 
\end{equation*}
The inequalities hold for $q=p\geq 2$ (Theorem~\ref{Cwikel2+}, part (1)) and for $q=2>p$ (Theorem~\ref{Cwikel2-}). 

Note that $\sL_{p,\infty}$-estimates with $p>2$ and $\sL_p$-estimate with $p\geq 2$ were established in~\cite{MSX:CMP19} with unspecificed constantsas a special case of general results of Levitina-Sukochev-Zanin~\cite{LeSZ:2020}. We further elaborate on the arguments of~\cite{LeSZ:2020} to get explicit constants for these estimates. 

The $\sL_{2,\infty}$-estimate is deduced from the estimates in the $p>2$ case and $p<2$ case by an interpolation argument. The proof of the estimates in the $p<2$ case essentially follows the approach of~\cite{LeSZ:2020} to the Cwikel-type estimates on NC Euclidean spaces. As it turns out, on noncommutative tori some simplifications occur, which allows us to get much sharper results. In fact, as far as Schatten and weak Schatten classes are concerned, the estimates of this papers are as sharp as in the (commutative) Euclidean space setting.  

The general Cwikel estimates~(\ref{eq:Intro.CwikelNCT}) specialize to estimates for operators of the form $\lambda(x)\Delta^{-n/2p}$, where $\Delta=\nabla^*\nabla$ is the (positive) Laplacian on $\T^n_\theta$ (Theorem~\ref{thm.specific-Cwikel}). Set
\begin{equation*}
 \nu_0(n):= \sup_{\lambda \geq 1} \lambda^{-\frac{n}{2}}\#\left\{k\in \Z^n\setminus 0; \ |k|\leq \sqrt{\lambda} \right\}. 
\end{equation*}
If $x\in L_q(\T^n_\theta)$, and either $p\neq 2$ and $q=\max(p,2)$, or if $p=2<q$, then
\begin{equation}
  \big\|\lambda(x) \Delta^{-\frac{n}{2p}}\big\|_{\sL_{p,\infty}} \leq c(p,q)\nu_0(n)^{\frac1{p}} \|x\|_{L_q}.
   \label{eq:Intro.specific-Cwikel}   
\end{equation}
For $p\neq 2$ the constant $c(p,q)$ is the best constant from the corresponding Cwikel-type inequality~(\ref{eq:Intro.CwikelNCT}). For $p=2$ the inequality follows from the case $p>2$ by using H\"older's inequality, and so we obtain a better constant. 

This allows us to get estimates for operators of the form $\Delta^{-n/4p}\lambda(x) \Delta^{-n/4p}$ (Theorem~\ref{thm:Specific-Cwikel.sandwiched}). More precisely, if $x\in L_q(\T^m_\theta)$, and either $p\neq 1$ and $q=\max(p,1)$, or if $p=1<q$, then 
\begin{equation}
  \big\|\Delta^{-\frac{n}{4p}}\lambda(x) \Delta^{-\frac{n}{4p}}\big\|_{\sL_{p,\infty}} \leq 2^{\frac1p}c(2p,2q)\nu_0(n)^{\frac1{p}}\|x\|_{L_q}. 
  \label{eq:Intro.specific-Cwikel-symmetrized} 
\end{equation}
This inequality holds without the extra $2^{1/p}$-factor if $x\geq 0$ (see Remark~\ref{rmk:Specific.positive-case}).

In the Euclidean space setting the CLR inequality~(\ref{eq:Intro.CLR}) can be deduced from the Cwikel estimates~(\ref{eq:Intro.CwikelAM}) by using the Birman-Schwinger principle~\cite{Bi:AMST66, Sc:PNAS61}. In its abstract form due to Birman-Solomyak~\cite{BS:1989} (see also Proposition~\ref{prop:BSP.abstractBSP} and Corollary~\ref{cor:CLR.BKP-qinf}) the Birman-Schwinger principle implies that if $H$ is a (semi-bounded) non-negative selfadjoint operator on Hilbert space and $V$ is a non-positive relatively form-compact perturbation such that $(H+1)^{-1/2}V(H+1)^{-1/2}\in \sL_{p,\infty}$, $p>0$, then 
 \begin{equation}
 N(H+V;\lambda) \leq  \big\|(H+\lambda)^{-\frac{1}2}V(H+\lambda)^{-\frac{1}{2}}\big\|_{\sL_{p,\infty}}^p\qquad \forall \lambda<0. 
 \label{eq:Intro.BSP}
\end{equation}

When $H$ is the Laplacian on $\R^n$, $n\geq 3$, the inequality~(\ref{eq:Intro.BSP}) continues to hold for $\lambda=0$, thereby providing an estimate for $N^{-}(H+V)$. In various examples, including the Laplacians on NC tori, the origin is in the discrete spectrum, and so the resolvent $(H-\lambda)^{-1}$ has a pole singularity at $\lambda=0$. This prevents us from letting $\lambda\rightarrow 0^{-}$ in~(\ref{eq:Intro.BSP}). To remedy we derive a ``borderline'' Birman-Schwinger principle (Theorem~\ref{thm:Borderline-BSP}). In particular, we show that if $0$ is in the discrete spectrum of $H$ and $V$ is a non-positive relatively form-compact perturbation such that $H^{-1/2}VH^{-1/2}\in \sL_{p,\infty}$, $p>0$, then, we have
\begin{equation}
  0\leq N^{-}(H+V)- N^{-}(\Pi_0 V \Pi_0) \leq  \big\|H^{-\frac{1}2}VH^{-\frac{1}{2}}\big\|_{\sL_{p,\infty}}^p. 
  \label{eq:Intro.BorderlineBSP} 
\end{equation}
Here $\Pi_0$ is the orthogonal projection onto the nullspace of $H$. This result seems to be new, at least at this level of generality. Its scope of validity goes beyond the scope of this paper. For instance, it also encompasses (fractional) Schr\"odinger operators on closed manifolds or compact manifolds with boundary under Neumann boundary conditions, as well as Schr\"odinger operators on hyperbolic manifolds with infinite volume. 

Combining the specific Cwikel estimates~(\ref{eq:Intro.specific-Cwikel-symmetrized}) and the borderline Birman-Schwinger principle~(\ref{eq:Intro.BorderlineBSP}) allows us to get CLR-type inequalities for fractional Schr\"odinger operators $\Delta^{n/2p}+\lambda(V)$ on NC tori (Theorem~\ref{thm:CLR.CLR-NCtori}). Namely, if $V=V^*\in L_q(\T^n_\theta)$ and, either $p\neq 1$ and $q=\max(p,1)$, or $p=1<q$, then 
\begin{equation}
 N^{-}\big(\Delta^{\frac{n}{2p}}+\lambda(V)\big)-1 \leq c(2p,2q)^{2p}\nu_0(n) \tau\big[|V_{-}|^q\big]^{\frac{p}{q}},
 \label{eq:Intro.CLRNCT}
\end{equation}
 where $V_{-}=\frac12(|V|-V)$ is the negative part of $V$ and $\tau$ is the standard normalized trace of $L_\infty(\T^n_\theta)$. This inequality is consistent with Lieb's version of the CLR inequality for closed manifolds~\cite{Li:BAMS76, Li:1980}. For $p=n/2$ we get a CLR inequality for Schr\"odinger operators $\Delta+\lambda(V)$ with $V=V^*\in L_{n/2}(\T^n_\theta)$ if $n\geq 3$ or with $V=V^*\in L_q(\T^n_\theta)$, $q>1$, if $n=2$.  The latter condition is consistent with the CLR inequality for Schr\"odinger operators on bounded regions of $\R^2$ by Birman-Solomyak~\cite{BS:JFAA70} (see also~\cite{BS:AMST80}). In particular, unlike in the Euclidean space setting we do get an inequality in dimension~2. 

Under the semiclassical limit $h\rightarrow 0^+$ the CLR inequality~(\ref{eq:Intro.CLRNCT}) implies that 
\begin{equation}
 N^{-}\big(h^{\frac{n}{p}}\Delta^{\frac{n}{2p}}+\lambda(V)\big) \leq   c(2p,2q)^{2p}\nu_0(n)h^{-n}\tau\big[|V_{-}|^q\big]^{\frac{p}{q}} +\op{O}(1). 
 \label{eq:Intro.semiclassical-CLR}  
\end{equation}
This leads us to conjecture that if $V=V^*\in L_q(\T^n_\theta)$ and, either $p\neq 1$ and $q=\max(p,1)$, or $p=1<q$, then we have the semi-classical Weyl's law, 
\begin{equation}
 N^{-}\big(h^{\frac{n}{p}}\Delta^{\frac{n}{2p}}+\lambda(V) \big)= \op{Vol}(\mathbb{B}^n)h^{-n}
            \tau\big[|V_{-}|^{p}\big] +\op{o}\big(h^{-n}\big). 
            \label{eq:Intro.Weyl} 
\end{equation}
In the same way as in the Euclidean space setting, the semiclassical CLR inequality~(\ref{eq:Intro.semiclassical-CLR}) allows us to reduce the proof of~(\ref{eq:Intro.Weyl}) for $L_q$-potentials to that for smooth potentials. We believe a semiclassical Weyl law for smooth potentials can be established by using semiclassical pseudodifferential calculus on NC tori. However, such a pseudodifferential calculus has yet to be set up. As this project falls out of the scope of this paper we leave~(\ref{eq:Intro.Weyl}) as a conjecture. 

As is well known (see, e.g., \cite{Si:AMS15a}), in dimension~$\geq 3$ the simplest way to derive the Lieb-Thirring inequalities~(\ref{eq:Intro.LT}) is to deduce them from the CLR inequality~(\ref{eq:Intro.CLR}) (although this approach does not lead to good bounds for the best LT constants $L_{\gamma,n}$). Lieb-Thirring inequalities on ordinary tori (and spheres) have been established by Ilyin~\cite{Il:JST12} in dimension 2 and by Laptev-Ilyin~\cite{IL:SM16} in dimension~$\leq 19$ (see also~\cite{IL:StPMJ20, ILZ:MN19, ILZ:JFA20}). We establish Lieb-Thirring  inequalities on NC tori by following the CLR-route

As in~\cite{Il:JST12}, we restrict ourselves to the zero mean-value subspace $\{u; \ \tau(u)=0\}$, i.e.,  the orthogonal complement of the 
nullspace $\ker \Delta=\C\cdot 1$. Denoting by $\dot{\Delta}$ and $\dot{\lambda}(V)$ the relevant restrictions of $\Delta$ and $\lambda(V)$, we have CLR inequalities for the operators $\dot{\Delta}^{\frac{n}{2p}}+\dot{\lambda}(V)$ for $V=V^*\in L_q(\T^n_\theta)$ with $p$ and $q$ as above (Theorem~\ref{thm:CLR.CLR-NCtori2}). Namely, 
\begin{equation*}
 N^{-}\big(\dot{\Delta}^{\frac{n}{2p}}+\dot{\lambda}(V)\big) \leq c(2p,2q)^{2p}\nu_0(n) \tau\big[|V_{-}|^q\big]^{\frac{p}{q}}. 
\end{equation*}
This leads to LT inequalities for the operators $\dot{\Delta}^{\frac{n}{2p}}+\dot{\lambda}(V)$ for $p>1$ (Theorem~\ref{thm:LT-inequality}). More precisely, if $\gamma>0$ and 
 $V=V^*\in L_{p+\gamma}(\T^n_\theta)$, then  
\begin{equation*}\label{eq:Intro.LT-NCTori}
    \sum_j \big|\nlambda_j(\Delta^{\frac{n}{2p}}+\dot{\lambda}(V))\big|^\gamma \leq L_{p,\gamma,n}\tau\big[|V_-|^{p+\gamma}\big], 
\end{equation*}
where the best constant $L_{p,\gamma,n}$ is such that
\begin{equation*}
L_{p,\gamma,n} \leq \gamma \frac{\Gamma(p+1)\Gamma(\gamma)}{\Gamma(p+\gamma+1)}c(2p,2p)^{2p}\nu_0(n). 
\end{equation*}

As with the original Lieb-Thiring inequalities~(\ref{eq:Intro.LT}) on $\R^n$, for $p=n/2$ and $\gamma=1$ with $n\geq 3$, the LT inequality~(\ref{eq:Intro.LT-NCTori}) is equivalent to a Sobolev inequality on NC tori (Theorem~\ref{thm:LT.Sobolev}). More precisely, there is a constant $K_n>0$, such that, for any orthonormal family $\{u_0,\ldots, u_N\}$ of zero-mean value elements of the Sobolev space $W_2^1(\T^n_\theta)$, we have 
\begin{equation}\label{eq:Intro.Sobolev}
   \sum_{\ell=0}^N \tau\big[ |\nabla u_\ell|^2\big] \geq K_n \tau\left[ \bigg(\sum_{\ell=1}^N |u_\ell |^2\bigg)^{1+n/2}\right]. 
\end{equation}
Moreover, the best constant $K_n$ is related to the best LT constant $L_n=L_{\frac{n}{2},1,n}$ by
\begin{equation*}
 K_n= \frac{n}{n+2} \bigg( \frac{n+2}{2}L_n\bigg)^{-\frac{n}{2}}. 
\end{equation*}


We refer to~\cite{MP:Part2} for further applications of the Cwikel estimates~(\ref{eq:Intro.specific-Cwikel})--(\ref{eq:Intro.specific-Cwikel-symmetrized}) to curved noncommutative tori, i.e., noncommutative tori equipped with a Riemannian metric. In this setting the role of the flat Laplacian is played by the corresponding Laplace-Beltrami operator. In particular, we get ``curved" analogues of the CLR inequalities~(\ref{eq:Intro.CLRNCT}) and obtain $L_2$-versions of the Connes' integration formulas of~\cite{MSZ:MA19, Po:JMP20}.
  
The paper is organized as follows. In Section~\ref{sec:NCtori}, we review the main background on noncommutative tori. Section~\ref{sec:boundedness} contains some technical preliminaries, including sufficient conditions for the operators appearing in Cwikel-type estimates to be bounded. In 
 Section~\ref{sec:Cwikel}, we present our main Cwikel-type estimates for noncommutative tori and show how to deduce the $p=2$ case from the $p<2$ case. The $p<2$ case is proved in Section~\ref{sec:Cwikel2-}. In Section~\ref{sec:specific-Cwikel} we specialize the main Cwikel estimates into the specific Cwikel estimates~(\ref{eq:Intro.specific-Cwikel})--(\ref{eq:Intro.specific-Cwikel-symmetrized}). 
 In Section~\ref{sec:Birman-Schwinger},  we establish the borderline Birman-Schwinger principle~(\ref{eq:Intro.BorderlineBSP}). In Section~\ref{CLR_section} we establish our CLR inequalities and LT for fractional Schr\"odinger operators on NC tori and use them to derive the Sobolev inequality~(\ref{eq:Intro.Sobolev}). 
In  Appendix~\ref{app:BSP},  for reader's convenience  we reproduce Birman-Solomyak's proof of the abstract Birman-Schwinger principle.

\subsection*{Notation} 
Throughout the paper we make the convention that $c_{abd}$ and $C_{abd}$  are positive constants which depend only on the parameters $a$, $b$, $d$, etc., and may change from line to line. They do not depend on the other variables that are floating around. 

\section{Noncommutative Tori} \label{sec:NCtori}
In this section, we review the main definitions and properties of noncommutative $n$-tori, $n\geq 2$. We refer to~\cite{Co:NCG, HLP:IJM19a, Ri:CM90}, and the references therein, for a more comprehensive account.
 
Throughout this paper, we let $\theta =(\theta_{jk})$ be a real anti-symmetric $n\times n$-matrix, and denote by $\theta_1, \ldots, \theta_n$ its column vectors.  We also let  $L_2(\T^n)$ be the Hilbert space of $L_2$-functions on the ordinary torus $\T^n=\R^n\slash (2\pi \Z)^n$ equipped with the  inner product, 
\begin{equation} \label{eq:NCtori.innerproduct-L2}
 \scal{\xi}{\eta}= (2\pi)^{-n} \int_{\T^n} \xi(x)\overline{\eta(x)}d x, \qquad \xi, \eta \in L_2(\T^n). 
\end{equation}
 For $j=1,\ldots, n$, let $U_j:L_2(\T^n)\rightarrow L_2(\T^n)$ be the unitary operator defined by 
 \begin{equation*}
 \left( U_j\xi\right)(x)= e^{ix_j} \xi\left( x+\pi \theta_j\right), \qquad \xi \in L_2(\T^n). 
\end{equation*}
 We then have the relations, 
 \begin{equation} \label{eq:NCtori.unitaries-relations}
 U_kU_j = e^{2i\pi \theta_{jk}} U_jU_k, \qquad j,k=1, \ldots, n. 
\end{equation}

The \emph{noncommutative torus} is the noncommutative space whose $C^*$-algebra  $C(\T^n_\theta)$ and von Neuman algebra $L_\infty(\T^n_\theta)$ are generated by the unitary operators $U_1, \ldots, U_n$.  For $\theta=0$ we obtain the $C^*$-algebra $C(\T^n)$ of continuous functions on the ordinary $n$-torus $\T^n$ and the von Neuman algebra $L_\infty(\T^n)$ of essentially bounded measurable functions on $\T^n$. Note that~(\ref{eq:NCtori.unitaries-relations}) implies that $C(\T^n_\theta)$ (resp., $L_\infty(\T^n_\theta)$) is the norm closure (resp., weak closure) in $\sL(L_2(\T^n))$ of the linear span of the unitary operators, 
 \begin{equation*}
 U^k:=U_1^{k_1} \cdots U_n^{k_n}, \qquad k=(k_1,\ldots, k_n)\in \Z^n. 
\end{equation*}

\subsection{GNS representation} Let $\tau:\sL(L_2(\T^n))\rightarrow \C$ be the state defined by the constant function $1$, i.e., 
 \begin{equation*}
 \tau (T)= \scal{T1}{1}=(2\pi)^{-n}\int_{\T^n} (T1)(x) d  x, \qquad T\in \sL\left(L_2(\T^n)\right).
\end{equation*}
This induces a continuous tracial state on the von Neuman algebra $L_\infty(\T^n_\theta)$ such that $\tau(1)=1$ and $\tau(U^k)=0$ for $k\neq 0$. 
 The GNS construction then allows us to associate with $\tau$ a $*$-representation of $L_\infty(\T^n_\theta)$ as follows. 
 
 Let $\scal{\cdot}{\cdot}$ be the sesquilinear form on $C(\T^n_\theta)$ defined by
\begin{equation}
 \scal{u}{v} = \tau\left( uv^* \right), \qquad u,v\in C(\T^n_\theta). 
 \label{eq:NCtori.cAtheta-innerproduct}
\end{equation}
Note that the family $\{ U^k; k \in \Z^n\}$ is orthonormal with respect to this sesquilinear form. We let $L_2(\T^n_\theta)$ be the Hilbert space arising from the completion of $C(\T^n_\theta)$ with respect to the pre-inner product~(\ref{eq:NCtori.cAtheta-innerproduct}). The action of $C(\T^n_\theta)$ on itself by left-multiplication uniquely extends to a $*$-representation of $L_\infty(\T^n_\theta)$ in $L_2(\T^n_\theta)$. When $\theta=0$ we recover the Hilbert space $L_2(\T^n)$ with the inner product~(\ref{eq:NCtori.innerproduct-L2}) and the representation of $L_\infty(\T^n)$ by bounded multipliers. In addition,  as $(U^k)_{k \in \Z^n}$ is an orthonormal basis of $L_2(\T^n_\theta)$, every $u\in L_2(\T^n_\theta)$ can be uniquely written as 
\begin{equation} \label{eq:NCtori.Fourier-series-u}
 u =\sum_{k \in \Z^n} u_k U^k, \qquad u_k:=\scal{u}{U^k}, 
\end{equation}
where the series converges in $L_2(\T^n_\theta)$. When $\theta =0$ we recover the Fourier series decomposition in  $L_2(\T^n)$. 

\subsection{The smooth algebra $C^\infty(\T^n_\theta)$} The natural action of $\R^n$ on $\T^n$ by translation gives rise to an action on $\sL(L_2(\T^n))$. This induces a $*$-action $(s,u)\rightarrow \alpha_s(u)$ on $C(\T^n_\theta)$ given by 
\begin{equation*}
\alpha_s(U^k)= e^{is\cdot k} U^k, \qquad  \text{for all $k\in \Z^n$ and $s\in \R^n$}. 
\end{equation*}
This action is strongly continuous, and so we obtain a $C^*$-dynamical system $(C(\T^n_\theta), \R^n, \alpha)$. We are especially interested in the subalgebra $C^\infty(\T^n_\theta)$ of smooth elements of this $C^*$-dynamical system, i.e., $u\in C(\T^n_\theta)$ such that $\alpha_s(u) \in C^\infty(\R^n; C(\T^n_\theta))$. 

The unitaries $U^k$, $k\in \Z^n$, are contained in $C^\infty(\T^n_\theta)$, and so $C^\infty(\T^n_\theta)$ is a dense subalgebra of $C(\T^n_\theta)$. Denote by $\cS(\Z^n)$ the space of rapid-decay sequences with complex entries. In terms of the Fourier series decomposition~(\ref{eq:NCtori.Fourier-series-u}) we have
\begin{equation*}
 C^\infty(\T^n_\theta)=\bigg\{ u=\sum_{k\in \Z^n} u_k U^k; (u_k)_{k\in \Z^n}\in  \cS(\Z^n)\bigg\}. 
\end{equation*}
When $\theta=0$ we recover the algebra $C^\infty(\T^n)$ of smooth functions on the ordinary torus $\T^n$ and the Fourier-series description of this algebra. 

For $j=1,\ldots, n$, let $\partial_j:C^\infty(\T^n_\theta)\rightarrow C^\infty(\T^n_\theta)$ be the  derivation defined by 
\begin{equation*}
 \partial_j(u) = \partial_{s_j} \alpha_s(u)|_{s=0}, \qquad u\in C^\infty(\T^n_\theta), 
\end{equation*}
When $\theta=0$ it agrees with the derivation $\partial_{x_j}$ on $C^\infty(\T^n)$. In general, we have
\begin{equation*}
 \partial_j(U_l) = \left\{ 
 \begin{array}{ll}
 iU_j & \text{if $l=j$},\\
 0 & \text{if $l\neq j$}. 
\end{array}\right.
\end{equation*}

\subsection{$L_p$-Spaces} 
The $L_p$-spaces of $\T^n_\theta$ are special instances of noncommutative $L_p$ spaces associated with a semi-finite faithful normal trace on a von Neumann algebra~\cite{Ku:TAMS58, Se:AM53} (see also~\cite{FK:PJM86, Ne:JFA74}). We refer to~\cite{FK:PJM86, Ku:TAMS58} for the main background on noncommutative $L_p$-spaces needed in this paper. 

By definition $L_\infty(\T^n_\theta)$ is a von Neuman algebra of bounded operators on $L_2(\T^n)$. Thus, a closed densely defined operator on $L_2(\T^n)$ is $L_\infty(\T^n_\theta)$-affiliated when it commutes with the commutant of $L_\infty(\T^n_\theta)$ in $\sL(L_2(\T^n))$. Furthermore, as $\tau$ is a finite faithful positive trace on $L_\infty(\T_\theta^n)$ every such operator is $\tau$-measurable in the sense of~\cite{FK:PJM86, Ne:JFA74}. Therefore, these operators form a $*$-algebra, where the sum and product of such operators are meant as the closures of their usual sum and product in the sense of unbounded operators (see~\cite{Ne:JFA74}). 

The space $L_1(\T^n_\theta)$ consists of all $L_\infty(\T^n_\theta)$-affiliated operators $x$ on $L_2(\T^n)$ such that
\begin{equation*}
 \tau\left(|x|\right):=\int_0^\infty \lambda d(\tau_*E_\lambda)<\infty, 
\end{equation*}
where $E_\lambda=\car_{[0,\lambda]}(|x|)$ is the spectral measure of $|x|$. We obtain a Banach space upon equipping $L_1(\T^n_\theta)$ with the norm, 
\begin{equation*}
 \|x\|_{L_1}:= \tau\big(|x|\big), \qquad x \in L_1(\T^n_\theta). 
\end{equation*}
We have a continuous inclusion with dense range of $L_\infty(\T^n_\theta)$ into $L_1(\T^n_\theta)$. The trace $\tau$ uniquely extends to a continuous linear functional on $L_1(\T^n_\theta)$ such that
\begin{equation*}
    |\tau(x)|\leq \tau(|x|) = \|x\|_1\qquad \forall x \in L_1(\T^n_\theta).
\end{equation*}

For $p>1$, the space $L_p(\T^n_\theta)$ consists of all $L_\infty(\T^n_\theta)$-affiliated operators $x$ on $L_2(\T^n)$ such that $|x|^p\in L_1(\T^n_\theta)$. This is a Banach space with respect to the norm, 
\begin{equation*}
 \|x\|_{L_p}:= \tau(|x|^p)^{1/p} =\bigg(\int_0^\infty \lambda^p d(\tau_*E_\lambda)\bigg)^{\frac1p},\qquad x \in L_p(\T^n_\theta), 
\end{equation*}
where as above $E_\lambda$, $\lambda \geq 0$, is the spectral measure of $|x|$. We also have a continuous inclusion with dense range of $L_\infty(\T^n_\theta)$ into $L_p(\T^n_\theta)$.  In particular, for $p=2$ the above definition is consistent with the previous definition of $L_2(\T^n_\theta)$, since in both cases we get the completion of $L_\infty(\T^n_\theta)$ with respect to the same norm. 

We have the following version of H\"older's inequality. 

\begin{proposition}[\cite{FK:PJM86, Ku:TAMS58}]\label{prop:Holder}  Suppose that $p^{-1}+q^{-1}=r^{-1}\leq 1$. If $x\in L_p(\T^n_\theta)$ and $y\in L_q(\T^n_\theta)$, then $xy\in L_r(\T^n_\theta)$ with norm inequality, 
\begin{equation}
 \|xy\|_{L_r} \leq \|x\|_{L_p} \|y\|_{L_q}. 
\label{eq:Holder-Lp} 
\end{equation}
\end{proposition}

This implies that, for every $p\geq 1$, the multiplication of $L_\infty(\T_\theta^n)$ uniquely extends to continuous bilinear maps, 
\begin{equation*}
 L_\infty(\T^n_\theta)\times L_p(\T^n_\theta) \longrightarrow L_p(\T^n_\theta), \qquad L_p(\T^n_\theta)\times L_\infty(\T^n_\theta) \longrightarrow L_p(\T^n_\theta). 
\end{equation*}
In particular, for $p=2$ we recover the GNS representation of $L_\infty(\T_\theta^2)$ associated with $\tau$.

If $x$ is an $L_\infty(\T^n_\theta)$-affiliated operator on $\sL(L_2(\T^n_\theta))$, its polar decomposition~\cite[Theorem VIII.32]{RS1:1980} takes the  form $x=u|x|$, where $u$ is a partial isometry in $L_\infty(\T^n_\theta)$. 
Moreover, the equality $x^*=u|x|u^*$ (see, e.g.,~\cite[Theorem 1.8.3]{BS:Book}) and H\"older's inequality show that
\begin{equation*}
 x\in L_{p}(\T^n_\theta) \Longrightarrow x^*\in L_p(\T^n_\theta)\ \text{and}\  \|x^*\|_{L_p}=\|x\|_{L_p}. 
\end{equation*}

We may also define $L_p$-spaces on $\T^n_\theta$ for $0<p<1$ as above. In this case we obtain a quasi-Banach spaces (see~\cite{Ci:BPASM83, FK:PJM86}). 

\subsection{Sobolev spaces} Given any $s\geq 0$, the Sobolev space $W_2^s(\T^n_\theta)$ is defined by 
\begin{equation}\label{sobolev_space_definition}
 W_2^s(\T^n_\theta):=\Big\{u =\sum_{k\in \Z^n} u_kU^k\in L_2(\T^n_\theta); \sum_{k\in \Z^n} (1+|k|^2)^s|u_k|^2<\infty\Big\}. 
\end{equation}
This is a Hilbert space with respect to the inner product and norm, 
\begin{equation*}
 \scal{u}{v}_s=\sum_{k\in \Z^n} (1+|k|^2)^{s}u_k\overline{v}_k, \qquad \|u\|_{W_2^s}=\bigg(\sum_{k\in \Z^n}  (1+|k|^2)^s|u_k|^2\bigg)^{\frac12}. 
\end{equation*}
Note that $W_2^0(\T^n_\theta)=L_2(\T^n_\theta)$. 

Equivalently, let $\Delta=-(\partial_1^2+\cdots + \partial_n^2)$ be the Laplacian on $\T^n_\theta$. This is a non-negative selfadjoint operator on $L_2^(\T^n_\theta)$ with domain $W_2^2(\T^n_\theta)$. We have 
\begin{equation*}
 \Delta\big(U^k)=|k|^2U^k, \qquad k\in \Z^n. 
\end{equation*}
In particular, $\Delta$ is isospectral to the Laplacian on the ordinary torus $\T^n$. Set $\Lambda=(1+\Delta)^{\frac12}$. Given any $s\geq 0$, we have 
\begin{equation*}
 W_2^s(\T^n_\theta):=\Big\{u\in L_2(\T^n_\theta);\ \Lambda^su\in L_2(\T^n_\theta)\Big\}, \qquad \|u\|_{W_2^s}=\|\Lambda^su\|_{W_2^0}. 
\end{equation*}
Note also (see~\cite{Sp:Padova92, XXY:MAMS18}) that, given an integer $p\geq 0$, we have
\begin{equation*}
 W_2^p(\T^n_\theta)= \Big\{u\in L_2(\T^n_\theta); \ \delta^\alpha u\in L_2(\T^n_\theta) \  \ \forall \alpha \in \N_0^n, \, |\alpha|\leq p\big\}. 
\end{equation*}

We mention the following versions of Sobolev's embedding theorems. 

 \begin{proposition}[see~{\cite[Theorem 6.6]{XXY:MAMS18}}] \label{prop:Sobolev-embeddingLp} 
 Let $p\in [2,\infty)$. For every $s\geq n(1/2-p^{-1})$, we have a continuous embedding $W_2^s(\T^n_\theta) \subset L_p(\T^n_\theta)$. This embedding is compact when the inequality is strict. 
 \end{proposition}
 
\begin{proposition} \label{prop:Sobolev-embeddingC0}
 For any $s>n/2$, we have a compact embedding $W^s_2(\T^n_\theta) \subset C(\T^n_\theta)$. 
\end{proposition}
\begin{remark}
 The continuity of the inclusion of the $W^s_2(\T^n_\theta) \subset C(\T^n_\theta)$ for $s>n/2$ is established in \cite[p.~67]{HLP:IJM19b} (see also Remark~\ref{rmk:embedding-Ws-hell-C} below). We obtain compactness by factorizing it through any inclusion $W^s_2(\T^n_\theta) \subset W^{s'}_2(\T^n_\theta)$ with $s>s'>n/2$. 
\end{remark}

\section{$L_p$-Action and Boundedness of $\lambda(x)g(-i\nabla)$ and  $(1+\Delta)^{-\frac{p}{4n}}\lambda(x)(1+\Delta)^{-\frac{p}{4n}}$}\label{sec:boundedness}
In what follows, we denote by $\lambda$ the left-regular representation of $L_\infty(\T^n_\theta)$ on $L_2(\T^n_\theta)$. 
In this section, we shall explain how to extend it to $L_p(\T^n_\theta)$, $p\geq 1$. We shall then give sufficient conditions for the boundedness of operators of the forms $\lambda(x)g(-i\nabla)$ and $(1+\Delta)^{-\frac{p}{4n}}\lambda(x)(1+\Delta)^{-\frac{p}{4n}}$.   

\subsection{Left-multipliers $\lambda(x)$} First, we observe that H\"older's inequality (Proposition~\ref{prop:Holder}) yields the following extension result for the left-regular representation $\lambda: L_\infty(\T^n_\theta)\rightarrow \sL(L_2(\T^n))$. 

\begin{proposition}[\cite{Ku:TAMS58}]\label{prop:left-reg-Lp} 
Suppose that $p^{-1}+q^{-1}=r^{-1}\geq 1$.  Then the left-regular representation uniquely extends to a continuous linear map, 
\begin{equation*}
 \lambda: L_p(\T^n_\theta) \longrightarrow \sL\big(L_q(\T^n_\theta),L_r(\T^n_\theta)\big).
\end{equation*}
In particular, if $p\geq 2$ and $p^{-1}+q^{-1}=2$, then we get a continuous linear map, 
\begin{equation*}
 \lambda: L_p(\T^n_\theta) \longrightarrow \sL\big(L_q(\T^n_\theta),L_2(\T^n_\theta)\big).
\end{equation*}
\end{proposition}
 
\begin{remark}
 The above linear maps are isometries (see~\cite{Ku:TAMS58}). 
\end{remark}
   
If $x\in L_p(\T^n_\theta)$ with $p\geq 2$, the above corollary asserts that $\lambda(x)$ a continuous linear operator from $L_q(\T_\theta^n)$ to $L_2(\T^n_\theta)$ with $p^{-1}+q^{-1}=1/2$. In particular, we may regard $\lambda(x)$ as an unbounded operator on $L_2(\T^n_\theta)$ with domain $L_q(\T_\theta^n)$. 

Given any $q\geq 1$, we denote by $L_q(\T^n_\theta)^*$ the anti-linear dual of $L_q(\T^n_\theta)$, i.e., the space of continuous anti-linear forms on 
$L_q(\T^n_\theta)$. We observe that the left-regular regular representation of $L_\infty(\T^n_\theta)$ can be regarded as a continuous linear map $\lambda:L_\infty(\T^n_\theta)\rightarrow \sL(L_2(\T^n_\theta), L_2(\T^n_\theta)^*)$ such that
\begin{equation*}
 \acou{\lambda(x)u}{v}=\scal{\lambda(x)u}{v}=\tau\big[v^*xu\big], \qquad x\in L_\infty(\T^n_\theta), \quad u,v\in L_2(\T^n_\theta),
\end{equation*}
where $\acou{\cdot}{\cdot}:L_2(\T^n_\theta)^*\times L_2(\T^n_\theta)\rightarrow \C$ is the duality pairing. 

\begin{proposition}\label{prop:left-reg-LpLqLq*}
Suppose that $p^{-1}+2q^{-1}=1$.  Then the left-regular representation uniquely extends to a continuous linear map $\lambda: L_p(\T^n_\theta) \rightarrow \sL(L_q(\T^n_\theta), L_q(\T^n_\theta)^*)$ such that
\begin{equation}\label{eq:left-reg-LpLqLq*}
 \acou{\lambda(x)u}{v}=\tau\big[v^*xu\big] \qquad \forall x\in L_p(\T^n_\theta)\quad  \forall u,v\in L_q(\T^n_\theta). 
\end{equation}
\end{proposition}
\begin{proof}
 Let $x\in L_p(\T^n_\theta)$ and $u,v\in L_q(\T^n_\theta)$. By H\"older's inequality $xu\in L_r(\T^n_\theta)$ with $r^{-1}=p^{-1}+q^{-1}=1-q^{-1}$, and so $v^*(xu)\in L_1(\T^n_\theta)$. Moreover, we have 
\begin{equation*}
 \left|\tau(v^*xu)\right|\leq \|v^*(xu)\|_{L_1} \leq \|v^*\|_{L_q}\|xu\|_{L_r} \leq \|x\|_{L_p} \|u\|_{L_q} \|v\|_{L_q}. 
\end{equation*}
This gives the result. 
\end{proof}

In particular, if $x\in L_p(\T^n_\theta)$ with $1\leq p<2$, then the above proposition shows that $\lambda(x)$ makes sense as a bounded operator from $L_q(\T^n_\theta)$ to $L_q(\T^n_\theta)^*$, with $p^{-1}+2q^{-1}=1$, i.e., $q^{-1}=\frac{1}{2}(1-p^{-1})$. 

\subsection{The operators $(1+\Delta)^{-s/2}\lambda(x) (1+\Delta)^{-s/2}$} 
Given any $s>0$, we denote by $W_2^{-s}(\T^2_\theta)$ the anti-linear dual of the Sobolev space $W_2^s(\T^n_\theta)$. Set $\Lambda=(1+\Delta)^{1/2}$. As mentioned above $\Lambda^s: W^s_2(\T^n_\theta)\rightarrow L_2(\T^n_\theta)$ is an isometric isomorphism. By duality we get a continuous isomorphism $\Lambda^s:L_2(\T^n_\theta)\rightarrow W_2^{-s}(\T^n_\theta)$ such that
\begin{equation*}
 \acou{\Lambda^su}{v}=\scal{u}{\Lambda^sv}, \qquad u\in L_2(\T^n_\theta), \quad v\in W_2^{s}(\T^n_\theta),
\end{equation*}
where $\acou{\cdot}{\cdot}: W_2^{-s}(\T^n_\theta)\times W_2^{s}(\T^n_\theta)\rightarrow \C$ is the duality pairing. Its inverse $\Lambda^{-s}: W_2^{-s}(\T^n_\theta)\rightarrow L_2(\T^n_\theta)$ is given by
\begin{equation}
 \scal{\Lambda^{-s}u}{v}=\acou{u}{\Lambda^{-s}v},  \qquad u\in W_2^s(\T^n_\theta), \quad v\in L_2(\T^n_\theta). 
 \label{eq:Boundedness.Lambdas-dual}
\end{equation}

\begin{lemma}\label{lem:Boundedness.sandwich1}
 Let $x\in L_p(\T^n_\theta)$, $p\geq 1$, and assume that, either $s>n/2p$, or $s=n/2p$ and $p>1$. Then $\lambda(x)$ uniquely extends to a bounded operator $\lambda(x):W_2^s(\T^n_\theta)\rightarrow W_2^{-s}(\T^n_\theta)$. 
\end{lemma}
\begin{proof}
 Suppose that $p^{-1}+2q^{-1}=1$, i.e., $q^{-1}=\frac{1}{2}(1-p^{-1})$. We know  by Proposition~\ref{prop:left-reg-LpLqLq*} that $\lambda(x)$ is a bounded operator from $L_q(\T^n_\theta)$ to $L_q(\T^n_\theta)^*$. 
 
 Suppose that $p>1$ and $s\geq n/2p$. Then $q\in [2,\infty)$, and so by Proposition~\ref{prop:Sobolev-embeddingLp} we have a continuous embedding of $W_2^s(\T^n_\theta)$ into $L_q(\T^n_\theta)$ since $s\geq n/{2p}=\frac{n}{2}(1-q^{-1})$. By duality we get a continuous embedding of $L_q(\T^n_\theta)^*$ into $W_2^{-s}(\T^n_\theta)$. It then follows that $\lambda(x)$ induces a bonded operator $\lambda(x):W_2^s(\T^n_\theta)\rightarrow W_2^{-s}(\T^n_\theta)$.

Assume now that $p=1$ and $s>n/2$. In this case $q=\infty$, and so $\lambda(x)$ is a bounded operator from $L_\infty(\T^n_\theta)$. As $s>n/2$, by Proposition~\ref{prop:Sobolev-embeddingC0} we have a continuous embedding of $W_2^s(\T^n_\theta)$ into $C(\T^n_\theta)$, and hence we get a continuous embedding into $L_\infty(\T^n_\theta)$. By duality we get a continuous embedding of $L_\infty(\T^n_\theta)^*$ into $W_2^{-s}(\T^n_\theta)$. Thus, as above $p>1$, $\lambda(x)$ induces a bounded operator $\lambda(x):W_2^s(\T^n_\theta)\rightarrow W_2^{-s}(\T^n_\theta)$. The proof is complete. 
\end{proof}

Combining the above lemma with the boundedness of the operators $\Lambda^{-s}: L_2(\T^n_\theta)\rightarrow W_2^{s}(\T^n_\theta)$ and $\Lambda^{-s}: W_2^{-s}(\T^n_\theta)\rightarrow L_2(\T^n_\theta)$ we arrive at the following result. 

\begin{proposition}\label{prop:Boundedness.sandwich2}
  Let $x\in L_p(\T^n_\theta)$, $p\geq 1$, and assume that either $s>n/2p$, or $s=n/2p$ and $p>1$. Then the composition $\Lambda^{-s}\lambda(x)\Lambda^{-s}$ makes sense as a bounded operator on $L_2(\T^n_\theta)$. 
\end{proposition}

\begin{remark}
 The above result holds \emph{verbatim} if we replace $\Lambda$ by $\sqrt{\Delta}$. 
\end{remark}

\subsection{The operators $\lambda(x)g(-i\nabla)$} 
Recall that, given any $p \in (0,\infty)$, the quasi-Banach space $\ell_p(\mathbb Z^n)$ consists of $p$-summable sequences $a = (a_k)_{k\in \mathbb Z^n} \subset \mathbb C$ with quasi-norm, 
\begin{equation*}
    \|a\|_{\ell_p} := \bigg(\sum_{k\in \mathbb Z^n} |a_k|^p\bigg)^{1/p}, \qquad a=(a_k)\in \ell_p(\Z^n). 
\end{equation*}
For $p\geq 1$ this is a norm, and so in this case $\ell_p(\mathbb Z^n)$ is a Banach space. In addition, we denote by $\ell_{\infty}(\mathbb Z^n)$ the Banach space of bounded sequences with norm, 
\begin{equation*}
    \|a\|_{\ell_\infty} := \sup_{k\in \mathbb Z^n} |a_k|, \qquad a=(a_k)\in \ell_\infty(\Z^n). 
\end{equation*}

For $p\in (0,\infty)$, the weak $\ell_p$-space $\ell_{p,\infty}(\mathbb Z^n)$ is defined as follows. Given any sequence $a=(a_k)_{k\in \Z^n}\in \ell_\infty(\Z^n)$, let $\mu(a)=(\mu_j(a))_{j\geq 0}$ be its symmetric decreasing re-arrangement, i.e., 
\begin{equation*}
 \mu_j(a):= \sup_{k^0, \ldots, k^j\in \Z^n}\min \{|a_{k^0}|,|a_{k^1}|,\ldots,|a_{k^j}|\}. 
\end{equation*}
In other words, $\mu(a)=(\mu_j(a))_{j\geq 0}$ is the non-increasing rearrangement of the sequence $(|a_k|)_{k\in\mathbb Z^n}$. The space $\ell_{p,\infty}(\mathbb Z^n)$ then consists of sequences $a=(a_k)_{k\in \Z^n}\in \ell_\infty(\Z^n)$ such that 
\begin{equation*}
 \mu_j(a) =\op{O}\big(j^{-\frac1{p}}\big)\qquad \text{as $j\rightarrow \infty$}.
\end{equation*}
We equip it with the quasi-norm, 
\begin{equation*}
 \|a\|_{\ell_{p,\infty}}: = \sup_{j\geq 0}\, (j+1)^{1/p} \mu_j(a), \qquad a \in \ell_{p,\infty}(\Z^n). 
\end{equation*}
With this quasi-norm $\ell_{p,\infty}(\Z^n)$ is a quasi-Banach space. In fact, for $p>1$ the above quasi-norm is equivalent to the norm,
\begin{equation*}
  \|a\|_{\ell_{p,\infty}}': = \sup_{N\geq 1} N^{-1+\frac1{p}}\sum_{j<N} \mu_j(a), \qquad a \in \ell_{p,\infty}(\Z^n). 
\end{equation*}
Therefore, for $p>1$ we actually obtain a Banach space. 

 The canonical derivations $\partial_1, \ldots, \partial_n$ pairwise commute with each other. Their joint spectrum is $i\Z^n$. 
 Given any $g \in \ell_\infty(\Itgr^n)$, the operator $g(-i\nabla)$ is given by
    \begin{equation*}
        g(-i\nabla)U^k = g(k)U^k,\quad k \in \Itgr^n.
    \end{equation*}
This is a bounded operator on $L_2(\T^n_\theta)$ with norm equal to $\|g\|_{\ell_\infty}$. In fact, if we also denote by $g$ the operator of multiplication by $g$ on $\ell_2(\Z^n)$, then $\mu(g(-i\nabla))=\mu(g)$. In particular, this implies that 
\begin{itemize}
 \item If $g\in \ell_p$, then $g(-i\nabla)\in \sL_p$ and $\|g(-i\nabla)\|_{\sL_p}=\|g\|_{\ell_p}$. 
 
 \item If $g\in \ell_{p,\infty}$, then $g(-i\nabla)\in \sL_{p,\infty}$ and $\|g(-i\nabla)\|_{\sL_{p,\infty}}=\|g\|_{\ell_{p,\infty}}$. 
\end{itemize}

Let us now look at the mapping properties of the operators $g(-i\nabla)$. To this end, given any $p\in [1,2]$, it is convenient to introduce the space $\hell_p(\T^n_\theta)$ that consists of all $x= \sum \widehat{x}_kU^k$ in $L_2(\T^n_\theta)$ such that $\sum |\widehat{x}_k|^p<\infty$. We equip it with the norm,  
         \begin{equation*}
            \|x\|_{\hell_p}: =\bigg(\sum_{k\in \Z^n} |\widehat{x}_k|^p\bigg)^{\frac1p}, \qquad x\in \hell_p(\T^n_\theta)
        \end{equation*}
 In other words, $\hell_p(\T^n_\theta)$ is the inverse image of $\ell_p(\Z^n)$ under the Fourier transform $x\rightarrow (\hat{x}_k)$. Thus, under the Fourier transform the spaces $\hell_p(\T^n_\theta)$ and $\ell_p(\Z^n)$ are isometrically isomorphic. In particular, $\hell_p(\T^n_\theta)$ is a Banach space. 
 
 For $p=2$ the spaces $\hell_p(\T^n_\theta)$ and $L_2(\T^n)$ and their norms agree as well. Moreover, we have a continuous inclusion 
$\hell_1(\T^n_\theta) \subset C(\T_\theta^n)$, since, for every $x=\sum \hat{x}_kU^k$ in $\hell_1(\T^n_\theta)$, the Fourier series $\sum \hat{x}_kU^k$ converges normally in $C(\T_\theta^n)$, and we have
\begin{equation*}
 \|x\|\leq \sum_{k\in \Z^n}|\hat{x}_k|\|U^k\| =  \sum_{k\in \Z^n}|\hat{x}_k| = \|x\|_{\hell_1}. 
\end{equation*}

\begin{lemma}\label{lem:Holder-gnabla}
 Suppose that $p^{-1}+q^{-1}=r^{-1}$ and $g\in \ell_p(\Z^n)$ with $p\geq 2$ and $1\leq q,r\leq 2.$ Then $g(-i\nabla)$ induces a continuous linear operator 
 $g(-i\nabla): \hell_q(\T^n_\theta)\rightarrow \hell_{r}(\T^n_\theta)$ with norm inequality, 
\begin{equation*}
 \big\| g(-i\nabla)x\big\|_{\hell_r} \leq \|g\|_{\ell_p}\|x\|_{\hell_q} \qquad \forall x\in  \hell_q(\T^n_\theta). 
\end{equation*}
  \end{lemma}
\begin{proof}
 Let $x \in  \hell_q(\T^n_\theta)$. We have $g(-i\nabla)x=\sum g(k)\hat{x}(k)U^k$. By assumption $(g(k))_{k\in \Z^n}\in \ell_p(\Z^n)$ and $\hat{x}:=(\hat{x}(k))_{k\in \Z^n}\in \ell_q(\Z^n)$, and so by H\"older's inequality $(g(k)\hat{x}(k))_{k\in \Z^n}\in \ell_r(\Z^n)$, since $p^{-1}+q^{-1}=r^{-1}$. This means that $g(-i\nabla)x\in  \hell_r(\T^n_\theta)$. Moreover, we have the inequalities, 
\begin{equation*}
 \big\| g(-i\nabla)x\big\|_{\hell_r} = \|(g(k)\hat{x}(k))\|_{\ell_r}\leq \|g\|_{\ell_p}\|\hat{x}\|_{\ell_q}=\|g\|_{\ell_p}\|x\|_{\hell_q}. 
\end{equation*}
This proves the result. 
\end{proof}

The following two lemmas give the relationships of the $\hell_p$-spaces with the $L_p$-spaces and Sobolev spaces. 

    \begin{lemma}[Hausdorff-Young Inequality]\label{HausdorffYoung}
        Suppose that $q\geq 2$ and $q^{-1}+r^{-1}=1$. Then we have a continuous inclusion $\hell_r(\T^n_\theta)\subseteq L_{q}(\T^n_\theta)$ 
        with norm inequality,
       \begin{equation}
               \|x\|_{L_{q}} \leq \|x\|_{\hell_r} \qquad \forall x\in \hell_r(\T^n_\theta).
               \label{eq:HausdorffYoung} 
          \end{equation}
    \end{lemma}
    \begin{proof}
The proof uses complex interpolation theory essentially in the same way as in the proof of the classical Hausdorff-Young inequality (see, e.g., \cite{BS:PAM88}). For background on interpolation theory we refer to the short survey of Connes~\cite[Appendix IV.B]{Co:NCG} and the references therein. The main reference for complex interpolation theory there is the article of Calder\'on~\cite{Ca:SM64} (see also~\cite{BS:PAM88, KPS:AMS82}).  
  
As mentioned above $L_2(\T^n_\theta)$ and $\hell_2(\T_\theta^n)$ agree as Banach spaces and the inclusion of $\ell_1(\T^n_\theta)$ into $C(\T^n_\theta)$ gives rise to a contraction into $L_\infty(\T^n_\theta)$. In particular, we have the inequalities, 
\begin{gather}
 \|x\|_{L_2}=\|x\|_{\hell_2} \qquad \forall x\in \hell_2(\T^n_\theta) \label{eq:hell.Hausdorff-Young-L2},\\
 \|x\|_{L_\infty}\leq \|x\|_{\hell_1} \qquad \forall x\in \hell_1(\T^n_\theta). \label{eq:hell.Hausdorff-Young-Linfty}
\end{gather}

Suppose that $q\in (2,\infty)$. Then $L_{q}(\T_\theta^n)$ is a complex interpolation space for the pair of Banach spaces $(L_2(\T^n_\theta), L_\infty(\T^n_\theta))$. Namely, in the notation of~\cite[Appendix IV.B]{Co:NCG} we have $L_q(\T_\theta^n)= [L_2(\T^n_\theta), L_\infty(\T^n_\theta)]_\theta$, where $\theta$ is such that $q=(1-\theta)\frac{1}{2}$, i.e., $\theta=1-2q^{-1}$ (see~\cite[Section~4]{DDP:IEOT92} and~\cite{Ku:TAMS58}). Note also that, as $r\in (1,2)$, the space $\ell_r(\T^n_\theta)$ is an exact interpolation space for the pair $(\ell_2(\T^n_\theta),\ell_1(\T^n_\theta))$. Namely, $\ell_r(\T^n_\theta)=[\ell_2(\T^n_\theta),\ell_1(\T^n_\theta)]_\theta$, since we have
\[(1-\theta)\frac{1}{2}+\theta=\frac12(1+\theta)=\frac12(2-2q^{-1})=1-q^{-1}=r^{-1}.\] 
As the Fourier transform induces isometric isomorphisms between the spaces $\ell_p(\Z^n)$ and $\hell_p(\T^n_\theta)$, we see that 
$\hell_r(\T^n_\theta)=[\hell_2(\T^n_\theta),\hell_1(\T^n_\theta)]_\theta$. Combining all this with the inequalities~(\ref{eq:hell.Hausdorff-Young-L2})--(\ref{eq:hell.Hausdorff-Young-Linfty}) and using complex interpolation theory (see~\cite[Theorem IV.B.1]{Co:NCG}) then shows we have a continuous inclusion $\hell_r(\T^n_\theta)\subseteq L_{q}(\T^n_\theta)$ with the norm inequality~(\ref{eq:HausdorffYoung}). The proof is complete. 
\end{proof}

\begin{lemma}\label{lem:embedding-Ws-hell}
 Suppose that $1\leq p< 2$ and $s>\frac{n}{2}(2p^{-1}-1)$. Then we have a continuous inclusion $W_2^s(\T^n_\theta)\subset \hell_p(\T^n_\theta)$. 
\end{lemma}
\begin{proof}
 Let $u\in \sum u_k U^k$ be in $W_2^s(\T^n_\theta)$. By H\"older's inequality we have
\begin{align}
\sum_{k\in \Z^n} |u_k|^p & = \sum_{k\in \Z^n} \left(1+|k|^2\right)^{-\frac{ps}{2}} \big[ \left(1+|k|^2\right)^{s}|u|_k^2\big]^{\frac{p}{2}}  \nonumber \\ 
& \leq  \bigg[\sum_{k\in \Z^n} \left(1+|k|^2\right)^{-\frac{rps}{2}} \bigg]^{\frac1r}\bigg[ 
 \sum_{k\in \Z^n} \left(1+|k|^2\right)^{s}|u|_k^2\bigg]^{\frac{p}{2}}
 \label{eq:W2shellp.Holder-ineq} \\
& \leq  \bigg[\sum_{k\in \Z^n} \left(1+|k|^2\right)^{- \frac{1}{2}psr} \bigg]^{\frac1r}\|u\|_{W_2^s}^p. \nonumber
\end{align}
Here $r$ is such that $r^{-1}+(2p^{-1})^{-1}=1$, i.e., $r^{-1}=1-\frac1{2}p$. By assumption $s>\frac{n}{2}(2p^{-1}-1)=n(1-\frac{1}{2}p)p^{-1}=n(rp)^{-1}$, and hence $psr>n$ and $\sum (1+|k|^2)^{- \frac{1}{2}psr}<\infty$. Combining this with~(\ref{eq:W2shellp.Holder-ineq}) shows that $\sum_{k\in \Z^n} |u_k|^p<\infty$, i.e., $u\in \hell_p(\Z^n)$, and there is a constant $C_{nps}>0$ independent of $u$ such that
\begin{equation*}
 \|u\|_{\hell_p}= \bigg[\sum_{k\in \Z^n} |u_k|^p\bigg]^{\frac1p}\leq C_{nps} \|u\|_{W_2^s}. 
\end{equation*}
This proves the result. 
\end{proof}

\begin{remark}\label{rmk:embedding-Ws-hell-C}
 For $p=1$ and $s>n/2$, we get a continuous inclusion of $W_2^s(\T^n_\theta)$ into $\hell_1(\T^n_\theta)$. Combining it with the continuity of the inclusion of  $\hell_1(\T^n_\theta)$ into $C(\T^n_\theta)$ mentioned above, we get a continuous inclusion $W_2^s(\T^n_\theta)$ into $C(\T^n_\theta)$. 
\end{remark}

We are now in a position to prove the following result. 
\begin{proposition}\label{prop:mapping-gnabla-Lq}
 Suppose that $p^{-1}+q^{-1}=\frac{1}{2}$. 
\begin{enumerate}
 \item If $g\in \ell_p(\Z^n)$, then $g(-i\nabla)$ maps continuously $L_2(\T^n_\theta)$ to $L_q(\T^n)$. 
 
 \item  If $g\in \ell_{p,\infty}(\Z^n)$, then $g(-i\nabla)$ maps continuously $W^s_2(\T^n_\theta)$ to $L_q(\T^n)$ for every $s>0$. 
\end{enumerate}
\end{proposition}
\begin{proof}
Let $g\in \ell_p(\Z^n)$. Lemma~\ref{lem:Holder-gnabla} ensures us that $g(-i\nabla)$ maps continuously $L_2(\T^n_\theta)$ to $\hell_r(\T^n_\theta)$, where $r^{-1}=\frac{1}2+p^{-1}$. Note that $1-r^{-1}=\frac{1}2-p^{-1}=q^{-1}$, and so by Lemma~\ref{HausdorffYoung} we have a continuous inclusion of $\hell_r(\T^n_\theta)$ into $L_q(\T^n_\theta)$. It then follows that $g(-i\nabla)$ maps continuously $L_2(\T^n_\theta)$ to $L_q(\T_\theta^n)$. This proves the first part. 

To prove the 2nd part let $g\in \ell_{p,\infty}(\Z^n)$ and $s>0$. In addition, let $t\in[1,2)$ be such that $s>\frac{n}{2}(2t^{-1}-1)$. Then by Lemma~\ref{lem:embedding-Ws-hell} we have continuous embedding of $W_2^s(\T^n_\theta)$ into $\hell_t(\T^n_\theta)$. Set $p_1=[p^{-1}-(t^{-1}-\frac{1}2)]^{-1}$, i.e., $p_1^{-1}=p^{-1}-(t^{-1}-\frac{1}2)$. Note that $p_1>p$ since $t<2$. We also observe that $p_1^{-1}+t^{-1}=p^{-1}+\frac{1}{2}=1-(\frac{1}{2}-p^{-1})=1-q^{-1}$. The fact $p_1>p$ implies that $g\in \ell_{p_1}(\Z^n)$, and so  Lemma~\ref{lem:Holder-gnabla} ensures us that $g(-i\nabla)$ maps continuously $\hell_t(\T^n_\theta)$ into $\hell_{r}(\T^n_\theta)$ with $r^{-1}=p_1^{-1}+t^{-1}=1-q^{-1}$. Furthermore, by Lemma~\ref{HausdorffYoung} we have a continuous inclusion of 
$\hell_{r}(\T^n_\theta)$ into $L_{q_1}(\T^n_\theta)$. It then follows that $g(-i\nabla)$ maps  continuously $W_2^s(\T^n_\theta)$ to $L_q(\T^n_\theta)$. This proves the 2nd part and completes the proof. 
\end{proof}

By combining Proposition~\ref{prop:left-reg-Lp} and Proposition~\ref{prop:mapping-gnabla-Lq} we arrive at the following result. 

\begin{proposition}\label{eq:boundedness.Lp-ellp}
 Let $p\in [2,\infty)$. 
\begin{enumerate}
\item If $x\in L_p(\T^n_\theta)$ and $g\in \ell_{p}(\Z^n)$, then the operator $\lambda(x)g(-i\nabla)$ is bounded on $L_2(\T^n_\theta)$. 

\item If $x\in L_p(\T^n_\theta)$ and $g\in \ell_{p,\infty}(\Z^n)$, then the domain of   $\lambda(x)g(-i\nabla)$ contains $\cup_{s>0}W_2^s(\T^n_\theta)$. In particular, $\lambda(x)g(-i\nabla)$ is densely defined.
\end{enumerate}
\end{proposition}

\begin{remark}
 The inclusion $\cup_{s>0}W_2^s(\T^n_\theta)\subset L_2(\T^n_\theta)$ is strict. For instance, if $u=\sum u_k U^k$ with $u_k=(1+|k|)^{-\frac{n}{2}}[\log(1+|k|)]^{-1}$, then $u\in L_2(\T^n_\theta)$, but $u\not \in W_2^s(\T^n_\theta)$ for any $s>0$. 
\end{remark}

\section{Cwikel Estimates on NC Tori} \label{sec:Cwikel}
In this section, we establish Cwikel type estimates on NC tori for Schatten classes and their weak versions. 

\subsection{Schatten classes and weak Schatten classes} \label{subsec:schatten}
We briefly review the main definitions and properties of Schatten classes and weak Schatten classes on Hilbert space. We refer to~\cite{GK:AMS69, Si:AMS05} for further details. 

In what follows we let $\sH$ be a (separable) Hilbert space with inner product $\scal{\cdot}{\cdot}$. We also denote by $\sK$ the (closed) ideal of compact operators on $\sH$. Given any operator $T\in \sK$ we let $\mu=(\mu_j(T))_{j\geq 0}$ be its sequence of singular values, i.e., $\mu_j(T)$ is the $(j+1)$-the eigenvalue counted with multiplicity of the absolute value $|T|=\sqrt{T^*T}$. By the min-max principle~\cite{GK:AMS69, RS4:1978} we have
\begin{align}
 \mu_j(T)&=\min \left\{\|T_{|E^\perp}\|;\ \dim E=j\right\},
 \label{eq:min-max} \\
 & = \min\left\{\|T-R\|;\ \op{rk}(R)\leq j\right\}, \qquad j\geq 0.
 \label{eq:min-max2} 
\end{align}
This implies the following properties of singular values (see, e.g., \cite{GK:AMS69, Si:AMS05}), 
\begin{gather}
 \mu_j(T)=\mu_j(T^*)=\mu_j(|T|), 
  \label{eq:Quantized.properties-mun1}\\
 \mu_{j+k}(S+T)\leq \mu_j(S) + \mu_k(T),
 \label{eq:Quantized.properties-mun2}\\
 \mu_j(ATB)\leq \|A\| \mu_j(T) \|B\|,\qquad A, B\in \sL(\sH).  
\end{gather}
In addition, we have the monotonicity principle, 
\begin{equation}
 0\leq T\leq S \ \Longrightarrow \ \mu_j(T)\leq \mu_j(S) \quad \forall j\geq 0.
 \label{eq:monotonicity-principle} 
\end{equation}

In what follows, we denote by $\sL_1$ the trace-class with norm, 
\begin{equation*}
  \|T\|_{\sL_1}:=\Tr |T|=\sum_{j\geq 0} \mu_j(T), \qquad T\in \sL_1. 
\end{equation*}
Recall that for $p\in (0,\infty)$ the Schatten class $\sL_p$ consist of operators $T\in \sK$ such that $|T|^p$ is trace-class. It is equipped with the quasi-norm, 
\begin{equation*}
 \|T\|_{\sL_p}:=\big(\Tr|T|^p\big)^{\frac1p}= \bigg( \sum_{j\geq 0}\mu_j(T)^p\bigg)^{\frac1p}, \qquad T\in \sL_p. 
\end{equation*}
We obtain a quasi-Banach ideal. For $p\geq 1$ the $\sL_p$-quasi-norm is actually a norm, and so in this case $\sL_p$ is a Banach ideal. 

For $p\in (0,\infty)$, the weak Schatten class  $\sL_{p,\infty}$ is defined by
\begin{equation*}
 \sL_{p,\infty}:=\left\{T\in \sK; \ \mu_j(T)=\op{O}\big(j^{-\frac1p}\big)\right\}. 
\end{equation*}
 This is a two-sided ideal. We equip it with the quasi-norm,
\begin{equation}\label{def:lp_infty_quasinorm}
 \|T\|_{\sL_{p,\infty}}:=\sup_{j\geq 0}\;(j+1)^{\frac{1}{p}}\mu_j(T), \qquad T\in \sL_{p,\infty}. 
\end{equation}
Note that 
\begin{equation}
 T\in \sL_{p,\infty} \Longleftrightarrow |T|^p\in \sL_{1,\infty} \quad \text{and} \quad \|T\|_{\sL_{p,\infty}} = \left(\| |T|^p \|_{\sL_{1,\infty}}\right)^{\frac1{p}}.
 \label{eq:quantized.lpinf-lunf}  
\end{equation}
In addition, for $p> 1$, the quasi-norm $\|\cdot \|_{p,\infty}$ is equivalent to the norm, 
\begin{equation*}
 \|T\|_{\sL_{p,\infty}}':=  \sup_{N\geq 1} N^{-1+\frac1{p}}\sum_{j<N} \mu_j(T), \qquad T\in \sL_{p,\infty}.
\end{equation*}
More precisely (see, e.g., \cite[\S1.7]{Si:AMS05}), we have 
\begin{equation*}
  \|T\|_{\sL_{p,\infty}} \leq \|T\|_{\sL_{p,\infty}}' \leq \frac{p}{p-1}\|T\|_{\sL_{p,\infty}}. 
\end{equation*}
Therefore, in this case $\sL_{p,\infty}$ is a Banach ideal with respect to the equivalent norm $\|\cdot \|_{p,\infty}'$. 

We also have the following version of H\"older's inequality for weak Schatten classes. 

\begin{proposition}[\cite{GK:AMS69, Si:AMS05, SZ:2021}] \label{prop:Cwikel.Holder-weak-Schatten}
 Suppose that $p^{-1}+q^{-1}=r^{-1}$. If $S\in \sL_{p,\infty}$ and $T\in \sL_{q,\infty}$, then $ST\in\sL_{r,\infty}$ with norm inequality,
\begin{equation*}
 \|ST\|_{\sL_{r,\infty}} \leq \gamma(p,q)\|S\|_{\sL_{p,\infty}}  \|T\|_{\sL_{q,\infty}}, \qquad \gamma(p,q):=p^{-\frac1{q}}q^{-\frac1{p}} (p+q)^{\frac1{p}+\frac1{q}}.
\end{equation*}
Moreover, the constant $\gamma(p,q)$ above is optimal.
\end{proposition}

\begin{remark}
 The fact that the constant $\gamma(p,q)$ is optimal is established in~\cite{SZ:2021}. 
\end{remark}

\begin{remark}\label{eq:Cwikel.Holder-constant.pr}
 In terms of $p$ and $r$, we have $\gamma(p,q)=p^{1/r}r^{-1/p}(p-r)^{1/p-1/r}$. 
\end{remark}

Recall that, given non-increasing sequences of non-negative numbers $a =(a_j)_{j\geq 0} $ and $b =(b_j)_{j\geq 0} $, the  Hardy-Littlewood-P\'olya
submajorization  $a\prec\prec b$ means that
\begin{equation*}
 \sum_{j=0}^Na_j \leq \sum_{j=0}^Nb_j \quad \forall N\geq 0  \quad \text{and} \quad  \sum_{j=0}^\infty a_j = \sum_{j=0}^\infty b_j < \infty.
\end{equation*}
The  Hardy-Littlewood-P\'olya majorization  $a\prec b$ means that
\begin{equation*}
 a\prec \prec b \qquad \text{and} \qquad  \sum_{j=0}^\infty a_j = \sum_{j=0}^\infty b_j < \infty.
\end{equation*}
Given compact operators $S$ and $T$ on Hilbert space we shall write $S\prec \prec T$ (resp., $S\prec  T$) if $\mu(S) \prec \prec \mu(T)$ (resp., $\mu(S)\prec \mu(T)$). 

The Schatten $\sL_p$-norms are monotone under submajorization
\begin{equation}\label{schatten_is_fully_symmetric}
    T\prec\prec S \Longrightarrow \|T\|_{\sL_p}\leq \|S\|_{\sL_p}.
\end{equation}
Note that if $T\prec\prec S,$ then $\|T\|_{\sL_{p,\infty}}'\prec\prec \|S\|_{\sL_{p,\infty}}'.$ 
Therefore, for $p>1$, we have
\begin{align}\label{L_p_infty_is_fully_symmetric}
    T\prec\prec S \Longrightarrow \|T\|_{\sL_{p,\infty}} \leq \frac{p}{p-1}\|S\|_{\sL_{p,\infty}}.
\end{align}

\subsection{Cwikel type estimates} 
We shall distinguish between the following cases:
\begin{itemize}
 \item \underline{Case I:}  $\lambda(x)g(-i\nabla)\in \sL_p$ with $p\geq 2$ and $\lambda(x)g(-i\nabla)\in \sL_{p,\infty}$ with $p>2$.\smallskip 
 
 \item   \underline{Case II:}  $\lambda(x)g(-i\nabla)\in \sL_p$ and $\lambda(x)g(-i\nabla)\in \sL_{p,\infty}$ with $0<p<2$.\smallskip
 
 \item \underline{Case III:} $\lambda(x)g(-i\nabla)\in \sL_{2,\infty}$. 
\end{itemize}

The first case is dealt with by the following result. 

\begin{theorem}[see also~\cite{LeSZ:2020, MSX:CMP19}]\label{Cwikel2+}
        The following holds. 
        \begin{enumerate}
        \item If $x \in L_p(\T^n_\theta)$ and $g \in \ell_{p}(\Itgr^n)$ with $p\geq 2$, then $\lambda(x)g(-i\nabla)\in \sL_p$, and we have 
                   \begin{equation}
                \|\lambda(x)g(-i\nabla)\|_{\sL_p} \leq \|x\|_{L_p}\|g\|_{\ell_p}.
                \label{eq:Cwikel2+Lp}
            \end{equation}
    
       \item Assume $p>2$. There is $c_{+}(p)>0$ such that, if $x \in L_p(\T^n_\theta)$ and $g \in \ell_{p,\infty}(\Itgr^n)$ with $p>2$, then $\lambda(x)g(-i\nabla)\in \sL_{p,\infty}$, and we have 
                   \begin{equation}
                \|\lambda(x)g(-i\nabla)\|_{\sL_{p,\infty}} \leq c_{+}(p)
               \|x\|_{L_p}\|g\|_{\ell_{p,\infty}}.
                \label{eq:Cwikel2+Lpinf}
            \end{equation}
            Moreover, the best constant $c_+(p)$ satisfies  
\begin{equation*}
 c_+(p)\leq \left(\frac{130p}{p-2}\right)^{\frac12}. 
\end{equation*}
\end{enumerate}
\end{theorem}
\begin{proof}
The proof of the Cwikel estimate~(\ref{eq:Cwikel2+Lp}) for $p=2$ is elementary and is presented in the proof of~\cite[Theorem~3.1]{MSX:CMP19}. In fact, in this case we actually have an equality, 
\[
    \|\lambda(x)g(-i\nabla)\|_{\mathcal{L}_2} = \|x\|_{L_2}\|g\|_{\ell_2}, \qquad x\in L_2(\T^n_\theta), \ g\in \ell_2(\Z^n). 
\]
In particular, Hypothesis~3.1 of~\cite{LeSZ:2020} is satisfied with a constant equal to 1. 

It follows from \cite[Lemma 3.3]{LeSZ:2020} that if $x \in L_2(\T^n_\theta)$ and $g\in \ell_\infty(\Z^n)$,  then
\begin{equation}\label{eq:Cwikel.sub-majorization}
     \mu(\lambda(x)g(-i\nabla))^2 \prec\prec 130\mu(x\otimes g)^2.
\end{equation}
Thus, if $x\in L_\infty(\T^n_\theta)$ and $g\in \ell_p(\Z^n)$ are positive, then by using~\eqref{schatten_is_fully_symmetric} we get
\[
    \big\||\lambda(x)g(-i\nabla)|^2\big\|_{L_{\frac{p}2}} \leq 130\big\||x|^2\otimes |g|^2\big\|_{L_{\frac{p}2}\big(L_{\infty}(\T^n_\theta)\overline{\otimes} \ell_{\infty}(\Z^n)\big)} 
     \leq 130 \|x\|_{L_p}^2\|g\|_{\ell_p}^2. 
\]
As $   \||\lambda(x)g(-i\nabla)|^2\|_{L_{\frac{p}{2}}} = \|\lambda(x)g(-i\nabla)\|_{L_{p}}^2$ we get
\begin{equation*}\label{MSX_cwikel_with_sqrt_532}
   \big\|\lambda(x)g(-i\nabla)\big\|_{L_p} \leq \sqrt{130} \|x\|_{L_p}\|g\|_{\ell_p}.
\end{equation*}
Let $N\geq 1,$ and let $\theta^{\oplus N}$ be the $(Nn)\times (Nn)$-matrix given by the direct sum of $N$ copies of $\theta$.  Note that
\begin{equation}\label{eq:Cwikel.powersN1}
  \|x\|_{L_p(\T^n_\theta)} = \|x^{\otimes N}\|_{L_p\big(\T^{Nn}_{\theta^{\oplus N}}\big)}^{\frac1N}, \qquad  \|g\|_{\ell_p(\Itgr^n)} = \|g^{\oplus N}\|_{\ell_p(\Itgr^{Nn})}^{\frac1N}.
\end{equation}
Moreover, as $ L_2(\T^{Nn}_{\theta^{\oplus N}}) = L_2(\T^n_\theta)^{\otimes N}$, where $\otimes$ is the Hilbert space tensor product, we also have
\begin{equation}\label{eq:Cwikel.powersN2}
 \|T\|_{\sL_p(\T^n_\theta)} = \|T\|_{\sL_p\big(L_2\big(\T^{Nn}_{\theta^{\oplus N}}\big)\big)}^{\frac1N} \qquad \forall T\in \sL_p(L_2(\T^n_\theta)). 
\end{equation}

The inequality~(\ref{MSX_cwikel_with_sqrt_532}) holds on $\T^n_{\theta^{\oplus N}}$. Thus, if $x\in L_p(\T^n_\theta)$ and $g\in \ell_p(\Z^n)$, then 
\begin{align*}
  \big\|\ (\lambda(x)g(-i\nabla))^{\otimes N}\big\|_{\sL_{p}\big(L_2\big(\T^{nN}_\theta\big)\big)}&  = 
    \big \|\lambda(x^{\otimes N})g^{\otimes N}(-i\nabla)\big\|_{\mathcal{L}_{p}(L_2(\T^{nN}_\theta))}\\
     & \leq \sqrt{130}\big\|x^{\otimes N}\big\|_{L_p\big(\T^{Nn}_{\theta^{\oplus N}}\big)}\big\|g^{\oplus N}\big\|_{\ell_p(\Z^{Nn})}.
\end{align*}
Combining this with~(\ref{eq:Cwikel.powersN1})--(\ref{eq:Cwikel.powersN2}) gives 
\[
    \big\|\lambda(x)g(-i\nabla)\big\|_{\sL_p} \leq 130^{\frac{1}{2N}}\|x\|_{L_p}\|g\|_{\ell_{p}}.
\]
Letting $N\to\infty$ then yields 
\[
    \big\|\lambda(x)g(-i\nabla)\big\|_{\sL_p} \leq \|x\|_{L_p}\|g\|_{\ell_{p}}.
\]
This is the inequality~(\ref{eq:Cwikel2+Lp}). 

To estimate the weak Schatten norms we can do a similar argument using \eqref{L_p_infty_is_fully_symmetric} in place of \eqref{schatten_is_fully_symmetric}. Let $p>2$. By using~(\ref{L_p_infty_is_fully_symmetric}) and~(\ref{eq:Cwikel.sub-majorization}) and arguing as above shows that if $x\in L_\infty(\T^n_\theta)$ and $g\in \ell_{p,\infty}(\Z^n)$ are positive, then 
\begin{equation}\label{eq:Cwikel2+Lpinf-532}
    \big\|\lambda(x)g(-i\nabla)\big\|_{\sL_{p,\infty}} \leq \left(\frac{130p}{p-2}\right)^{\frac12}\|x\|_{L_p}\|g\|_{\ell_{p,\infty}}.
\end{equation}
As above the inequality continues to hold if $x$ and $g$ are not positive. Thanks to the density of $L_\infty(\T^n_\theta)$ in $L_p(\T^n_\theta)$, it follows that, if $x\in L_p(\T^n_\theta)$ and $g\in \ell_{p,\infty}(\Z^n)$, then $\lambda(x)g(-i\nabla)\in \sL_{p,\infty}$ and the inequality~(\ref{eq:Cwikel2+Lpinf-532}) holds. This gives the 2nd part of Theorem~\ref{Cwikel2+}. The proof is complete.  
\end{proof}

\begin{remark}
 The estimates~(\ref{eq:Cwikel2+Lp})--(\ref{eq:Cwikel2+Lpinf}) are established in~\cite[Theorem~3.1]{MSX:CMP19} up to unspecified constants depending on $p$. This is a consequence of the results of~\cite{LeSZ:2020}. More generally, the results of~\cite{LeSZ:2020} allows us to get Cwikel-type estimates for any interpolation space between $\sL_2$ and $\sK$ (see~\cite{MSX:CMP19}). 
\end{remark}

\begin{remark}
 The analogue of~(\ref{eq:Cwikel2+Lp}) on $\R^n$ is known as the Kato-Seiler-Simon inequality~\cite{SeSi:CMP75}. 
\end{remark}

Combining Theorem~\ref{Cwikel2+} with Proposition~\ref{prop:Sobolev-embeddingLp} we immediately obtain the following statement. 

\begin{corollary}\label{Cwikel2+Sobolev}
  Let $s\in (2,\infty)$ and set $s=n(1/2-p^{-1})$. The following holds. 
        \begin{enumerate}
        \item If $x \in W_2^s(\T^n_\theta)$ and $g \in \ell_{p}(\Itgr^n)$, then $\lambda(x)g(-i\nabla)\in \sL_p$, and we have 
                   \begin{equation*}
                \|\lambda(x)g(-i\nabla)\|_{\sL_p} \leq c_p\|x\|_{W_2^s}\|g\|_{\ell_p}.
            \end{equation*}
    
       \item If $x \in W_2^s(\T^n_\theta)$ and $g \in \ell_{p,\infty}(\Itgr^n)$, then $\lambda(x)g(-i\nabla)\in \sL_{p,\infty}$, and we have 
                   \begin{equation*}
                \|\lambda(x)g(-i\nabla)\|_{\sL_{p,\infty}} \leq c_p\|x\|_{W_2^s}\|g\|_{\ell_{p,\infty}}.
            \end{equation*}
\end{enumerate}
    \end{corollary}

Case~II  is dealt with by the following result, the proof of which is postponed to next section.  

\begin{theorem}\label{Cwikel2-}
        Suppose that $0<p<2$. The following holds.
        \begin{enumerate}
            \item Let $x\in L_2(\T^n_\theta)$ and $g\in \ell_p(\Z^n)$. Then $\lambda(x)g(-i\nabla)\in \sL_p$, and we have
            \begin{equation*}
                \|\lambda(x)g(-i\nabla)\|_{\sL_p} \leq \|x\|_{L_2}\|g\|_{\ell_p}.
            \end{equation*}
            \item There is $c_{-}(p)>0$ such that, if $x\in L_2(\T^n_\theta)$ and $g\in \ell_{p,\infty}(\Z^n)$, then $\lambda(x)g(-i\nabla)\in \sL_{p,\infty}$, and we have
            \begin{equation}\label{eq:Cwikel2-Lpinf} 
                \|\lambda(x)g(-i\nabla)\|_{\sL_{p,\infty}} \leq c_{-}(p)\|x\|_{L_2}\|g\|_{\ell_{p,\infty}}.
            \end{equation}
            Moreover, the best constant $c_{-}(p)$ satisfies 
                 \begin{equation*}
                          c_{-}(p)\leq 2^{\frac1p}(2-p)^{-\frac1p}. 
                 \end{equation*}
        \end{enumerate}
    \end{theorem}
      
Finally, for Case~III we shall prove the following result. 
  
\begin{theorem}\label{L_2_infty_cwikel}
    For any $x \in L_p(\T_\theta)$ with $p>2$ and $g \in \ell_{2,\infty}(\Itgr^n)$, the operator $\lambda(x)g(-i\nabla)$ is in $\sL_{2,\infty}$, and we have 
    \begin{equation}\label{eq:Cwikel-L2inf}
        \|\lambda(x)g(-i\nabla)\|_{\sL_{2,\infty}} \leq c_p\|x\|_{L_p}\|g\|_{\ell_{2,\infty}}.
    \end{equation}
    Moreover, the best constant in the inequality is less than or equal to 
    \begin{equation}\label{eq:Cwikel-L2inf-c2q}
         \inf_{1< q<2}\left\{ \gamma_1(p,q)\gamma_2(p,q) c_+(p)^{\frac{p(2-q)}{2(p-q)}} c_{-}(q)^{\frac{q(p-2)}{2(p-q)}}\right\} , 
   \end{equation}
   where $c_+(p)$ and $c_+(q)$ are the best constants in the inequalities~(\ref{eq:Cwikel2+Lpinf}) and~(\ref{eq:Cwikel2-Lpinf}), respectively, and the constants $\gamma_1(p,q)$ and $\gamma_2(p,q)$ arise from real interpolation (see~Eqs.~(\ref{eq:Cwikel.interpolation-ell})--(\ref{eq:Cwikel.interpolation-sL}) below). 
\end{theorem}
 \begin{proof}
The proof uses real interpolation theory. Once again, for background on interpolation theory we refer to the survey of Connes~\cite[Appendix IV.B]{Co:NCG} and the references therein. The main reference for real interpolation theory there is the article of Lions-Peetre~\cite{LP:IHES64} (see also~\cite{BS:PAM88, KPS:AMS82}).  

Let $x\in L_{p}(\T_\theta^n)$, $p>2$. By Theorem~\ref{Cwikel2+}, for every $g\in \ell_{p,\infty}(\Z^n)$, the operator $\lambda(x)g(-i\nabla)$ is in $\sL_{p,\infty}$. Thus, we have a linear operator $\Phi_x: \ell_{p,\infty}(\Z^n)\rightarrow \sL_{p,\infty}$ given by
\begin{equation*}
 \Phi_x(g)=\lambda(x)g(-i\nabla), \qquad g\in \ell_{p,\infty}. 
\end{equation*}
Furthermore, the Cwikel-type estimate~(\ref{eq:Cwikel2+Lpinf}) gives
\begin{equation}
 \| \Phi_x(g)\|_{\sL_{p,\infty}} \leq c_+(p)\|x\|_{L_p} \|g\|_{\ell_{p,\infty}} \qquad \forall g \in \ell_{p,\infty}.
 \label{eq:Cwikel2p} 
\end{equation}
In addition, as $L_p(\T^n_\theta)\subset L_2(\T^n_\theta)$, it follows from Theorem~\ref{Cwikel2-} that, given any $q\in (1,2)$, if $g\in \ell_{q,\infty}(\Z^n)$, then $\Phi_x(g)\in \sL_{q,\infty}$, and we have
\begin{equation}
  \big\| \Phi_x(g)\big\|_{\sL_{q,\infty}} \leq c_{-}(q) \|x\|_{L_2} \|g\|_{\ell_{q,\infty}}\leq c_{-}(q) \|x\|_{L_p} \|g\|_{\ell_{q,\infty}}.
   \label{eq:Cwikel2q}  
\end{equation}

As $2\in (q,p)$, the spaces  $\ell_{2,\infty}(\Z^n)$ and $\sL_{2,\infty}$ are real interpolation spaces for the pairs of Banach spaces $(\ell_{q,\infty}(\Z^n),\ell_{p,\infty}(\Z^n))$ and $(\sL_{q,\infty},\sL_{p,\infty})$ with the exact same exponents. Namely, in the notation of~\cite[Appendix~IV.B]{Co:NCG}, and with the same norm equivalences (see e.g.~\cite[Theorem~2.10]{Si:AMS05}), we have
$\ell_{2,\infty}(\Z^n)=[\ell_{q,\infty}(\Z^n),\ell_{p,\infty}(\Z^n)]_{\theta,\infty}$ and $\sL_{2,\infty}=[\sL_{q,\infty},\sL_{p,\infty}]_{\theta,\infty}$, with $\theta\in(0,1)$ such that $1/2=(1-\theta)q^{-1}+\theta p^{-1}$ (see~\cite[{\S}IV.2.$\alpha$ \& Appendix~IV.B]{Co:NCG} and \cite[Section~4]{DDP:IEOT92}). That is, there are constants $\gamma_1(p,q)>0$ and $\gamma_2(p,q)>0$ such that 
\begin{gather}
 \gamma_2(p,q)^{-1}\|g\|_{\ell_{2,\infty}} \leq \|g\|_{[\ell_{q,\infty},\ell_{p,\infty}]_{\theta,\infty}} \leq  \gamma_1(p,q) \|g\|_{\ell_{2,\infty}} 
 \qquad \forall g\in \ell_{2,\infty}(\Z^n), 
 \label{eq:Cwikel.interpolation-ell}\\
  \gamma_2(p,q)^{-1}\|T\|_{\sL_{2,\infty}} \leq \|T\|_{[\sL_{q,\infty},\sL_{p,\infty}]_{\theta,\infty}} \leq  \gamma_1(p,q) \|T\|_{\sL_{2,\infty}} \qquad \forall T\in \sL_{2,\infty}.
  \label{eq:Cwikel.interpolation-sL}
\end{gather}
We assume that $\gamma_1(p,q)$ and $\gamma_2(p,q)$ are the best constants in the above inequalities. 

By combining the estimates~(\ref{eq:Cwikel2p})--(\ref{eq:Cwikel2q}) with real interpolation theory (see~\cite[Theorem~IV.B.2]{Co:NCG}) and by using the 
inequalities~(\ref{eq:Cwikel.interpolation-ell})--(\ref{eq:Cwikel.interpolation-sL}) shows that $\Phi_x$ induces a continuous linear map from $\ell_{2,\infty}(\Z^n)$ to $\sL_{2,\infty}$, and, for all $g\in \ell_{2,\infty}(\Z^n)$, we have
\begin{align*}
  \big\| \lambda(x)g(-i\nabla)\big\|_{\sL_{2,\infty}} =  \big\| \Phi_x(g)\big\|_{\sL_{2,\infty}} & \leq  \gamma_1(p,q)\gamma_2(p,q) 
  (c_{-}(q)\|x\|_{L_p})^{1-\theta}  (c_+(p)\|x\|_{L_p})^\theta \|g\|_{\ell_{2,\infty}}
  \\  & \leq  
   \gamma_1(p,q)\gamma_2(p,q) c_{+}(p)^\theta c_{-}(q)^{1-\theta} \|x\|_{L_p} \|g\|_{\ell_{2,\infty}}.  
\end{align*}
This gives the inequality~(\ref{eq:Cwikel-L2inf}). 

Here $\theta$ is such that $1/2=(1-\theta)q^{-1}+\theta p^{-1}$. That is, 
\begin{equation*}
 \theta = \left(\frac12-\frac1q\right) \left(\frac1p-\frac1q\right)^{-1}= \frac{p(2-q)}{2(p-q)}, \qquad 1-\theta =  \frac{q(p-2)}{2(p-q)}. 
\end{equation*}
Thus, if we denote by $c(p)$ the best constant in~(\ref{eq:Cwikel-L2inf}), then
\begin{equation*}
 c(p)\leq \gamma_1(p,q)\gamma_2(p,q) c_{+}(p)^{\frac{p(2-q)}{2(p-q)}} c_{-}(q)^{\frac{q(p-2)}{2(p-q)}} \qquad \forall q\in (1,2). 
\end{equation*}
Taking the infimum over $q\in (1,2)$ shows that~(\ref{eq:Cwikel-L2inf-c2q}) is an upper bound for $c(p)$. The proof is complete. 
\end{proof} 

\begin{remark}
 Theorem~\ref{L_2_infty_cwikel} does not hold for $x\in L_2(\T^n_\theta)$ (see Remark~\ref{rmk:L1infty-does-not-hold}). 
\end{remark}

\begin{corollary}\label{L_2_infty_cwikel-Sobolev}
        For any $x \in W_2^s(\T_\theta)$ with $s>0$ and $g \in \ell_{2,\infty}(\Itgr^n)$, the operator $\lambda(x)g(-i\nabla)$ is in $\sL_{2,\infty}$, and we have 
        \begin{equation*}
            \|\lambda(x)g(-i\nabla)\|_{\sL_{2,\infty}} \leq c_s\|x\|_{W_2^s}\|g\|_{\ell_{2,\infty}}.
        \end{equation*}
    \end{corollary}
\begin{proof}
Let  $x \in W_2^s(\T_\theta)$ with $s>0$ and  $g \in \ell_{2,\infty}(\Itgr^n)$. Let $p>2$ be such that $s\geq n(1/2-p^{-1})$. By Proposition~\ref{prop:Sobolev-embeddingLp} this ensures us that $W_2^s(\T^n_\theta)$ embeds continuously into $L_p(\T^n_\theta)$. Combining this with Theorem~\ref{L_2_infty_cwikel} then gives the result.  
\end{proof}

\section{Proof of Theorem~\ref{Cwikel2-}}\label{sec:Cwikel2-}
In this section, we prove Theorem~\ref{Cwikel2-}. The proof attempts to follow the approach to the proof of Cwikel-type estimates for noncommutative Euclidean spaces in~\cite{LeSZ:2020}. However, some significant simplification occurs thanks to Proposition~\ref{majorization2-} below. 
  
\begin{lemma}[\cite{LeSZ:2020}]\label{anti_monotone}
    Suppose that $0 < p < 2$, and let $S,T\in \sL_2$ be such that $|S|^2 \prec |T|^2$. 
    \begin{enumerate}
        \item[(i)] If $S\in \sL_p$, then $T\in \sL_p$, and $\|T\|_{\sL_p}\leq \|S\|_{\sL_p}$. 
        \item[(ii)] If $S \in \sL_{p,\infty}$, then $T\in \sL_{p,\infty}$, and $\|T\|_{\sL_{p,\infty}} \leq 2^{\frac1p}(2-p)^{-\frac1p}\|S\|_{\sL_{p,\infty}} $. 
    \end{enumerate}
\end{lemma}
\begin{proof}
 The first part is the contents of the first part of~\cite[Proposition 2.7]{LeSZ:2020}. The 2nd part of \cite[Proposition 2.7]{LeSZ:2020} gives the 2nd part of the lemma with an upper bound $\|T\|_{\sL_{p,\infty}} \leq c_p\|S\|_{\sL_{p,\infty}}$ for an unspecified
    constant $c_p$. We can get the required constant by a slight modification of the proof in~\cite{LeSZ:2020}.
    
Suppose that $S\in \sL_{p,\infty}$. For $t\geq 0$, set $f(t)=\mu_{[t]}(S)$ and $g(t)=\mu_{[t]}(T)$, where $[\cdot]$ is the floor function. Note 
that $\sup_{t\geq 0}t^{1/p} f(t)= \sup_{j\geq 0}(j+1)^{\frac1p} \mu_j(S)= \|S\|_{\sL_{p,\infty}}$. Likewise, $\|T\|_{\sL_{p,\infty}}=\sup_{t\geq 0}t^{1/p} g(t)$. Moreover, as pointed out in the proof of~\cite[Proposition 2.7]{LeSZ:2020}, the fact that $T^2\prec S^2$ implies that
\begin{equation}\label{eq:Cwikel2inf.submajorization}
 \int_t^2g(s)^2ds \leq  \int_t^2f(s)^2ds, \qquad t\geq 0. 
\end{equation}
 We also observe that in~Eq.~(14) of~\cite{LeSZ:2020}  the constant is equal to $p(2-p)^{-1}$. Thus, we have
\begin{equation}\label{eq:Cwikel2inf.integral-ineq}
 \int_t^\infty f(s)^2 ds \leq \bigg(\frac{p}{2-p}\bigg)\|S\|_{\sL_{p,\infty}}^2t^{1-\frac2p}, \qquad t>0. 
\end{equation}
 
 Given any $a>1$ the fact that $g(t)$ is non-increasing ensures that
\begin{equation*}
 t(a-1)g(at)^2 \leq \int_t^{at}g(s)^2ds \leq  \int_t^{\infty}g(s)^2ds, \qquad t>0. 
\end{equation*}
Combining this with~(\ref{eq:Cwikel2inf.submajorization})--(\ref{eq:Cwikel2inf.integral-ineq}) gives 
\begin{equation*}
 tg(at)^2 \leq (a-1)^{-1} \bigg(\frac{p}{2-p}\bigg)\|S\|_{\sL_{p,\infty}}^2t^{1-\frac2p}, \qquad t>0. 
\end{equation*}
Thus, 
\begin{equation*}
 t^{\frac1p} g(at)  \leq (a-1)^{-\frac12} \bigg(\frac{p}{2-p}\bigg)^{\frac12}\|S\|_{\sL_{p,\infty}}, \qquad t\geq 0
\end{equation*}
 It follows that 
\begin{equation*}
 \sup_{t\geq 0} t^{\frac1p} g(t)=a^{\frac1p}  \sup_{t\geq0} t^{\frac1p} g(t) \leq a^{\frac1p}(a-1)^{-\frac12} \bigg(\frac{p}{2-p}\bigg)^{\frac12}\|S\|_{\sL_{p,\infty}}. 
\end{equation*}
 This shows that $T\in \sL_{p,\infty}$, and we have
\begin{equation*}
 \|T\|_{\sL_{p,\infty}} \leq a^{\frac1p}(a-1)^{-\frac12} \bigg(\frac{p}{2-p}\bigg)^{\frac12}\|S\|_{\sL_{p,\infty}} \qquad \forall a>1. 
\end{equation*}
The function $a\rightarrow  a^{\frac1p}(a-1)^{-\frac12}$ reaches its minimum on $(1,\infty)$ at $a=2(2-p)^{-1}$. The minimum value is  
$2^{1/p}p^{-1/2}(2-p)^{1/2-1/p}.$
 This yields the inequality $\|T\|_{\sL_{p,\infty}} \leq 2^{1/p}(2-p)^{-1/p}\|S\|_{\sL_{p,\infty}} $. The proof is complete. 
 \end{proof}
  
Recall that a sequence of bounded operators $(T_j)_{j\geq 0}$ on Hilbert space is called \emph{right-disjoint} (resp., \emph{left-disjoint}) when $T_jT_l^*=0$ (resp., $T_j^*T_l=0$) when $j\neq l$. When all the operators $T_j$ are selfadjoint notions of left-disjointness and right-disjointness are both equivalent to the condition $T_jT_l=0$ when $j\neq 0$. In that case we simply that say that we have a \emph{disjoint} sequence.  
    
  \begin{lemma}[{\cite[Lemma 2.9]{LeSZ:2020}}]\label{disjoint_lemma}
        Let $(T_j)_{j\geq 0}$ be a sequence in $\sL_2$ such that $\sum \|T_j\|_{\sL_2}^2 < \infty$. Assume further that the sequence $(T_j)_{j\geq 0}$ is left-disjoint or right-disjoint. Then, we have 
        \begin{equation*}
            \mu^2\bigg(\bigoplus_{j\geq0} T_j\bigg)\prec \mu^2\bigg(\sum_{j\geq 0} T_j\bigg).
        \end{equation*}
    \end{lemma}
 
 The operator-theoretic underpinning of Theorem~\ref{Cwikel2-} is the following majorization estimate.

\begin{proposition}\label{majorization2-} 
Let  $x\in L_2(\T^n_\theta)$ and $g\in \ell_{2}(\Z^n)$. Then, we have
\begin{equation}
 \|x\|_{L_2}^2 \mu^2\big(g(-i\nabla)\big) \prec  \mu^2\big(\lambda(x)g(-i\nabla)\big).
 \label{eq:majorization} 
\end{equation}
\end{proposition}
 \begin{proof} 
The fact that $g\in \ell_{2}(\Z^n)$ ensures us that $g(-i\nabla)\in \sL_2$. Furthermore, as $x\in L_2(\T^n_\theta)$ we know by Theorem~\ref{Cwikel2+} that $\lambda(x)g(-i\nabla)$ is in $\sL_2$ as well. 

Bearing this in mind, given any $k \in \Itgr^n$, we denote by $\Pi_k$ the orthogonal projection onto $\C U^k$. Thus, for all $u=\sum \hat{u}_kU^k$ in $L_2(\T^n_\theta)$, we have 
        \begin{equation*}
            \Pi_ku = \scal{U^k}{u}U^k= \hat{u}_k U^k.  
        \end{equation*}
Note that $(\Pi_k)_{k\in Z^n}$ is a disjoint family of rank~1 projections.  

For $k\in \Z^n$, we also set
\begin{equation*}
 T_k:= \lambda(x)g(-i\nabla)\Pi_k. 
\end{equation*}
Each operator $T_k$ has rank~$\leq 1$, and so this is a Hilbert-Schmidt operator. Moreover, if $k\neq l$, then $T_{k}T_{l}^*=(\lambda(x)g(-i\nabla))\Pi_k\Pi_l 
(\lambda(x)g(-i\nabla))^*=0$. Thus, the sequence $(T_k)_{k\in \Z^n}$ is right-disjoint.  We also note that, as $U^k$ is in the domain of $\lambda(x)$, for all $u=\sum \hat{u}_kU^k$ in $L_2(\T^n_\theta)$, we have 
\begin{equation*}
 T_ku= \lambda(x)g(-i\nabla)\Pi_ku=\hat{u}_k\lambda(x)g(-\nabla)U^k=\hat{u}_kg(k)\lambda(x)U^k. 
\end{equation*}
\begin{claim*}
For all $k\in \Z^n$, we have
 \begin{equation*}
  |T_k|=\|x\|_{L_2} |g(k)|\Pi_k. 
\end{equation*}
\end{claim*}
\begin{proof}[Proof of the claim]
 Let $k\in \Z^n$, and set $A=\lambda(x)g(-i\nabla)$. Note that $A$ is a bounded operator by Proposition~\ref{eq:boundedness.Lp-ellp}. We have 
\begin{equation*}
 T_k^*T_k=(A\Pi_k)^*(A\Pi_k)=\Pi_k A^*A\Pi_k. 
\end{equation*}
We observe that
 \begin{equation*}
 \Pi_k A^*AU^k=\scal{U^k}{A^*AU^k}U^k= \scal{AU^k}{AU^k}U^k. 
\end{equation*}
As $AU^k=\lambda(x)g(-\nabla)U^k=g(k)\lambda(x)U^k$, we get
\begin{equation*}
 \scal{AU^k}{AU^k}=|g(k)|^2\scal{\lambda(x)U^k}{\lambda(x)U^k}=|g(k)|^2\tau\big[xU^{k}(xU^k)^*\big]=|g(k)|^2\tau[xx^*]=|g(k)|^2\|x\|_{L_2}^2.   
\end{equation*}
Thus, 
\begin{equation*}
 \Pi_k A^*AU^k=\scal{AU^k}{AU^k}U^k=|g(k)|^2\|x\|_{L_2}^2U^k. 
 \end{equation*}
It then follows that, for all $u=\sum \hat{u}_k U^k$, we have
\begin{equation*}
 T_k^*T_k= \Pi_k A^*A\Pi_ku=\hat{u}_k\Pi_k A^*AU^k=\hat{u}_k|g(k)|^2\|x\|_{L_2}^2 U^k=|g(k)|^2\|x\|_{L_2}^2\Pi_ku. 
\end{equation*}
That is, $T_k^*T_k=|g(k)|^2\|x\|_{L_2}^2\Pi_k$. Thus,
\begin{equation*}
 |T_k|=\sqrt{T_k^*T_k}=|g(k)|\|x\|_{L_2}\sqrt{\Pi_k}=  |g(k)|\|x\|_{L_2}\Pi_k. 
\end{equation*}
This proves the claim. 
\end{proof}

Combining the claim above with the fact that $g\in\ell_{2}(\Z^n)$ gives
\begin{equation*}
 \sum_{k\in \Z^n} \|T_k\|^2_{\sL_2}=    \sum_{k\in \Z^n} \||T_k|\|^2_{\sL_2}= \sum_{k\in \Z^n} |g(k)|^2\|x\|_{L_2}^2 \|\Pi_k\|_{\sL_2} 
 = \|x\|_{L_2}^2 \|g\|^2_{\ell_2}<\infty.  
\end{equation*}
All this allows us to apply Lemma~\ref{disjoint_lemma} to get
\begin{equation}
            \mu^2\bigg(\bigoplus_{k\in \Itgr^n} T_k\bigg) \prec \mu^2\bigg(\sum_{k\in \Itgr^n} T_k\bigg)=  \mu^2\big(\lambda(x)g(-i\nabla)\big).
            \label{eq:majorization-oplus}  
\end{equation}

Set $S=\bigoplus_{k\in \Itgr^n} T_k$. The claim above implies that
\(
|S|= \bigoplus_{k\in \Itgr^n} |T_k| =   \bigoplus_{k\in \Itgr^n} \|x\|_{L_2}|g(k)|\Pi_k. 
\)
Therefore, we have
\begin{equation*}
  \mu(S) = \mu\bigg( \bigoplus_{k\in \Itgr^n} \|x\|_{L_2}|g(k)|\Pi_k\bigg)=\|x\|_{L_2} \mu\big(|g(-i\nabla)|\big)=\|x\|_{L_2} \mu\big(g(-i\nabla)\big). 
\end{equation*}
Combining this with~(\ref{eq:majorization-oplus}) then gives the majorization $\|x\|_{L_2}^2 \mu^2(g(-i\nabla)) \prec  \mu^2(\lambda(x)g(-i\nabla))$. The proof is complete.  
\end{proof}
 
 We are now in a position to prove Theorem~\ref{Cwikel2-}. 
  
\begin{proof}[Proof of Theorem~\ref{Cwikel2-}] 
Let $x\in L_2(\T^n_\theta)$ and $g\in \ell_{p,\infty}$, $p<2$. Recall that $g(-i\nabla)\in \sL_{p,\infty}$ and $\|g(-i\nabla)\|_{\sL_{p,\infty}}=\|g\|_{\ell_{p,\infty}}$. Note also that $g\in \ell_2(\Z^n)$, since $p<2$. Thus, combining the majorization~(\ref{eq:majorization}) with Lemma~\ref{anti_monotone} shows that $\lambda(x)g(-i\nabla) \in \sL_{p,\infty}$, and we have 
\begin{equation*}
 \big\| \lambda(x)g(-i\nabla)\big\|_{\sL_{p,\infty}}\leq2^{\frac1p}(2-p)^{-\frac1p}\|x\|_{L_2}\|g(-i\nabla)\|_{\sL_{p,\infty}} \leq 2^{\frac1p}(2-p)^{-\frac1p}\|x\|_{L_2}\|g\|_{\ell_{p,\infty}}. 
\end{equation*}

If $g\in \ell_p(\Z^n)$, then $g(-i\nabla)\in \sL_p$ and $\|g(-i\nabla)\|_{\sL_p}= \|g\|_{\ell_p}$. Therefore, in the same way as above, we deduce that $\lambda(x)g(-i\nabla) \in \sL_{p}$, and we have
\begin{equation*}
 \big\| \lambda(x)g(-i\nabla)\big\|_{\sL_{p}}\leq \|x\|_{L_2} \|g(-i\nabla)\|_{\sL_{p}} \leq \|x\|_{L_2}\|g\|_{\ell_{p}}. 
\end{equation*}
This completes the proof of Theorem~\ref{Cwikel2-}. 
   \end{proof} 
  
   
\section{Specific Cwikel Estimates}\label{sec:specific-Cwikel} 
 It is worth specializing the Cwikel estimates of Section~\ref{sec:Cwikel} to operators of the forms $\lambda(x)\Delta^{-n/2p}$ and $\Delta^{-n/4p}\lambda(x)\Delta^{-n/4p}$, since these estimates are used in the derivation of the CLR inequalities on NC tori in Section~\ref{CLR_section}. In the terminology of~\cite{SZ:arXiv20} these Cwikel-type estimates are called \emph{specific} Cwikel estimates. 
  
 In what follows we set
\begin{equation*}
 \nu_0(n):= \sup_{\lambda \geq 1} \lambda^{-\frac{n}{2}}\#\left\{k\in \Z^n\setminus 0; \ |k|\leq \sqrt{\lambda} \right\}. 
\end{equation*}

 \begin{theorem}\label{thm.specific-Cwikel}
 The following holds. 
\begin{enumerate}
 \item If $p>2$ and $x\in L_p(\T^n_\theta)$, then $\lambda(x) \Delta^{-n/2p}\in \sL_{p,\infty}$, and we have 
 \begin{equation}
 \big\| \lambda(x) \Delta^{-\frac{n}{2p}}\big\|_{\sL_{p,\infty}} \leq c_{+}(p) \nu_o(n)^{\frac1{p}}\|x\|_{L_p}.
 \label{eq:specific.cpq-nu2+}  
 \end{equation}
 
 \item If $0<p<2$ and $x\in L_2(\T^n_\theta)$, then $\lambda(x) \Delta^{-n/2p}\in \sL_{p,\infty}$, and we have 
 \begin{equation*}
 \big\| \lambda(x) \Delta^{-\frac{n}{2p}}\big\|_{\sL_{p,\infty}} \leq c_{-}(p) \nu_o(n)^{\frac1{p}}\|x\|_{L_2}. 
 \end{equation*}
 
  \item If $x\in L_p(\T^n_\theta)$, $p>2$, then $\lambda(x) \Delta^{-n/4}\in \sL_{2,\infty}$, and we have 
 \begin{equation*}
 \big\| \lambda(x) \Delta^{-\frac{n}{4}}\big\|_{\sL_{2,\infty}} \leq c_2(p) \nu_o(n)^{\frac1{2}}\|x\|_{L_p}\|x\|_{L_p}, 
 \end{equation*}
 where we have set
 \begin{equation}\label{eq:Cwikel-speficied-c2p}
  c_{2}(p)=2^{-\frac1{p}}p^{-\frac12}(p-2)^{\frac1p-\frac12}c_{+}(p). 
\end{equation}
\end{enumerate}
Here $c_{+}(p)$ and $c_{-}(p)$ are the best constants in the $\sL_{p,\infty}$-Cwikel estimates~(\ref{eq:Cwikel2+Lpinf}) and~(\ref{eq:Cwikel2-Lpinf}), respectively.
\end{theorem}
\begin{proof}
Suppose that $p>2$, and let $x\in L_p(\T^n_\theta)$. 
Observe that $\Delta^{-n/2p}=\lambda(x)g_p(-i\nabla)$, where $g_p(k)=|k|^{-n/p}$ for $k\neq 0$ and $g_p(0)=0$. In particular, $g_p\in \ell_{p,\infty}(\Z^n)$. Thus, by specializing Theorem~\ref{Cwikel2+} to $g=g_p$ we see that $\lambda(x) \Delta^{-n/2p}\in \sL_{p,\infty}$, and we have 
\begin{equation}\label{eq:specific.cpq-gp} 
   \big\| \lambda(x) \Delta^{-\frac{n}{2p}}\big\|_{\sL_{p,\infty}} \leq c_{+}(p) \|g_p\|_{\ell_{p,\infty}}\|x\|_{L_p}.
\end{equation}
 
Let us arrange the family $\{|k|^2;\ k\in \Z^n\setminus 0\}$ as a non-decreasing sequence 
$1=\lambda_1\leq \lambda_2\leq \cdots \leq \lambda_j \leq \cdots $. We have $\mu_j(g_p)=\lambda_{j+1}^{-n/2p}$ for all $j \geq 0$, and hence
\begin{equation}\label{eq:specific.norm-gp}
 \|g_p\|_{\ell_{p,\infty}} = \sup_{j\geq 0}(j+1)^{\frac1{p}} \mu_j(g_p)= \sup_{j\geq 0} (j+1)^{\frac1{p}}\lambda_{j+1}^{-\frac{n}{2p}} 
 = \big( \sup_{j\geq 1} j \lambda_j^{-n/2}\big)^{\frac1p}. 
\end{equation}
Set $N_0(\lambda)=\#\{j\geq 1; \lambda_j\leq \lambda\}$, $\lambda \geq 1$. In view of the definition of  $\nu_0(n)$ we have
\begin{equation}\label{eq:specific.N0lambda} 
 N_0(\lambda)=\#\left\{k\in \Z^n\setminus 0; \ |k|^2\leq \lambda \right\} =\#\left\{k\in \Z^n\setminus 0; \ |k|\leq \sqrt{\lambda} \right\} \leq \nu_0(n)\lambda^{\frac{n}2}. 
\end{equation}
It follows that $j\leq N_0(\lambda_j)\leq \nu_0(n)\lambda_j^{-n/2}$ for all $j\geq 1$. Together with~(\ref{eq:specific.norm-gp}) this implies that
$\|g_p\|_{\ell_{p,\infty}}\leq \nu_0(n)^{1/p}$. Combining this with~(\ref{eq:specific.cpq-gp}) then gives~(\ref{eq:specific.cpq-nu2+}). This gives the 1st part. 

If $p<2$, then we can similarly prove the 2nd part by using Theorem~\ref{Cwikel2-} instead of Theorem~\ref{Cwikel2+}. 

It remains to prove the 3rd part. Let $x\in L_p(\T^n_\theta)$, $p>2$, and let $q$ be such that $2^{-1}=p^{-1}+q^{-1}$. In particular, 
$\lambda(x)\Delta^{-n/4}=\lambda(x)\Delta^{-n/2p}\cdot \Delta^{-n/2q}$. As $p>2$ the operator $\lambda(x)\Delta^{-n/2p}$ is in $\sL_{p,\infty}$ by the 1st~part. Note also that $\Delta^{-n/2q}=g_q(-\nabla)\in \sL_{q,\infty}$. Therefore, by the version of H\"older's inequality  provided by Proposition~\ref{prop:Cwikel.Holder-weak-Schatten} the operator $\lambda(x)\Delta^{-n/4}$ is in $\sL_{2,\infty}$, and we have
\begin{equation*}
 \big\|\lambda(x) \Delta^{-\frac{n}4} \big\|_{\sL_{2,\infty}} \leq  \gamma(p,q)\big\|\lambda(x) \Delta^{-\frac{n}{2p}} \big\|_{\sL_{p,\infty}}  \big\|\Delta^{-\frac{n}{2q}} \big\|_{\sL_{q,\infty}}, 
\end{equation*}
where $\gamma(p,q)=p^{-\frac1{q}}q^{-\frac1{p}} (p+q)^{\frac1{p}+\frac1{q}}$. 
Here $\|\Delta^{-\frac{n}{2q}} \|_{\sL_{q,\infty}}=\|g_q(-\nabla) \|_{\sL_{q,\infty}} = \|g_q \|_{\ell_{q,\infty}}\leq \nu_0(n)^{1/q}$. Thus, by using~(\ref{eq:specific.cpq-nu2+}) we get
\begin{align*}
 \big\|\lambda(x) \Delta^{-\frac{n}4} \big\|_{\sL_{2,\infty}} & \leq \gamma(p,q) \big(c_+(p) \nu_0(n)^{\frac1p}(n) \|x\|_{L_p} \big) \nu_0(n)^{1/q}\\
 & \leq c_2(p) \nu_0(n)^{\frac12}\|x\|_{L_p},
\end{align*}
 where we have set $c_2(p)=\gamma(p,q)c_+(p)$. In fact, as $p^{-1}+q^{-1}=2^{-1}$, by Remark~\ref{eq:Cwikel.Holder-constant.pr} we have $\gamma(p,q)=2^{-1/p}p^{-1/2}(p-2)^{1/p-1/2}$, and so $c_2(p)$ is given by~(\ref{eq:Cwikel-speficied-c2p}). This proves the 3rd part. The proof is complete. 
\end{proof}

\begin{remark}
 As~(\ref{eq:specific.N0lambda}) shows, $N_0(\lambda)$ is the number of non-zero integer points in the ball centered at the origin of radius $\sqrt{\lambda}$. This ball is contained in the cube $[-\sqrt{\lambda},\sqrt{\lambda}]^n$. There are at most $2\sqrt{\lambda}+1$ integers in the interval $[-\sqrt{\lambda},\sqrt{\lambda}]$. Thus, 
\begin{equation*}
 N_0(\lambda)\leq \big(2\sqrt{\lambda}+1\big)^n-1 =\lambda^{\frac{n}2}\left(\big(2+1/\sqrt{\lambda}\big)^n -\lambda^{-\frac{n}2}\right)\leq \lambda^{\frac{n}2}(3^n-1). 
\end{equation*}
It then follows that 
\begin{equation*}
 \nu_0(n)=\sup_{\lambda\geq 1} \lambda^{-\frac{n}2}N_0(\lambda) \leq 3^n-1. 
\end{equation*}
For $n=2$ it can be shown that $N_0(\lambda) \leq 4 \lambda$ for $\lambda\geq 1$ (see~\cite[\S2]{IL:StPMJ20}). Thus, in this case we get 
$\nu_0(2)\leq 4$. 
\end{remark}

Let us now turn to the Cwikel operators $\Delta^{-n/4p}\lambda(x)\Delta^{-n/4p}$. Suppose that $p^{-1}+2q^{-1}=1$. As mentioned in Section~\ref{sec:NCtori}, if $x\in L_{p}(\T^n_\theta)$, then $\lambda(x)$ makes sense as a bounded operator from $L_q(\T^n_\theta)$ to its anti-linear dual $L_q(\T^n_\theta)^*$. Moreover, by Lemma~\ref{lem:Boundedness.sandwich1}, if $s>n/2p$, or if $s=n/2p$ and $p>1$, then $\lambda(x)$ induces a bounded operator $\lambda(x):W_2^{s}(\T^n_\theta)\rightarrow W_2^{-s}(\T^n_\theta)$. This allows us to makes sense of the composition $\Delta^{-s/2} \lambda(x) \Delta^{-s/2}$  as a bounded operator on $L_2(\T^n_\theta)$ (\emph{cf}.~Proposition~\ref{prop:Boundedness.sandwich2}). 

Let $y\in L_{2p}(\T^n_\theta)$. As $(2p)^{-1}+q^{-1}=\frac{1}{2}(p^{-1}+2q^{-1})=\frac12$, we know from Proposition~\ref{prop:left-reg-Lp} that $\lambda(y)$ makes sense as a bounded operator from $L_q(\T^n_\theta)$ to $L_2(\T^n_\theta)$. By duality we get a bounded operator $\lambda(y)^*:L_2(\T^n_\theta) \rightarrow L_q(\T^n_\theta)^*$ such that
\begin{equation}
 \acou{\lambda(y)^*u}{v}=\scal{u}{\lambda(y)v}, \qquad u\in L_2(\T^n_\theta), \ v\in L_q(\T^n_\theta).
 \label{eq:specific.lambday*}
\end{equation}
In addition, if $s>n/2p$, or if $s=n/2p$ and $p>1$, then we have a continuous embedding of $W_2^s(\T^n_\theta)$ into $L_q(\T^n_\theta)$,  and so the operator $\lambda(x) \Delta^{-s/2}$ is bounded on $L_2(\T^n_\theta)$. 

\begin{lemma}\label{lem:spefici.decomposition-lambdax} 
 Let $x\in L_p(\T^n_\theta)$, $p\geq 1$, be of the form $x=yz$ with $y,z\in L_{2p}(\T^n_\theta)$. In addition, assume that either $s>n/2p$, or $s=n/2p$ and $p>1$. Then 
\begin{equation*}
 \lambda(x)=\lambda(y^*)^*\lambda(z), \qquad \Delta^{-\frac{s}{2}} \lambda(x) \Delta^{-\frac{s}{2}}= \big[\lambda(y^*)\Delta^{-\frac{s}{2}}\big]^*\lambda(z)\Delta^{-\frac{s}{2}}. 
\end{equation*}
\end{lemma}
\begin{proof}
 Let $u,v\in L_q(\T^n_\theta)$,  $p+2q^{-1}=1$. In view of~(\ref{eq:left-reg-LpLqLq*}) we have
 \begin{equation*}
 \acou{\lambda(x)u}{v}=\tau\big[v^*xu\big]=\tau\big[v^*yzu\big]. 
\end{equation*}
As $zu=\lambda(z)u\in L_2(\T^n_\theta)$ and $v^*y=(\lambda(y^*)v)^*\in L_2(\T^n_\theta)$, we get 
\begin{equation}
 \acou{\lambda(x)u}{v}=\tau\big[\left(\lambda(y^*)v\right)^*\lambda(z)u\big]=\bigscal{\lambda(z)u}{\lambda(y^*)v}. 
 \label{eq:specific.lambdax-scallambdayz}
\end{equation}
Combining this with~(\ref{eq:specific.lambday*}) gives
\begin{equation*}
  \acou{\lambda(x)u}{v}=\bigacou{\lambda(y^*)^*\lambda(z)u}{v} \qquad \forall u,v\in L_q(\T^n_\theta).
\end{equation*}
That is, $ \lambda(x)=\lambda(y^*)^*\lambda(z)$. 

Suppose that, either $s>n/2p$, or $s=n/2p$ and $p>1$. Using~(\ref{eq:Boundedness.Lambdas-dual}) and~(\ref{eq:specific.lambdax-scallambdayz}) shows that, for all $u,v\in L_2(\T^n_\theta)$, we have
\begin{align*}
 \bigscal{\Delta^{-\frac{s}{2}} \lambda(x) \Delta^{-\frac{s}{2}}u}{v}&= \bigacou{\lambda(x) \Delta^{-\frac{s}{2}}u}{\Delta^{-\frac{s}{2}} v}\\
 & = \bigscal{\lambda(z) \Delta^{-\frac{s}{2}}u}{\lambda(y^*)\Delta^{-\frac{s}{2}} v}\\ &=  \bigscal{\big[\lambda(y^*)\Delta^{-\frac{s}{2}}\big]^*\lambda(z)\Delta^{-\frac{s}{2}}u}{v}. 
\end{align*}
 It then follows that $\Delta^{-s/2} \lambda(x) \Delta^{-s/2}=[\lambda(y^*)\Delta^{-s/2}]^*\lambda(z)\Delta^{-s/2}$. The proof is complete.
\end{proof}

\begin{lemma}\label{lem:specific.productLp}
 Any $x\in L_p(\T^n_\theta)$, $p\geq 1$, can be written in the form $x=yz$ with $y,z\in L_{2p}(\T^n_\theta)$ such that $\|y\|_{L_{2p}}=\|z\|_{L_{2p}}=(\|x\|_{L_p})^{1/2}$.  
\end{lemma}
\begin{proof}
 Let $x=u|x|$ be the polar decomposition of $x$. As mentioned in Section~\ref{sec:NCtori} the phase $u$ is a partial isometry in $L_\infty(\T^n_\theta)$, and hence $\|u\|_{L_\infty}\leq 1$. Thus, if we set $y=u|x|^{1/2}$ and $z=|x|^{1/2}$, then $x=yz$ and $z\in L_{2p}(\T^n_\theta)$ with $\|z\|_{L_{2p}}=(\|x\|_{L_p})^{1/2}$. In addition, H\"older's inequality (see Proposition~\ref{prop:Holder}) ensures that $y=u|x|^{1/2}\in L_{2p}(\T^n_\theta)$, and we have
\begin{equation*}
 \|y\|_{L_{2p}} \leq \|u\|_{L_\infty}\||x|^{1/2}\|_{L_{2p}} \leq (\|x\|_{L_p})^{1/2}. 
\end{equation*}
As H\"older's inequality also gives
\begin{equation*}
 \|x\|_{L_p}=\|yz\|_{L_p} \leq \|y\|_{L_{2p}}\|z\|_{L_{2p}} \leq \|y\|_{L_{2p}}\big(\|x\|_{L_p}\big)^{1/2}, 
\end{equation*}
 we deduce that $\|y\|_{L_{2p}}=(\|x\|_{L_p})^{1/2}$. The proof is complete. 
\end{proof}
  
 We are now in a position to establish Cwikel estimates  for the operators $\Delta^{-n/4p}\lambda(x)\Delta^{-n/4p}$. 
 
 \begin{theorem}\label{thm:Specific-Cwikel.sandwiched} 
 The following hold. 
\begin{enumerate}
 \item If $x\in L_p(\T^n_\theta)$, $p>1$, then $\Delta^{-{n}/{4p}} \lambda(x) \Delta^{-{n}/{4p}} \in \sL_{p,\infty}$, and we have
\begin{equation}\label{eq:specific.sando-Cwikel1+}
 \big\| \Delta^{-\frac{n}{4p}} \lambda(x) \Delta^{-\frac{n}{4p}}\big \|_{\sL_{p,\infty}}\leq 2^{\frac1p}c_+(2p)^2 \nu_0(n)^{\frac{1}{p}} \|x\|_{L_p}. 
\end{equation}
  
 \item If $x\in L_p(\T^n_\theta)$, $p>1$, then $\Delta^{-n/4} \lambda(x) \Delta^{-n/4}  \in \sL_{1,\infty}$,  and we have
\begin{equation}
 \big\| \Delta^{-\frac{n}{4}} \lambda(x) \Delta^{-\frac{n}{4}} \big \|_{\sL_{1,\infty}}\leq 2c_2(2p)^2 \nu_0(n) \|x\|_{L_p}.
 \label{eq:specific.sando-Cwikel-1infty}  
\end{equation}

 \item If $x\in L_1(\T^n_\theta)$ and $p<1$, then $\Delta^{-{n}/{4p}}  \lambda(x) \Delta^{-{n}/{4p}} \in \sL_{p,\infty}$, and we have
\begin{equation}\label{eq:specific.sando-Cwikel1-}
 \big\| \Delta^{-\frac{n}{4p}} \lambda(x) \Delta^{-\frac{n}{4p}}\big \|_{\sL_{p,\infty}}\leq  2^{\frac1p}c_{-}(2p)^2 \nu_0(n)^{\frac{1}{p}} \|x\|_{L_1}. 
\end{equation}
\end{enumerate}
Here $c_{+}(2p)$ and $c_{-}(2p)$ are the best constants in the $\sL_{2p,\infty}$-Cwikel estimates~(\ref{eq:Cwikel2+Lpinf}) and~(\ref{eq:Cwikel2-Lpinf}), respectively,
 and $c_2(2p)$ is given by~(\ref{eq:Cwikel-speficied-c2p}). 
\end{theorem}
\begin{proof}
 Let $p>0$, and suppose that $q=\max(p,1)$ if $p \neq 1$ or $q>1$ if $p=1$. Note that $q\geq 1$. Set $c(p,q)$ to be equal to $c_{+}(2p)$ (resp., $c_{-}(2p)$, $c_2(2q)$) if $p>1$ (resp., $p<1$, $p=1$).   We need to show that if $x\in L_q(\T^n_\theta)$, then  $\Delta^{-{n}/{4p}} \lambda(x) \Delta^{-{n}/{4p}} \in \sL_{p,\infty}$, and we have 
\begin{equation}\label{eq:specific.sando-Cwikel}
 \big\| \Delta^{-\frac{n}{4p}} \lambda(x) \Delta^{-\frac{n}{4p}}\big \|_{\sL_{p,\infty}}\leq 2^{\frac1p} c(p,q)^2 \nu_0(n)^{\frac{1}{p}} \|x\|_{L_q}. 
\end{equation}
 
Let $x\in L_q(\T^n_\theta)$. By Lemma~\ref{lem:specific.productLp} we may write $x=yz$, with  $y,z\in L_{2q}(\T^n_\theta)$ such that $\|y\|_{L_{2q}}=\|z\|_{L_{2q}}=(\|x\|_{L_q})^{1/2}$. 
 By Lemma~\ref{lem:spefici.decomposition-lambdax} we have
\begin{equation*}
 \Delta^{-\frac{n}{4p}} \lambda(x) \Delta^{-\frac{n}{4p}}= \big[\lambda(y^*) \Delta^{-\frac{n}{4p}}\big]^* \lambda(z) \Delta^{-\frac{n}{4p}}. 
\end{equation*}
 It follows from Theorem~\ref{thm.specific-Cwikel} that $[\lambda(y^*) \Delta^{-\frac{n}{4p}}]^*$ and $ \lambda(z) \Delta^{-\frac{n}{4p}}$ are both in the weak Schatten class $\sL_{2p,\infty}$, and we have 
\begin{gather}\label{eq:Specific-Cwikel.sandwiched-yLq}
\big\| \lambda(z) \Delta^{-\frac{n}{24}}\big\|_{\sL_{2p,\infty}} \leq c(p,q)\nu_0(n)^{\frac1{2p}}\|z\|_{L_{2q}}=c(p,q)\nu_0(n)^{\frac1{2p}}\|x\|^{\frac12}_{L_{2q}},
\\
\big\| \big[\lambda(y^*) \Delta^{-\frac{n}{4p}}\big]^*\big\|_{\sL_{2p,\infty}}=\big\| \lambda(y^*) \Delta^{-\frac{n}{2p}}\big\|_{\sL_{2p,\infty}} 
\leq c(p,q)\nu_0(n)^{\frac1{2p}}\|y^*\|_{L_{2p}}= c(p,q)\nu_0(n)^{\frac1{2p}}\|x\|^{\frac12}_{L_{2q}}. 
\end{gather}
The H\"older's inequality for weak Schatten classes provided by Proposition~\ref{prop:Cwikel.Holder-weak-Schatten} then implies that $ \Delta^{-\frac{n}{4p}} \lambda(x) \Delta^{-\frac{n}{4p}}\in \sL_{p,\infty}$, and we have
\begin{align*}
 \big\| \Delta^{-\frac{n}{4p}} \lambda(x) \Delta^{-\frac{n}{4p}} \|_{\sL_{p,\infty}} & \leq   2^{\frac1p}\big\|\big[ \lambda(y^*) \Delta^{-\frac{n}{4p}}\big]^*\big\|_{\sL_{2p,\infty}} 
 \big\| \lambda(z) \Delta^{-\frac{n}{4p}}\big\|_{\sL_{2p,\infty}}\\ & \leq 2^{\frac1p}c(p,q)^2 \nu_0(n)^{\frac{1}{p}} \|x\|_{L_q}.  
\end{align*}
This proves~(\ref{eq:specific.sando-Cwikel}). The proof is complete. 
\end{proof}

 \begin{remark}\label{rmk:Specific.positive-case}
 If $x\geq 0$, then the inequalities~(\ref{eq:specific.sando-Cwikel1+})--(\ref{eq:specific.sando-Cwikel1-}) hold without the extra $2^{\frac1p}$-factor. Indeed, if $p$ and $q$ are as in the proof of Theorem~\ref{thm:Specific-Cwikel.sandwiched} and 
 $0\leq x\in L_q(\T^n_\theta)$, then we may take $y=z=\sqrt{x}$. In this case 
 $ \Delta^{-{n}/{4p}} \lambda(x) \Delta^{-{n}/{4p}}= \big[\lambda(\sqrt{x}) \Delta^{-{n}/{4p}}\big]^* \lambda(\sqrt{x}) \Delta^{-{n}/{4p}}= 
 | \lambda(\sqrt{x}) \Delta^{-{n}/{4p}}|^2$. Thus, 
\begin{equation*}
  \big\| \Delta^{-\frac{n}{4p}} \lambda(x) \Delta^{-\frac{n}{4p}} \big\|_{\sL_{p,\infty}} = 
  \left\|  \big| \lambda(\sqrt{x}) \Delta^{-{n}/{4p}}\big|^2 \right\|_{\sL_{p,\infty}}= 
   \big\|  \lambda(\sqrt{x}) \Delta^{-\frac{n}{4p}} \big\|^2_{\sL_{2p,\infty}} 
 \end{equation*}
 Combining this with the inequality~(\ref{eq:Specific-Cwikel.sandwiched-yLq}) for $y=\sqrt{x}$ then gives
\begin{equation*}
  \big\| \Delta^{-\frac{n}{4p}} \lambda(x) \Delta^{-\frac{n}{4p}} \big\|_{\sL_{p,\infty}}  \leq c(p,q)^2\nu_0(n)^{\frac{1}{p}} \|x\|_{L_q},
\end{equation*}
which proves our claim. 
 \end{remark}

\begin{remark}\label{rmk:L1infty-does-not-hold} 
The estimate~(\ref{eq:specific.sando-Cwikel-1infty}) does not hold for $x\in L_1(\T^n_\theta)$ (see~\cite[Lemma~5.7]{LPS:JFA10}).   
\end{remark}

\begin{remark}
For the ordinary torus $\T^n$, i.e., $\theta=0$, by a recent result of Sukochev-Zanin~\cite{SZ:arXiv20} the estimate~(\ref{eq:specific.sando-Cwikel-1infty}) 
still holds if $x$ is in the Orlicz space $\LLogL(\T^n)$ (see also~\cite{So:PLMS95}). It would be interesting to have an analogue of this result for $\theta \neq 0$.  
\end{remark}

\section{Borderline Birman-Schwinger Principle} \label{sec:Birman-Schwinger}
In this section, we establish a ``borderline'' version of the abstract Birman-Schwinger principle for the number of negative values of relatively form-compact perturbations of non-negative semi-bounded operators on Hilbert space. 

To  a large extent we follow the original approach of Birman-Solomyak~\cite{BS:1989}, which we recast in the framework of~\cite{Si:AMS15}. However, our ultimate result (Theorem~\ref{thm:Borderline-BSP}) seems to be new, at least at the level of generality it is stated. In particular, it can be applied to 
Schr\"odinger operators $\Delta_g+V$, and more generally fractional Schr\"odinger operators $\Delta_g^\alpha+V$, in the following setups: closed Riemannian manifolds, compact manifolds with boundary with suitable boundary condition, or even hyperbolic manifolds with infinite volume.

Throughout this section we let $\sH$ be a (separable) Hilbert  space.

\subsection{Glazman's Lemma} 
The proof of the abstract Birman-Schwinger principle by Birman-Solomyak~\cite{BS:1989} relies on Glazman's lemma. We shall now briefly recall this result and set some notation along the way. 

Let $A$ be a selfadjoint operator on $\sH$ which is bounded or semi-bounded. We denote by $Q_A$ its quadratic form. If $A$ is bounded, then $Q_A$ is the quadratic form on $\sH$ defined by
\begin{equation*}
 Q_A(\xi,\eta):=\scal{A\xi}{\eta} \qquad \forall \xi,\eta \in \sH. 
\end{equation*}
If $A$ is semi-bounded, then  $Q_A$ has domain $\dom (A-\lambda)^{1/2}$ with $\mu\in \R\setminus \Sp(A)$.
We also denote by  $\Sp_{\text{ess}}(A)$ the essential spectrum of $A$, i.e., the complement of the discrete spectrum (which consists of isolated eigenvalues with finite multiplicity). 

Given any $\lambda\in \R$ we set 
\begin{equation}
 N^+(A;\lambda) = \dim \big(\ran \car_{(\lambda,\infty)}(A)\big), \qquad  N^-(A;\lambda) = \dim\big(\ran \car_{(-\infty,\lambda)}(A)\big). 
 \label{eq:BSP.counting-functions}
\end{equation}
Thus, if $[\lambda,\infty)$ (resp., $(-\infty,\lambda]$) does not meet the essential spectrum of $A$, then $N^+(A;\lambda)$ (resp., $N^-(A;\lambda)$) is the number of eigenvalues of $A$  counted with multiplicity that are~$>\lambda$ (resp., $<\lambda$). 

In what follows, we denote by $\sF^\pm(A;\lambda)$ (resp., $\sF_0^\pm(A;\lambda)$) the collection of all subspaces $F\subset \dom(Q_A)$ such that $\pm Q_A(\xi,\xi)>\pm \lambda \|\xi\|^2$ (resp., $\pm Q_A(\xi,\xi)\geq \pm \lambda \|\xi\|^2$) on $F\setminus 0$. 
We have the following variational principles. 

\begin{lemma}[Glazman's Lemma~\cite{BS:Book, Gl:GIFML}] \label{lem:BSP.Glazman} 
 For all $\lambda \in \R$, we have
\begin{align}
 N^\pm(A;\lambda) = & \max\left\{ \dim F; \ F\in \sF^\pm(A;\lambda)\right\},
  \label{eq:BSP.Glazman}\\
 = & \min\left\{ \dim F^\perp; \ F\in \sF^\mp_0(A;\lambda)\right\}. 
 \label{eq:BSP.Glazman0}
\end{align}
\end{lemma}

\begin{remark}
The variational principle~(\ref{eq:BSP.Glazman0}) is due to Glazman~\cite{Gl:GIFML}. A proof of~(\ref{eq:BSP.Glazman}) is given in~\cite{BS:Book} (see~Theorem~10.2.3; see also Theorem~9.2.6 for the compact case). 
\end{remark}

In what follows given selfadjoint operators $A$ and $B$ we shall write $A\leq B$ if $\dom(Q_A)=\dom(Q_B)$ and $Q_A(\xi,\xi)\leq Q_B(\xi,\xi)$ for all $\xi$ in $\dom(Q_A)=\dom(Q_B)$. Recall that Glazman's Lemma implies the following monotonicity principle. 

\begin{corollary}\label{cor:BSP.monotonicity}
 Let $A$ and $B$ be selfadjoint operators on $\sH$ such that $A\leq B$. Then, for all $\lambda \in \R$, we have
\begin{equation*}
 N^+(A;\lambda)\leq N^+(B;\lambda), \qquad   N^-(A;\lambda)\geq N^-(B;\lambda). 
\end{equation*}
\end{corollary}

\subsection{The Abstract Birman-Schwinger principle} 
From now on we let $H$ be a (densely defined) selfadjoint operator on $\sH$ with non-negative spectrum containing $0$. Its quadratic form $Q_H$ has domain $\dom(Q_H)=\dom (H+1)^{\frac12}$. We denote by $\sH_{+}$ the Hilbert space obtained by endowing $\dom(Q_H)$ with the Hilbert space norm, 
\begin{equation*}
 \|\xi\|_{+}=\big(Q_H(\xi,\xi)+\|\xi\|^2\big)^{\frac12}=\big\|(1+H)^{1/2}\xi\big\|, \qquad \xi \in \dom(Q_H). 
\end{equation*}
We also let $\sH_{-}$ be the Hilbert space of continuous \emph{anti-linear} functionals on $\sH_{+}$. Note that we have a continuous embedding $\iota:\sH\hookrightarrow \sH_{-}$ with dense range given by 
\begin{equation*}
 \acou{\iota(\xi)}{\eta}=\scal{\xi}{\eta}, \qquad \xi\in \sH, \ \eta \in \sH_+. 
\end{equation*}

The operator $(H+1)^{1/2}:\sH_+\rightarrow \sH$ is a unitary isomorphism. By selfadjointness it extends to a unitary isomorphism $(H+1)^{1/2}:\sH\rightarrow \sH_{-}$ such that
\begin{equation}
\big\langle (H+1)^{1/2}\xi , \eta \big\rangle
 =\big\langle{\xi}\big|{(H+1)^{1/2}\eta} \big\rangle, \qquad \xi\in \sH, \ \eta \in \sH_+,
  \label{eq:BSP.sqrt}
\end{equation}
where $\acou{\cdot}{\cdot}:\sH_{-}\times \sH_{+}\rightarrow \C$ is the natural duality pairing. In particular, we have bounded inverses $(H+1)^{-1/2}:\sH\rightarrow \sH_{+}$ and $(H+1)^{-1/2}:\sH_{-}\rightarrow \sH$ such that
\begin{equation}
 \bigscal{(H+1)^{-1/2}\xi}{\eta}=\bigacou{\xi}{(H+1)^{-1/2}\eta}, \qquad \xi\in \sH_{-}, \ \eta \in \sH.
 \label{eq:CLR.H1sqrt-inverses}  
\end{equation}
More generally, for any $\lambda <0$, we have bounded operators $(H-\lambda)^{1/2}:\sH_{+}\rightarrow \sH$ and $(H-\lambda)^{1/2}:\sH\rightarrow \sH_{-}$ with bounded inverses 
$(H-\lambda)^{-1/2}:\sH\rightarrow \sH_{+}$ and $(H-\lambda)^{-1/2}:\sH_{-}\rightarrow \sH$.

Similarly, the operator $H$ extends to a bounded operator $H:\sH_{+}\rightarrow \sH_{-}$ such that
\begin{equation}
 \acou{H\xi}{\eta}=Q_H(\xi,\eta) \qquad \xi,\eta\in \sH_{+}.
 \label{eq:BSP.H-extension}  
\end{equation}
More generally, for any $\lambda<0$, the operator $H-\lambda=(H-\lambda)^{1/2}(H-\lambda)^{1/2}$ extends to a bounded operator from  $\sH_{+}$ to $\sH_{-}$ with bounded inverse $(H-\lambda)^{-1}:\sH_{-}\longrightarrow \sH_+$. 

In what follows, we let $V:\sH_{+}\rightarrow \sH_{-}$ be a bounded operator. We denote by $Q_V$ the corresponding quadratic form with domain $\sH_+$ and given by
\begin{equation*}
 Q_V(\xi,\eta):=\acou{V\xi}{\eta}, \qquad \xi,\eta\in \sH_+. 
\end{equation*}
We assume that $Q_V$ is \emph{symmetric} and \emph{$H$-form compact}. The latter condition means that the operator $V:\sH_{+}\rightarrow \sH_{-}$ is compact, or equivalently, $(H+1)^{-1/2}V(H+1)^{-1/2}$ is a compact operator on $\sH$. 

Our main focus is the operator $H_V:=H+V$. It makes sense as a bounded operator $H_V:\sH_+\rightarrow \sH_-$. Furthermore, as the symmetric quadratic form $Q_V$ is $H$-form compact, it is $H$-form bounded with zero $H$-bound (see~\cite[\S7.8]{Si:AMS15}). Therefore, by the KLMN theorem (see, e.g., \cite{RS2:1975, Sc:Springer12}) the restriction of $H_V$ to $\dom(H_V):=H_V^{-1}(\sH)$ is a bounded from below selfadjoint operator on $\sH$ whose quadratic form is precisely $Q_H+Q_V$. 

\begin{lemma}[see {\cite[Theorem 7.8.4]{Si:AMS15}}]\label{lem:CLR.essential-spectrum-HV} 
The following holds. 
\begin{enumerate}
 \item[(i)] For all $\lambda \not\in \Sp(H)\cup \Sp(H_V)$, the operator $(H_V-\lambda)^{-1}-(H-\lambda)^{-1}$  is compact. 
 
 \item[(ii)] The operators $H$ and $H_V$ have the same essential spectrum. 
 
 \item[(iii)] If $H$ has compact resolvent, then so does $H_V$. 
\end{enumerate}
\end{lemma}

As $H_{V}$ is bounded from below, we define the counting function $N^{-}(H_V;\lambda)$ as in~(\ref{eq:BSP.counting-functions}). If $\lambda <\inf \Sp_{\textup{ess}}(H)$, then $N^{-}(H_V;\lambda)$ this is the number of eigenvalues of $H_V$ which are~$<\lambda$. We also set  $N^{-}(H_V):=N(H_V;0)$; this is the number of ``bound states'' of $H_V$.

Given operators $V_j:\sH_{+}\rightarrow \sH_{-}$, $j=1,2$, we shall write $V_1\leq V_2$ when $Q_{V_1}\leq Q_{V_2}$, i.e., $Q_{V_1}$ and $Q_{V_2}$ are both symmetric and $Q_{V_1}(\xi,\xi)\leq Q_{V_2}(\xi,\xi)$ for all $\xi \in \sH_{+}$. In this case $Q_{H_{V_1}} (\xi,\xi)\leq Q_{H_{V_2}}(\xi,\xi)$ for all $\xi \in \sH_{+}$, and so by using Corollary~\ref{cor:BSP.monotonicity} we get
\begin{equation}
 V_1\leq V_2\ \Longrightarrow \ N(H_{V_2};\lambda)\leq N(H_{V_1};\lambda) \qquad \forall \lambda\in \R.  
 \label{eq:CLR.mono-principle-N-} 
\end{equation}

The Birman-Schwinger principle was established by Birman~\cite{Bi:AMST66} and Schwinger~\cite{Sc:PNAS61} for Schr\"odinger operators $\Delta+V$ on $\R^n$, $n\geq 3$. It is the main impetus for using Cwikel estimates to establish the Cwikel-Lieb-Rozenblum inequality (see~\cite{Cw:AM77, Si:TAMS76}). In our setting it relates the counting function 
$N(H_{V};\lambda)$ to the eigenvalues of the Birman-Schwinger operators,
\begin{equation*}
 K_V(\lambda)=-
(H-\lambda)^{-\frac12}V(H-\lambda)^{-\frac12},
 \qquad \lambda<0.  
\end{equation*}
The compactness of $V$ ensures us that $K_V(\lambda)$ is a compact operator on $\sH$. 

Note also  that $K_V(\lambda)$ is related to the quadratic form $Q_V$ by
\begin{align}
 \scal{K_V(\lambda)\xi}{\eta}&=-\bigacou{V(H-\lambda)^{-1/2}\xi}{(H-\lambda)^{-1/2}\eta} \nonumber\\
 & =-Q_V\big((H-\lambda)^{-1/2}\xi,(H-\lambda)^{-1/2}\eta\big), \qquad \xi,\eta\in \sH.
 \label{eq:BSP.KV-QV} 
\end{align}
Thus, the fact that $Q_V$ is symmetric ensures us that $K_V(\lambda)$ is a selfadjoint compact operator. We then define the counting function $N^+(K_V(\lambda);\mu)$, $\mu>0$, as above. If in addition $V\leq 0$, then $K_V(\lambda)\geq 0$, 
and so in this case the eigenvalues of $K_V(\lambda)$ agree with its singular values.  

\begin{proposition}[Abstract Birman-Schwinger Principle~{\cite[Lemma 1.4]{BS:1989}}]
\label{prop:BSP.abstractBSP} 
For all $\lambda<0$, we have
\begin{equation}
 N(H_{V};\lambda)= N^+\left(K_V(\lambda); 1\right).
 \label{eq:CLR.BKP} 
\end{equation}
\end{proposition}

\begin{remark}
 For reader's convenience a proof of Proposition~\ref{prop:BSP.abstractBSP} is given in Appendix~\ref{app:BSP}.
\end{remark}
 
 The following is a well known consequence of the Birman-Schwinger principle (see~\cite{Cw:AM77, Si:TAMS76}). For reader's convenience, a proof has been included in Appendix~\ref{app:BSP}. 

\begin{corollary}\label{cor:CLR.BKP-qinf} 
 Assume  $V\leq 0$ and $(H+1)^{-1/2}V(H+1)^{-1/2}\in \sL_{p,\infty}$ for some $q\in (0,\infty)$. Then, for all $\lambda<0$, the operator $K_V(\lambda)$ is in the weak Schatten class $\sL_{p,\infty}$, and we have 
\begin{equation}
  N(H_{V};\lambda)\leq \left\| K_V(\lambda) \right\|_{\sL_{p,\infty}}^p. 
   \label{eq:CLR.BKP-qinf-lambda}  
\end{equation}
\end{corollary}

\subsection{Borderline Birman-Schwinger principle} 
The proofs of the standard CLR inequality for Schr\"odinger operators on $\R^n$, $n\geq 3$, uses the fact that for the Laplacian on $\R^n$ with $n\geq 3$ the essential  spectrum contains $0$ and the resolvent $(\Delta -\lambda)^{-1}$ has a weak limit as  $\lambda \rightarrow 0^{-}$. This allows us to take $\lambda=0$ in~(\ref{eq:CLR.BKP-qinf-lambda}) and get an upper bound for $N^{-}(H_{V})$ (see~\cite[Theorem~7.9.11]{Si:AMS05}). 

However, on  closed manifolds and NC tori, $0$ is in the discrete spectrum of the Laplacian, and so the resolvent has a pole singularity at $\lambda=0$. This prevents us from letting $\lambda \rightarrow  0^{-}$  in~(\ref{eq:CLR.BKP-qinf-lambda}). This stresses out the need for a ``borderline'' version of the Birman-Schwinger~(\ref{eq:CLR.BKP-qinf-lambda}) when $0$ is the discrete spectrum. 

Assume  that  $0$ lies in the discrete spectrum of $H$, i.e., $0$ is an isolated eigenvalue of $H$. Thus, the essential spectrum of $H$ is contained in some  interval $[a, \infty)$ with $a>0$. As $H$ and $H_V$ have same essential spectrum by Lemma~\ref{lem:CLR.essential-spectrum-HV}, it follows that $H_V$ has at most finitely many non-positive eigenvalues. 

We denote by $H^{-1}$ the partial inverse of $H$. That is, $H^{-1}$ vanishes on $\ker H$ and inverts $H$ on $(\ker H)^\perp =\ran H$. This is a selfadjoint bounded operator with non-negative spectrum. We then define $H^{-1/2}$ to be $(H^{-1})^{1/2}$. Equivalently, $H^{-1}=f(H)$ and $H^{-1/2}=g(H)$, where $f(t)=\mathbbm{1}_{[\epsilon,\infty)}(t)t^{-1}$ and $g(t)=\mathbbm{1}_{[\epsilon,\infty)}(t)t^{-1/2}$, with  $\epsilon>0$ small enough so that $\Sp(H)\cap(0,\epsilon]=\emptyset$. We then set

We shall need the following abstract version of Lemma~\ref{lem:spefici.decomposition-lambdax}. 

\begin{lemma}\label{lem:BSP.decomposition} 
 Suppose that $V\leq 0$. Then there is a compact operator $W:\sH_+\rightarrow\sH$ such that $V=-W^*W$, where $W^*:\sH\rightarrow \sH_{-}$ is the adjoint map. 
\end{lemma}
\begin{proof}
Set $K=K_V(-1)$. As $V\leq 0$, we know that $K$ is a positive compact operator. Set $W=K^{1/2}(H+1)^{1/2}$. Then $W$ is a compact operator from $\sH_+$ to $\sH$. By duality we get a linear operator $W^*:\sH\rightarrow \sH_{-}$ such that
\begin{equation}
 \acou{W^*\xi}{\eta}=\scal{\xi}{W\eta}, \qquad \xi\in \sH, \ \eta \in \sH_{+}. 
 \label{eq:BSP.W*}
\end{equation}
In particular, for $\xi,\eta\in \sH$, we have
\begin{align*}
 \acou{W^*W\xi}{\eta}= \scal{W\xi}{W\eta} & = \bigscal{K^{\frac12}(H+1)^{\frac12}\xi}{K^{\frac12}(H+1)^{\frac12}\eta}\\ 
 &=  \bigscal{K(H+1)^{\frac12}\xi}{(H+1)^{\frac12}\eta}\\
 & =  \bigacou{(H+1)^{\frac12}K(H+1)^{\frac12}\xi}{\eta}\\
 &= -\acou{V\xi}{\eta}. 
\end{align*}
This shows that $V=-W^*W$. The proof is complete. 
\end{proof}

In what follows we denote by $\Pi_0$ the orthogonal projection on $\ker H$. This is a selfadjoint finite rank operator whose range is contained in $\sH_+$, and hence we get a bounded operator $\Pi_0:\sH \rightarrow \sH_+$. By duality we obtain a bounded operator $\Pi_0: \sH_{-}\rightarrow \sH$ such that
\begin{equation*}
 \scal{\Pi_0\xi}{\eta}=\acou{\xi}{\Pi_0\eta}, \qquad \xi\in \sH_{-}, \ \eta \in \sH. 
\end{equation*}
The composition $\Pi_0 V\Pi_0$ then is a selfadjoint finite-rank operator on $\sH$. As above, we denote by $N^{-}(\Pi_0 V \Pi_0)$ its number of negative eigenvalues counted with multiplicity. 

\begin{theorem}[Borderline Birman-Schwinger Principle]\label{thm:Borderline-BSP} 
 Suppose that $0$ is  in the discrete spectrum of $H$. Set $K_V=-H^{-1/2}VH^{-1/2}$. Then, the following hold. 
\begin{enumerate}
 \item We have
\begin{equation}
 N^{+}(K_V;1) \leq N^{-}(H_V) \leq N^{+}(K_V;1)+\dim \ker H. 
 \label{eq:BSP.borderline-NN}
\end{equation}

\item Assume further that $V\leq 0$ and $K_V\in \sL_{p,\infty}$ for some $p\in (0,\infty)$. Then
 \begin{equation}
  0\leq N^{-}(H_{V})- N^{-}(\Pi_0 V \Pi_0) \leq  \big\|K_V\big\|_{\sL_{p,\infty}}^p.
   \label{eq:CLR.BorderlineBKP1}   
\end{equation}
In particular, if $V(\ker H)\subset \ran H$, then
\begin{equation}
 N^{-}(H_{V}) \leq   \big\| K_V\big\|_{\sL_{p,\infty}}^p.
 \label{eq:CLR.BorderlineBKP2}  
\end{equation}
\end{enumerate}
\end{theorem}
\begin{proof} 
 Set $\sH_1=\ran H$ and $V_1=(1-\Pi_0)V(1-\Pi_0)$. We  also denote by $H_{V|\sH_1}$ and $K_{V|\sH_1}$ the operators on $\sH_1$ induced by $(1-\Pi_0)H_{V}(1-\Pi_0)=H_{V_1}$ and $K_V$, respectively. By applying the Birman-Schwinger principle~(\ref{eq:CLR.BKP}) to $H_{V|\sH_1}$  on $\sH_1$ we get
\begin{equation}
 N^-\big(H_{V|\sH_1}\big)=  N^{+}\big(K_{V|\sH_1};1\big)=N^{+}\big(K_{V};1\big).
 \label{eq:BSP.NN}
\end{equation}
Moreover, as $\sF^-(H_{V|\sH_1};0)\subset \sF^-(H_{V};0)$ by using the variational principles~(\ref{eq:BSP.Glazman})--(\ref{eq:BSP.Glazman0}) we get
\begin{equation}\label{eq:BSP.borderline-NNsH1}
 N^-\big(H_{V|\sH_1}\big)\leq N^-\big(H_V\big) \leq N^-\big(H_{V|\sH_1}\big) +\dim \ker H. 
\end{equation}
Combining this with~(\ref{eq:BSP.NN}) gives the inequalities~(\ref{eq:BSP.borderline-NN}). 

From now on we assume that $V\leq 0$ and $K_V\in \sL_{p,\infty}$, $p>0$. Set $A=H^{1/2}(H+1)^{-1/2}$. Then $(H+1)^{-1/2}=H^{-1/2}A+\Pi_0$, and so in the same way as in~(\ref{eq:BSP.KV-Hl-H1}) we have
\begin{equation*}
 (H+1)^{-1/2}V(H+1)^{-1/2}= A^*K_V A + R,
\end{equation*}
where $R$ is a finite rank operator. This implies that $(H+1)^{-1/2}V(H+1)^{-1/2}$ is in the weak Schatten class $\sL_{p,\infty}$. In particular, this is a compact operator. Likewise, the Birman-Schwinger operators $K_V(\lambda)$, $\lambda<0$, are  in $\sL_{p,\infty}$ as well. 

Bearing this in mind, set  $N=N^{-}(H_V)$ and $N_0=N^{-}(\Pi_0 V \Pi_0)$. As $\Pi_0V\Pi_0=\Pi_0H_V\Pi_0$, we see that $\sF^-(\Pi_0V\Pi_0;0)\subset \sF^-(H_V;0)$, and so it follows from~(\ref{eq:BSP.Glazman})  that $N_0\leq N$. 
 
It remains to show the 2nd inequality in~(\ref{eq:CLR.BorderlineBKP1}). We may assume $N-N_0\geq 1$, since otherwise the inequality is trivially satisfied. Note that $N=N(H_V;\lambda)$ as soon as $\lambda<0$ is small enough. The Birman-Schwinger principle~(\ref{eq:CLR.BKP}) then asserts that $N=N^+(K_V(\lambda);1)$. As $K_V(\lambda)\geq 0$, in the same way as in the proof of Corollary~\ref{cor:CLR.BKP-qinf} this ensures that $K_V(\lambda)$ has exactly $N$ singular values~$>1$, and hence $ \mu_{N-1}(K_V(\lambda))\geq 1$. 
  
By Lemma~\ref{lem:BSP.decomposition} we can find a compact operator $W:\sH_+\rightarrow \sH$ such that $V=-W^*W$. Set $T(\lambda)=W(H-\lambda)^{-\frac12}$. This is a compact operator on $\sH$. Moreover, by using (\ref{eq:CLR.H1sqrt-inverses}) and~(\ref{eq:BSP.W*}) we see that, for all $\xi,\eta \in \sH$, we have 
\begin{equation*}
 \bigscal{(H-\lambda)^{-\frac12} W^*\xi}{\eta}=\bigacou{ W^*\xi}{(H-\lambda)^{-\frac12}\eta} = \bigscal{\xi}{W(H-\lambda)^{-\frac12}\eta} =
 \scal{\xi}{T(\lambda)\eta}. 
\end{equation*}
This shows that $T(\lambda)^*= (H-\lambda)^{-1/2} W^*$. Thus, 
\begin{equation*}
 K_V(\lambda)= (H-\lambda)^{-\frac12} W^*W (H-\lambda)^{-\frac12}=T(\lambda)^*T(\lambda)=|T(\lambda)|^2. 
\end{equation*}
Set  $\tilde{K}_V(\lambda)= T(\lambda)T(\lambda)^*=W(H-\lambda)^{-1}W^*$. By using~(\ref{eq:Quantized.properties-mun1}) we get
\begin{equation}
\mu_{j}\left(K_V(\lambda)\right)=  \mu_{j}\left(T(\lambda)\right)^2=  \mu_{j}\left(T(\lambda)^*\right)^2= \mu_{j}\big(\tilde{K}_V(\lambda)\big), \quad j\geq 0. 
\label{eq:BSP.KV-tildeKV}
\end{equation}

Observe that $\tilde{K}_V(\lambda)=T(\lambda)(1-\Pi_0) T(\lambda)^*+T(\lambda)\Pi_0 T(\lambda)^*$.
 By using~(\ref{eq:Quantized.properties-mun2}) we get
\begin{equation*}
\mu_{N-1}\big(\tilde{K}_V(\lambda)\big) \leq \mu_{N-N_0-1}\big(T(\lambda)(1-\Pi_0) T(\lambda)^*\big)+\mu_{N_0} \big( T(\lambda)\Pi_0 T(\lambda)^*\big). 
\end{equation*}
As $\rk T(\lambda)\Pi_0 T(\lambda)^*\leq \rk \Pi_0= N_0$,  it follows from~(\ref{eq:min-max2}) that $\mu_{N_0}( T(\lambda)\Pi_0 T(\lambda)^*)=0$. 
Thus, 
\begin{equation}
 \mu_{N-1}\big(\tilde{K}_V(\lambda)\big)\leq \mu_{N-N_0-1}\big(T(\lambda)(1-\Pi_0) T(\lambda)^*\big). 
 \label{eq:CLR.muNN_0TlambdaW} 
\end{equation}
We also have
\begin{align*}
 T(\lambda)(1-\Pi_0)T(\lambda)^*&=W(H-\lambda)^{-1/2}(1-\Pi_0)(H-\lambda)^{-1/2}W^*,\\
 &= W(H+1)^{-1/2}\cdot (1-\Pi_0)(H+1)(H-\lambda)^{-1}\cdot(H+1)^{-1/2}W^*,\\
 &= T(-1)(1-\Pi_0)(H+1)(H-\lambda)^{-1}T(-1)^*. 
\end{align*}
As $(1-\Pi_0)(H+1)(H-\lambda)^{-1}\leq (H+1)H^{-1}$, it follows that
\begin{equation*}
 T(\lambda)(1-\Pi_0)T(\lambda)^*\leq T(-1) (H+1)H^{-1}T(-1)^*= WH^{-1}W^*. 
\end{equation*}
Therefore, by using~(\ref{eq:CLR.muNN_0TlambdaW}) and the monotonicity principle~(\ref{eq:monotonicity-principle})  we get 
\begin{equation}
 \mu_{N-1}\big(\tilde{K}_V(\lambda)\big)\leq  \mu_{N-N_0-1}\big(T(\lambda)(1-\Pi_0) T(\lambda)^*\big) \leq \mu_{N-N_0-1}\big(WH^{-1}W^*\big). 
\label{eq:CLR.muNTPi0}
\end{equation}

As in~(\ref{eq:BSP.KV-tildeKV}) we have $\mu_j(WH^{-1}W^*)=\mu_j(H^{-1/2}VH^{-1/2})=\mu_j(K_V)$ for all $j\geq 0$. Combining this with~(\ref{eq:CLR.muNTPi0}) and the fact that $\mu_{N-1}(\tilde{K}_V(\lambda))\geq 1$ gives
\begin{equation}
 1\leq  \mu_{N-1}\big(\tilde{K}_V(\lambda)\big) \leq \mu_{N-N_0-1}\big(K_V\big). 
\end{equation}
In the same way as in~(\ref{eq:CLR.BKP-qinf2}) this implies that 
\begin{equation*}
 N-N_0 \leq (N-N_0)\mu_{N-N_0-1}\big(K_V\big)^{p} \leq  
 \big\|K_V\big\|_{\sL_{p,\infty}}^p.
\end{equation*}
This gives the 2nd inequality in~(\ref{eq:CLR.BorderlineBKP1}). The proof is complete.
\end{proof}

\begin{remark}
 By using~(\ref{eq:BSP.borderline-NN}) and arguing along similar lines as that of the proof of Corollary~\ref{cor:CLR.BKP-qinf} in Appendix~\ref{app:BSP} it is easier to get the inequality, 
  \begin{equation}
 N^{-}(H_{V})- \dim \ker H \leq  \big\|K_V\big\|_{\sL_{p,\infty}}^p.
 \label{eq:CLR.BBSP-Birman-Solomyak}
\end{equation}
We also recover this inequality from~(\ref{eq:CLR.BorderlineBKP1}). This inequality and the inequalities~(\ref{eq:BSP.borderline-NN}) are known already. They appear in various forms in the work of Birman and Solomyak, in particular for dealing with Schr\"odinger operators on bounded domains under Neuman's boundary condition (see, e.g., \cite{BS:JFAA70, BS:AMST80, So:PLMS95}). In most of the instances where it has been used by Birman and Solomyak, the inequality~(\ref{eq:CLR.BBSP-Birman-Solomyak}) is relevant, since in those situations $\ker H=1$ and $\Pi_0V\Pi_0$ has a negative eigenvalue is $V\leq 0$ and $V\neq 0$. However, in general we might have $N^{-}(H_{V})- \dim \ker H <0$ (e.g., if $V(\ker H)\subset \ran H$), and so in such a case~(\ref{eq:CLR.BBSP-Birman-Solomyak}) is not sharp (see Example~\ref{ex:BBSP.closed-Riemannian} below). 
\end{remark}

\begin{example}\label{ex:BBSP.closed-Riemannian} 
 Let $(M^n,g)$ be a closed Riemannian manifold, and take $H$ to be the (positive) Laplacian $\Delta_g$ on $\sH=L_2(M;g)$ with domain the Sobolev space $W_2^2(M)$. In this case, $\ker \Delta_g$ is spanned by the characteristic functions of the connected components of $M$, and so $\dim \ker \Delta_g=\dim H^0(M)$, where $H^0(M)$ is the zero-th degree de Rham cohomology space of $M$. Let $V\in L_\infty(M)$. Then, $\Delta_g^{-1/2}V\Delta_g^{-1/2}\in \sL_{n/2,\infty}$, and so from~(\ref{eq:CLR.BBSP-Birman-Solomyak}) we get
 \begin{equation}
 N^{-}(\Delta_g+V)- \dim H^0(M) \leq  \big\|\Delta_g^{-\frac{1}2}V\Delta_g^{-\frac{1}{2}}\big\|_{\sL_{\frac{n}{2},\infty}}^{\frac{n}2}.
 \label{eq:CLR.BorderlineBKP-Birman-Solomyak-Deltag}
\end{equation}

Suppose now that $M$ has at least two connected components, and let $M_1$ be such a component. If $V=-\alpha\car_{M_1}$ with $\alpha>0$,  then $\Pi_0V\Pi_0=-\alpha\Pi_1$, where $\Pi_1$ is the orthogonal projection onto $\car_{M_1}$, and hence $ N^{-}(\Pi_0 V \Pi_0)=1$. Furthermore, if $\alpha$ is small enough, then
\begin{equation*}
 N^{-}(\Delta_g-\alpha\car_{M_1})= N^{-}(\Delta_{g|M_1}-\alpha)+N^{-}(\Delta_{g|M\setminus M_1})=N^{-}(\Delta_{g|M_1};\alpha)=1<\dim H^0(M). 
\end{equation*}
Thus, in this case the l.h.s.~of~(\ref{eq:CLR.BorderlineBKP-Birman-Solomyak-Deltag}) is~$<0$. 
\end{example}

\section{Cwikel-Lieb-Rozenblum Inequalities and Lieb-Thirring Inequalities}\label{CLR_section}
In this section, we combine the result of the previous two sections to establish Cwikel-Lieb-Rozenblum (CLR) inequalities and Lieb-Thirring (LT) inequalities for (fractional) Schr\"odinger operators on NC tori. As we shall see, this will lead us to a Sobolev inequality for NC tori. 

\subsection{CLR inequalities} 
In the notation of the previous section we let $\sH=L_2(\T^n_\theta)$ and let $H$ be the fractional Laplacian $\Delta^{n/2p}$, $p>0$, with domain $W_2^{n/2p}(\T^n_\theta)$. Note that $\Delta^{n/2p}$ has compact resolvent and its nullspace is $\C\cdot1$. In particular, $0$ is an isolated eigenvalue. 
Up to equivalent norms, the Hilbert spaces $\sH_+$ and $\sH_{-}$ are the Sobolev spaces $W_2^{n/4p}(\T^n_\theta)$ and $W_2^{-n/4p}(\T^n_\theta)$, respectively 

Suppose that, either $s>n/2q$ and $q\geq 1$, or $s=n/2q$ and $q\geq1$. In this case, if $V\in L_q(\T^n_\theta)$, then we know by Lemma~\ref{lem:Boundedness.sandwich1} that $\lambda(V)$ induces a bounded operator $\lambda(V): W^s_2(\T^n_\theta)\rightarrow  W^{-s}_2(\T^n_\theta)$. Thus, it defines a quadratic form on $W_2^s(\T^n_\theta)$ by
\begin{equation*}
 Q_{\lambda(V)}(u,v)=\acou{\lambda(V)u}{v}=\tau\big[v^*Vu\big], \qquad u,v\in W_2^{s}(\T^n_\theta). 
\end{equation*}
In particular, if $V^*=V$, then 
\begin{equation*}
 Q_{\lambda(V)}(u,v) =\overline{\tau\big[u^*Vv\big]}=\overline{Q_{\lambda(V)}(v,u)}. 
\end{equation*}
That is, $Q_{\lambda(V)}$ is symmetric. If in addition $V\geq 0$, then
\begin{equation*}
 Q_{\lambda(V)}(u,u)= \tau\big[u^*Vu]=\tau\big[(V^{1/2}u)^*V^{1/2}u\big]\geq 0. 
\end{equation*}
Thus, $Q_{\lambda(V)}\geq 0$. In particular, in the notation of Section~\ref{sec:Birman-Schwinger} we have 
\begin{equation}
 V_1\leq V_2 \Longrightarrow \lambda(V_1)\leq \lambda(V_2). 
 \label{eq:CLR.monotonicity-lambda}
\end{equation}

Bearing this in mind, Theorem~\ref{thm:Specific-Cwikel.sandwiched} implies that the operator $\lambda(V):W_2^{n/2p}(\T^n_\theta) \rightarrow W_2^{-n/2p}(\T^n_\theta)$ is compact under any of the following conditions:
\begin{enumerate}
 \item[(i)] $p>1$ and $V\in L_p(\T^n_\theta)$. 
 
 \item[(ii)] $p=1$ and $V\in L_{q}(\T^n_\theta)$, $q >1$. 
 
 \item[(iii)] $p<1$ and $V\in L_1(\T^n_\theta)$.
\end{enumerate}
It follows that, if $V$ is selfadjoint and any of the above conditions (i)--(iii) holds, then the pair $(\Delta^{n/2p}, \lambda(V))$ fits into the framework of the previous section. Thus, we may define the fractional Schr\"odinger operator $\Delta^{\frac{n}{2p}} +\lambda(V)$ as a form sum. 
We obtain a bounded from below selfadjoint operator on $L_2(\T^n_\theta)$ with compact resolvent. In particular, it has no essential spectrum. As above we denote by $N^{-}(\Delta^{n/2p} +\lambda(V))$ the number of negative eigenvalues counted with multiplicity. 

In what follows, given any selfadjoint element $V\in L_p(\T^n_\theta)$, we let $V_{\pm}=\frac{1}{2}(|V|\pm V)$ be the positive and negative parts of $V$. Note that $V_+$ and $V_{-}$ are positive elements of $L_p(\T^n_\theta)$.

We are now in a position to establish CLR inequalities on NC tori in the following form. 

\begin{theorem}[CLR Inequalities on NC Tori]\label{thm:CLR.CLR-NCtori}
Let $n\geq 2$.  The following holds. 
\begin{enumerate}
 \item[(i)] If $p>1$ and $V=V^*\in L_p(\T^n_\theta)$, then 
\begin{equation}
\label{eq:CLR.sup-critical}
 N^{-}\big(\Delta^{\frac{n}{2p}}+\lambda(V)\big)-1 \leq c_{+}(2p)^{2p}\nu_0(n)s \tau\big[|V_{-}|^p\big]. 
\end{equation}

\item[(ii)] If $V=V^*\in L_p(\T^n_\theta)$, $p>1$, then 
\begin{equation}
 N^{-}\big(\Delta^{\frac{n}2}+\lambda(V)\big)-1 \leq c_2(2p)^{2p}\nu_0(n) \tau\big[|V_{-}|^p\big]^{\frac1{p}}. 
 \label{eq:CLR.critical} 
\end{equation}

\item[(iii)] If $p<1$ and $V=V^*\in L_1(\T^n_\theta)$, then
\begin{equation}
 N^{-}\big(\Delta^{\frac{n}{2p}}+\lambda(V)\big)-1 \leq c_{-}(2p)^{2p}\nu_0(n)\tau\big[|V_{-}|\big]^{p}.
 \label{eq:CLR.sub-critical} 
\end{equation}
\end{enumerate}
Here $c_{+}(2p)$ and $c_{-}(2p)$ are the best constants in the $\sL_{p,\infty}$-Cwikel estimates~(\ref{eq:Cwikel2+Lpinf}) and~(\ref{eq:Cwikel2-Lpinf}), respectively, 
 and $c_2(2p)$ is given by~(\ref{eq:Cwikel-speficied-c2p}). 
\end{theorem}
\begin{proof}
Suppose that $q=\max(p,1)$ if $p\neq1$, or $q>1$ if $p=1$. Set $c(p,q)$ to be equal to $c_{+}(2p)$ (resp., $c_{-}(2p)$, $c_2(2q)$) if $p>1$ (resp., $p<1$, $p=1$).   The proof amounts to show that if $V=V^*\in L_q(\T^n_\theta)$, then 
\begin{equation}
  N^{-}\big(\Delta^{\frac{n}{2p}}+\lambda(V)\big)-1 \leq c(p,q)^{2p}\nu_0(n) \tau\big[|V_{-}|^q\big]^{\frac{p}{q}}.
  \label{eq:CLR.CLR-compressed} 
\end{equation}
Note also that if $V_{-}=0$, then $V=V_{+}\geq 0$ and $Q_{\lambda(V)}\geq 0$, and hence $\Delta^{n/2p}+\lambda(V)\geq 0$. Thus, in this case $N^{-}(\Delta^{n/2p}+\lambda(V))=0$, and so the inequality~(\ref{eq:CLR.CLR-compressed}) is trivially satisfied. Therefore, we may assume $V_{-}\neq 0$. 

As $V=V_+-V_{-}\geq -V_{-}$, by~(\ref{eq:CLR.monotonicity-lambda}) we have $\lambda(V)\geq -\lambda(V_{-})$, and so the monotonicity principle~(\ref{eq:CLR.mono-principle-N-}) implies that
\begin{equation}
 N^{-}\big(\Delta^{\frac{n}{2p}}+\lambda(V)\big) \leq N^{-}\big(\Delta^{\frac{n}{2p}}-\lambda(V_{-})\big).
  \label{eq:CLR.NHV-NHV-}  
\end{equation}
 Let $\Pi_0$ be the orthogonal projection onto $\ker \Delta^{n/2p}=\C\cdot 1$.  It is a rank 1 projection given by 
\begin{equation*}
 \Pi_0u=\tau(u)1, \qquad u\in L_2(\T^n_\theta). 
\end{equation*}
The formula continues to make sense for any $u\in L_1(\T^n_\theta)$. We have 
\begin{equation*}
 -\Pi_0\lambda(V_{-}) \Pi_0 1= -\tau(1)\Pi_0V_{-}= -\tau(V_{-})<0. 
\end{equation*}
Thus, $-\tau(V_{-})$ is a negative eigenvalue of $-\Pi_0\lambda(V_{-}) \Pi_0$. As $-\Pi_0\lambda(V_{-}) \Pi_0$ has rank~1 this is the unique negative eigenvalue, and hence $N^{-}(-\Pi_0\lambda(V_{-}) \Pi_0)=1$. Combining this with~(\ref{eq:CLR.NHV-NHV-}) gives
\begin{equation}
 N^{-}\big(\Delta^{\frac{n}{2p}}+\lambda(V)\big)-1 \leq N^{-}\big(\Delta^{\frac{n}{2p}}-\lambda(V_{-})\big)-N^{-}(-\Pi_0\lambda(V_{-}) \Pi_0).
  \label{eq:CLR.NHV-NHV-1}
\end{equation}

Moreover, as $0\leq V_{-}\in L_q(\T^n_\theta)$, Theorem~\ref{thm:Specific-Cwikel.sandwiched}  and Remark~\ref{rmk:Specific.positive-case} ensure us that $\Delta^{-n/4p}\lambda(V_{-})\Delta^{-n/4p}\in \sL_{p,\infty}$ with norm estimate, 
\begin{equation}
 \big\| \Delta^{-\frac{n}{4p}}\lambda(V)\Delta^{-\frac{n}{4p}}\big\|_{\sL_{p,\infty}} \leq 
 c(p,q)^2\nu_0(n)^{\frac1{p}} \|V_{-}\|_{L_{q}}= c(p,q)^2\nu_0(n)^{\frac1{p}} \tau\big[|V_{-}|^q\big]^{\frac{1}{q}}.
 \label{eq:CLR.Cwikel} 
\end{equation}
This allows us to apply the borderline Birman-Schwinger principle (Theorem~\ref{thm:Borderline-BSP}) to get
\begin{equation*}
 N^{-}\big(\Delta^{\frac{n}{2p}}-\lambda(V_{-})\big)-N^{-}(-\Pi_0\lambda(V_{-}) \Pi_0)\leq   \big\| \Delta^{-\frac{n}{4p}}\lambda(V_{-})\Delta^{-\frac{n}{4p}}\big\|_{\sL_{p,\infty}}^p. 
\end{equation*}
Combining this with~(\ref{eq:CLR.NHV-NHV-1}) and~(\ref{eq:CLR.Cwikel}) gives the inequality~(\ref{eq:CLR.CLR-compressed}). 
The proof is complete.  
\end{proof}

\begin{remark}
 For the ordinary torus $\T^n$, i.e., $\theta =0$, the critical CLR inequality~(\ref{eq:CLR.critical}) actually hold for potentials $V$ in the Orlicz class $\LLogL(\T^n)$ 
 (see~\cite{Po:Weyl-Orlicz, Ro:arXiv21, RS:EMS21}).
\end{remark}

\begin{remark}
 Suppose that either $q=\max(p,1)$ and $p\neq 1$,  or $q>1=p$. The CLR inequalities~(\ref{eq:CLR.sup-critical})--(\ref{eq:CLR.sub-critical}) are consistent with Lieb's version of the CLR inequality for Schr\"odinger operators on closed manifolds (see~\cite{Li:BAMS76, Li:1980}). Note that we cannot expect estimates of the form,
\begin{equation}
 N^{-}\big(\Delta^{\frac{n}{2p}}+\lambda(V)\big)  \leq c_{npq} \tau\big[|V_{-}|^q\big]^{\frac{p}{q}}.
 \label{eq:CLR.direct-CLR} 
\end{equation}
To see this take $V=-\epsilon$ with $\epsilon>0$. In that case the right-hand side of~(\ref{eq:CLR.direct-CLR}) is $c_{npq}\tau[\epsilon^q]^{p/q}=c_{npq}\epsilon^p$. In addition $(\Delta^{{n}/{2p}}-\epsilon)1=-\epsilon 1$,  and hence $-\epsilon$ is always a negative eigenvalue of $\Delta^{{n}/{2p}}-\epsilon$. Thus, $N^{-}(\Delta^{{n}/{2p}}-\epsilon)\geq 1$. Therefore, if we had the inequality~(\ref{eq:CLR.direct-CLR}), then we would have
\begin{equation*}
 1\leq N^{-}\big(\Delta^{\frac{n}{2p}}-\epsilon\big) \leq c_{npq} \epsilon^{p} \qquad \forall \epsilon>0,
\end{equation*}
which does not hold as soon as $\epsilon$ is small enough. 
\end{remark}

\begin{remark}
 The most interesting case of the above CLR inequalities is for the Schr\"odinger operator $\Delta+\lambda(V)$, i.e., $p=n/2$. In this case Theorem~\ref{thm:CLR.CLR-NCtori} specializes to the following:
\begin{enumerate}
 \item[(i)] If $n\geq 3$ and $V=V^*\in L_{n/2}(\T^n_\theta)$, then
 \begin{equation}
 N^{-}\big(\Delta+\lambda(V)\big)-1 \leq c_{+}(n)^n \nu_0(n) \tau\big[|V_{-}|^{\frac{n}{2}}\big].
 \label{eq:CLR.Schrodingern3} 
\end{equation}

\item[(ii)] If $n=2$ and $V=V^*\in L_{p}(\T^2_\theta)$, $p>1$, then 
 \begin{equation}
 N^{-}\big(\Delta+\lambda(V)\big)-1 \leq c_2(2p)^{2p} \nu_0(2) \tau\big[|V_{-}|^{p}\big]^{\frac1{p}}. 
 \label{eq:CLR.Schrodingern2} 
\end{equation}
\end{enumerate}
\end{remark}

\begin{remark}
 The 2-dimensional case~(\ref{eq:CLR.Schrodingern2}) shows a stark contrast with the Euclidean space case, since on $\R^2$ the Birman-Schwinger principle and the CLR inequality for Schr\"odinger operators $\Delta+V$ do not hold. In particular, recent results of Hoang~\emph{et al.}~\cite[Theorem~3.1]{HHRV:arXiv16} show that if $V\in L_1(\R^2)\setminus 0$ is such that $\int V(x)dx\leq 0$, then $\Delta +V$ always has at least one negative eigenvalue (see also Simon~\cite{Si:AP76}).  
\end{remark}

\begin{remark}
 The cases $p\neq n/2$ are also of interest, since fractional Schr\"odinger operators naturally appear in the framework of fractional quantum mechanics~\cite{La:PL00}. For $p=n$ we get the hyper-relativistic Schr\"odinger operator $|\nabla|+V$ (see, e.g., \cite{Da:CMP83}). 
\end{remark}

\subsection{Semiclassical Weyl's laws} 
Although the inequality~(\ref{eq:CLR.direct-CLR}) does not hold, it holds \emph{semiclassically}. Namely, we have the following result.  

\begin{corollary}[Semiclassical CLR Inequality] \label{cor:CLR.CLR-NCtori-semiclassical}
 Assume that, either $q=\max(p,1)$ with $p\neq 1$, or $q>p=1$. Let $V =V^*\in L_{q}(\T^n_\theta)$. Then, as $h\rightarrow 0^+$ we have
 \begin{equation}
 N^{-}\big(h^{\frac{n}{p}}\Delta^{\frac{n}{2p}}+\lambda(V) \big) \leq c(p,q)^{2p}\nu_0(n) h^{-n} 
 \tau\big[|V_{-}|^q\big]^{\frac{p}{q}}+\op{O}(1).
  \label{eq:CLR.semi-classical-CLR} 
\end{equation}
Here, $c(p,q)$ is the same as in the proof of Theorem \ref{thm:CLR.CLR-NCtori}.
\end{corollary}
\begin{proof}
 As $h^{n/p}\Delta^{n/2p}+\lambda(V)=h^{n/p}(\Delta+\lambda(h^{-n/p}V))$, the operators $h^{n/p}\Delta^{n/2p}+\lambda(V)$ and $\Delta+\lambda(h^{-n/p}V)$ have the same number of negative eigenvalues counted with multiplicity. Therefore, the CLR inequality~(\ref{eq:CLR.CLR-compressed}) for $h^{-n/p}V$ gives
\begin{align*}
 N^{-}\big(h^{\frac{n}{p}}\Delta^{\frac{n}{2p}}+\lambda(V) \big)=N^{-}\big(\Delta^{\frac{n}{2p}}+\lambda\big(h^{-\frac{n}{p}}V\big) \big)& \leq 
 c(p,q)^{2p}\nu_0(n) \tau\big[\big|h^{-\frac{n}{p}}V_{-}\big|^{q}\big]^{\frac{p}{q}}+1\\
 & \leq  c(p,q)^{2p}\nu_0(n) h^{-n} 
 \tau\big[|V_{-}|^q\big]^{\frac{p}{q}}+\op{O}(1).
\end{align*}
The proof is complete. 
\end{proof}

The semiclassical CLR inequality on $\R^n$, $n\geq 3$, is the main tool to extend the semi-classical Weyl's law for Schr\"odinger operators $h^2\Delta+V$ with smooth potentials $V$ to Schr\"odinger operators with potentials in $L_{n/2}(\R^n)$ (see~\cite[Proposition 5.2]{Si:TAMS76}). This result can be traced back to the work of Birman and his collaborators in the late 60s and early 70s (see~\cite{BS:JFAA70, BS:AMST80}). 
Likewise, we can extend the semiclassical Weyl's law for fractional Schr\"odinger operators $h^{\frac{n}{p}}\Delta^{\frac{n}{2p}}+V$ for smooth potentials to potentials in a suitable $L_q$-class (see~\cite{So:PLMS95} and the references therein). 

Similarly, the semiclassical CLR inequality~(\ref{eq:CLR.semi-classical-CLR}) for NC tori would enable us to extend the semiclassical Weyl law on NC tori for smooth potentials to non-smooth potentials. However, to date the semiclassical Weyl law for smooth potentials has still not been established. We believe it can be established in a similar fashion as in the Euclidean case by setting up a semi-classical pseudodifferential calculus on NC tori. However, this falls out of the scope of this paper. Therefore, we only state the semiclassical Weyl's law on NC tori as a conjecture. 

\begin{conjecture}[Semiclassical Weyl Law on NC Tori]\label{CLR.Weyl-law} 
Let $n\geq 2$. Suppose that, either $p\neq 1$ and $q=\max(p,1)$, or $p=1<q$. If $V=V^*  \in L_{q}(\T^n_\theta)$, then as $h\rightarrow 0^+$ we have 
         \begin{equation}\label{semiclassical_weyl_law}
         N^{-}\big(h^{\frac{n}{p}}\Delta^{\frac{n}{2p}}+\lambda(V) \big)= c(n)h^{-n}
            \tau\big[|V_{-}|^{p}\big] +\op{o}\big(h^{-n}\big), 
            \quad \text{where}\ c(n) =  |\mathbb{B}^{n}|.
         \end{equation}   
In particular, if $n\geq 3$ and $V=V^*\in L_{n/2}(\T^n_\theta)$, or if $n=2$ and $V=V^*\in L_q(\T^n_\theta)$ with $q>1$, then 
         \begin{equation*}
         N^{-}\big(h^{2}\Delta+\lambda(V) \big)= c_{n}h^{-n}
            \tau\big[|V_{-}|^{\frac{n}{2}}\big] +\op{o}\big(h^{-n}\big). 
         \end{equation*}   
\end{conjecture}

\begin{remark}
 In the critical case $p=1$ we may hope that the semiclassical Weyl's law~(\ref{semiclassical_weyl_law}) further hold for potentials in the NC Orlicz class $\LLogL(\T^n_\theta)$. Establishing this would require extending the Cwikel estimate~(\ref{eq:specific.sando-Cwikel-1infty}) and the CLR inequality~(\ref{eq:CLR.critical}) to this class of potentials. However, at this point it is still unclear how to do this.  
\end{remark}

\begin{remark}
 An alternative route to establish the semiclassical Weyl's law~(\ref{semiclassical_weyl_law}) for $L_q$-potentials is outlined in the sequel article~\cite{MP:Part2}. It is based on the (borderline) Birman-Schwinger principle~(\ref{eq:BSP.borderline-NN}) and a (conjectural) version for NC tori of the celebrated Weyl's law for negative order pseudodifferential operators of Birman-Solomyak~\cite{BS:VLU77, BS:VLU79, BS:SMJ79} (see also~\cite{Po:NCIntegration}).  Note that the conjecture in~\cite{MP:Part2} extends Conjecture~\ref{CLR.Weyl-law} to \emph{curved} NC tori. 
\end{remark}

\begin{remark}
After earlier versions of this article and the sequel~\cite{MP:Part2} were posted on arXiv, the semiclassical Weyl's law~(\ref{semiclassical_weyl_law}) was established for $n\geq 3$ and $p=n/2$ by the first named author in a joint preprint with F.\ Sukochev and D.\ Zanin (see~\cite{MSZ:arXiv21}). The cases $n=2$ or $p\neq n/2$, as well as the curved version conjectured in~\cite{MP:Part2}, still remain open to date.  
\end{remark}


\subsection{Lieb-Thirring Inequalities} 
We may also consider the action of the Schr\"odinger operators $\Delta^{n/2p}+\lambda(V)$ on the orthogonal complement of the nullspace $\ker \Delta=\C\cdot 1$. 
Set $\dot{L}_2(\T^n_\theta)=(\C \cdot 1)^\perp$ and $\dot{W}_2^s(\T^n_\theta)={W}_2^s(\T^n_\theta)\cap \dot{L}_2(\T^n_\theta)$. The orthogonal projection 
$\Pi_0:L_2(\T^n_\theta)\rightarrow L_2(\T^n_\theta)$ onto $\C\cdot 1$ is given by 
\begin{equation*}
 \Pi_0u = \tau(u)1, \qquad u\in L_2(\T^n_\theta). 
\end{equation*}
Thus, $\dot{L}_2(\T^n_\theta)$ consists of all $u\in L_2(\T^n_\theta)$ with zero mean value $\tau(u)=0$. If $s>0$,  the $\dot{W}_2^s(\T^n_\theta)$ is a closed subspace of ${W}_2^s(\T^n_\theta)$ and $\Pi_0$ induces a selfadjoint bounded projection on $W_2^s(\T^n_\theta)$. By duality it extends to a bounded projection on the anti-linear dual $W_2^{-s}(\T^n_\theta)$. 

Let $\dot{\Delta}$ be the restriction of $\Delta$ to $\dot{L}_2(\T^n_\theta)$. Its domain is $\dot{W}_2^s(\T^n_\theta)$. This is a selfadjoint operator with same spectrum as $\Delta$ at the exception of the $0$-eigenvalue. Moreover, for any $s\in \R$, the power $\dot{\Delta}^{s/2}$ agrees with the operator $\Delta^{s/2}$ on 
$\dot{W}_2^{s}(\T^n_\theta)$. 

Suppose that $p>0$ and either $q=\max(p,1)$ if $p\neq 1$ or $q>1$ if $p=1$. For $V\in L_q(\T^n_\theta)$, we let $\dot{\lambda}(V): 
\dot{W}_2^{n/2p}(\T^n_\theta) \rightarrow \dot{W}_2^{-n/2p}(\T^n_\theta)$ the operator defined by 
\begin{equation*}
 \bigacou{\dot{\lambda}(V)u}{v}=\acou{\lambda(V)u}{v}, \qquad u,v\in \dot{W}_2^{\frac{n}{2p}}(\T^n_\theta). 
\end{equation*}
In other words, the corresponding quadratic form $Q_{\dot{\lambda}(V) }$ is just the restriction of $Q_{\lambda(V)}$ to $\dot{W}_2^{n/2p}(\T^n_\theta)$. In particular, if $V^*=V$, then the  quadratic form $Q_{\dot{\lambda}(V) }$ is symmetric,s and if $V_1\leq V_2$, then $\dot{\lambda}(V_1)\leq \dot{\lambda}(V_2)$ in the sense used in Section~\ref{sec:Birman-Schwinger}.

With respect to the orthogonal splitting $L_2(\T^n_\theta)=\dot{L}(\T^n_\theta)\oplus (\C\cdot 1)$ we have 
\begin{equation}\label{eq:LT.dotDeltaV}
\Delta^{-\frac{n}{4p}}  \lambda(V)\Delta^{-\frac{n}{4p}} = 
\begin{pmatrix}
 \dot{\Delta}^{-\frac{n}{4p}}  \dot{\lambda}(V) \dot{\Delta}^{-\frac{n}{4p}} & 0\\ 0 &0
\end{pmatrix}. 
\end{equation}
It follows that the operator $ \dot{\Delta}^{-{n}/{4p}}\dot{\lambda}(V) \dot{\Delta}^{-{n}/{4p}}$ is in the weak Schatten class $\sL_{p,\infty}$ and has same $\sL_{p,\infty}$-norm as $\Delta^{-{n}/{4p}}  \lambda(V)\Delta^{-{n}/{4p}}$. Therefore, if $V^*=V$, then we may define the operator $\dot{\Delta}^{n/2p}+\dot{\lambda}(V)$ as a form sum on $\dot{L}_2(\T^n_\theta)$ with domain $\dot{W}_2^{n/p}(\T^n_\theta)$. Equivalently, this is the selfadjoint operator whose quadratic form is the restriction to $\dot{W}_2^{n/2p}(\T^n_\theta)$ of $Q_{\Delta^{n/2p}}+Q_{\lambda(V)}$. As above we obtain a bounded from below selfadjoint operator with compact resolvent. 

\begin{theorem}[CLR Inequalities on NC Tori; 2nd Version]\label{thm:CLR.CLR-NCtori2}
Let $n\geq 2$.  Keeping on  using the notation of Theorem~\ref{thm:CLR.CLR-NCtori}, the following hold. 
\begin{enumerate}
 \item If $p>1$ and $V=V^*\in L_p(\T^n_\theta)$, then 
\begin{equation}
\label{eq:CLR.sup-critical2}
 N^{-}\big(\dot{\Delta}^{\frac{n}{2p}}+\dot{\lambda}(V)\big) \leq c_{+}(2p)^{2p}\nu_0(n) \tau\big[|V_{-}|^p\big]. 
\end{equation}

\item If $V=V^*\in L_p(\T^n_\theta)$, $p>1$, then 
\begin{equation}
 N^{-}\big(\dot{\Delta}^{\frac{n}2}+\dot{\lambda}(V)\big) \leq c_2(2p)^{2p}\nu_0(n) \tau\big[|V_{-}|^p\big]^{\frac1{p}}. 
 \label{eq:CLR.critical2} 
\end{equation}

\item If $p<1$ and $V=V^*\in L_1(\T^n_\theta)$, then
\begin{equation}
 N^{-}\big(\dot{\Delta}^{\frac{n}{2p}}+\dot{\lambda}(V)\big) \leq c_{-}(2p)^{2p}\nu_0(n)\tau\big[|V_{-}|\big]^{p}.
 \label{eq:CLR.sub-critical2} 
\end{equation}
\end{enumerate}
\end{theorem}
\begin{proof}
 The proof follows the same outline as that of the proof of Theorem~\ref{thm:CLR.CLR-NCtori}. Suppose that $q=\max(p,1)$ if $p\neq1$, or $q>1$ if $p=1$, and let $V=V^*\in L_q(\T^n_\theta)$. As $-\dot{\lambda}(V_{-})\leq \dot{\lambda}(V)$ in the same way as in~(\ref{eq:CLR.NHV-NHV-}) we have
\begin{equation*}
 N^{-}\big(\dot{\Delta}^{\frac{n}{2p}}+\dot\lambda(V)\big) \leq N^{-}\big(\dot\Delta^{\frac{n}{2p}}-\dot\lambda(V_{-})\big).
\end{equation*}
Here $\dot{\Delta}$ is a selfadjoint operator with compact resolvent and positive spectrum. Therefore, we may apply Corollary~\ref{cor:CLR.BKP-qinf} and use~(\ref{eq:LT.dotDeltaV}) to get 
\begin{equation*}
 N^{-}\big(\dot\Delta^{\frac{n}{2p}}-\dot\lambda(V_{-})\big) \leq \big\|\dot{\Delta}^{-\frac{n}{4p}}  \dot{\lambda}(V) \dot{\Delta}^{-\frac{n}{4p}}\big\|_{\sL_{p,\infty}}^p 
 =  \big\|{\Delta}^{-\frac{n}{4p}} {\lambda}(V) {\Delta}^{-\frac{n}{4p}}\big\|_{\sL_{p,\infty}}^p.
\end{equation*}
In the same way as in the proof of Theorem~\ref{thm:CLR.CLR-NCtori}, combining the above inequalities with the Cwikel estimates provided by Theorem~\ref{thm:Specific-Cwikel.sandwiched} yields the result. 
\end{proof}

\begin{remark}
 We can recover the inequalities~(\ref{eq:CLR.sup-critical})--(\ref{eq:CLR.sub-critical}) from the inequalities~(\ref{eq:CLR.sup-critical2})--(\ref{eq:CLR.sub-critical2}), since in this setup~(\ref{eq:BSP.borderline-NNsH1}) gives 
\begin{equation*}
 N^{-}\big({\Delta}^{\frac{n}{2p}}+{\lambda}(V)\big) \leq N^{-}\big(\dot{\Delta}^{\frac{n}{2p}}+\dot{\lambda}(V)\big) + 1. 
\end{equation*}
\end{remark}

In what follows given any selfadjoint operator $A$ on a Hilbert space $\sH$  which is bounded from below operator and has discrete negative spectrum we arrange its negative eigenvalues as a non-decreasing sequence, 
\begin{equation*}
 \nlambda_0(A) \leq \nlambda_1(A)\leq \cdots \leq 0.  
\end{equation*}
Here each eigenvalue is repeated according to multiplicity and we make the convention that $\nlambda_j(A)=0$ for $j\geq N^{-}(A)$. In other words, $\nlambda_j(A)=-\mu_j(A_{-})$, where $A_{-}=\frac12(|A|-A)$ is the negative part of $A$. 

It is well known that the CLR inequality implies the Lieb-Thirring inequalities~\cite{LT:PRL75, LT:SMP76} for the Riesz means of negative eigenvalues of Schr\"odinger operators (see, e.g., \cite[Proposition~6.17]{Si:AMS15a}), although this does not lead to the best bounds for the best constants of these inequalities. Likewise, as a consequence of the CLR inequalities provided by Theorem~\ref{thm:CLR.CLR-NCtori2} we shall obtain the following version of the Lieb-Thirring inequalities for NC tori. 

\begin{theorem}[Lieb-Thirring Inequalities on NC Tori]\label{thm:LT-inequality}
Let $\gamma>0$ and $p>1$. There is a constant $L_{p,\gamma,n}>0$ such that, for all $V=V^*\in L_{p+\gamma}(\T^n_\theta)$, we have
 \begin{equation}\label{eq:LT-Inequality}
 \sum_{j\geq 0} \left|\nlambda_j\big(\dot{\Delta}^{\frac{n}{2p}}+\dot{\lambda}(V)\big)\right|^{\gamma} \leq L_{p,\gamma, n} \tau\big[|V_{-}|^{p+\gamma}\big]. 
\end{equation}
Moreover, the best constant $L_{p,\gamma,n}$ is such that
\begin{equation*}
 L_{p,\gamma,n} \leq \gamma \frac{\Gamma(p+1)\Gamma(\gamma)}{\Gamma(p+\gamma+1)}c_{+}(2p)^{2p}\nu_0(n). 
\end{equation*}
\end{theorem}
\begin{proof}
Given any $V^*=V\in L_{p+\gamma}(\T^n_\theta)$, set $\dot{H}_V=\dot{\Delta}^{\frac{n}{2p}}+\dot{\lambda}(V)$. We have the usual formula (see, e.g., 
\cite{LS:CUP10, Si:AMS15a}), 
\begin{equation*}
 \sum_{j\geq 0} \left|\nlambda_j\big(\dot{H}_V\big)\right|^\gamma = \int_{-\infty}^0 (-t)^\gamma dN^{-}\big(\dot{H}_V;t\big)= \gamma \int_0^\infty t^{\gamma-1}N^{-}\big(\dot{H}_V;-t\big)dt.   
\end{equation*}
By the CLR inequality~\eqref{eq:CLR.sup-critical2} we have 
\begin{equation*}
 N^{-}\big(\dot{H}_V;-t\big) = N^{-}\big(\dot{H}_{V+t}\big) \leq c_+(2p)^{2p}\nu_0(n)\tau \big[(V+t)_{-}^p\big].  
\end{equation*}
Thus, 
\begin{equation}
 \sum_j \left|\nlambda_j\big(\dot{H}_V\big)\right|^\gamma \leq 
 \gamma c_+(2p)^{2p}\nu_0(n) \int_0^\infty t^{\gamma-1}\tau \big[|(V+t)_{-}|^p\big]dt.
 \label{eq:LT.Riesz-mean-tau} 
\end{equation}

Let $E_v=\car_{(-\infty,v]}(-V)$, $v\in \R$, be the spectral measure of $-V$. The spectral measure of $(V+t)_{-}=(-V-t)_+$, $t\geq 0$, is  
$\car_{[0,\infty)}(v)E_{v+t}$, and hence 
\begin{equation*}
 \tau \big[(V+t)_{-}^p\big]= \int_0^\infty v^p d(\tau_*E_{v+t}) = \int_{t}^\infty (v-t)^p d(\tau_*E_v).  
\end{equation*}
Therefore, we have 
\begin{align*}
  \int_0^\infty t^{\gamma-1}\tau \big[(V+t)_{-}^p\big] dt& = \int_0^\infty \bigg(\int_{t}^\infty t^{\gamma-1} (v-t)^p d(\tau_*E_v)\bigg) dt\\
  & = \int_0^\infty \bigg(\int_{0}^\lambda  t^{\gamma-1} (v-t)^p dt\bigg) d(\tau_*E_v). 
\end{align*}
Note that
\begin{equation*}
 \int_{0}^\lambda  t^{\gamma-1} (v-t)^p dt =v^{p+\gamma} \int_0^1t^\gamma(1-t)^{p}dt=v^{p+\gamma} B(\gamma,p+1), 
\end{equation*}
where $B(x,y)=\Gamma(x+y)^{-1}\Gamma(x)\Gamma(y)$ is the beta function. Thus, 
\begin{equation*}
  \int_0^\infty t^{\gamma-1}\tau \big[(V+t)_{-}^p\big] dt =  B(\gamma,p+1)\int_0^\infty v^{p+\gamma} d(\tau_*E_v)=B(\gamma,p+1)  
  \tau \big[|V_{-}|^{p+\gamma}\big]. 
\end{equation*}
Combining this with~(\ref{eq:LT.Riesz-mean-tau}) gives the result. 
\end{proof}

\begin{remark}
 For $\gamma=1$ and $p=n/2$ with $n\geq 3$ the LT inequality~(\ref{eq:LT-Inequality}) asserts that if $V=V^*\in L_{n/2}(\T^n_\theta)$, then 
 \begin{equation*}
 \sum_j \left|\nlambda_j\big(\dot{\Delta}+\dot{\lambda}(V)\big)\right| \leq L_{n} \tau\big[|V_{-}|^{\frac{n}2+1}\big], 
\end{equation*}
where the best constant $L_{n}=L_{\frac{n}{2},1,n}$ satisfies 
\begin{equation*}
 L_n\leq \frac{2}{n+2} \Gamma\big(\frac{n}2+1\big)c_+(n)^n  \nu_0(n)
\end{equation*}
\end{remark}

\begin{remark}
 For the ordinary torus $\T^n$, i.e., $\theta=0$, Lieb-Thirring inequalities were obtained by Ilyin~\cite{Il:JST12} for $\gamma=1$, $n=2$ and $p=1$ (see 
 also~\cite{IL:StPMJ20, ILZ:MN19, ILZ:JFA20}) and by Ilyin-Laptev~\cite{IL:SM16} for $\gamma\geq 1$, $2\leq n\leq 19$, and $p=n/2$. Even for the ordinary torus $\T^n$ the Lieb-Thirring inequalities for $0<\gamma<1$ or $p\neq n/2$ seem to be new. 
\end{remark}

\begin{remark}
 We can get LT inequalities in the critical case $p=1$ and get better bounds for the best LT constants $L_{n,\gamma,p}$. For $p\geq 1$, as $(U^k)_{k\in \Z^n}$ is an orthonormal basis of $L_2(\T^n_\theta)$ consisting of unitaries in $L_\infty(\T^n_\theta)$, we may proceed along the lines of the proof of the LT inequalities for $\T^2$ in~\cite{Il:JST12} to get LT inequalities~(\ref{eq:LT-Inequality}) with the bounds,
\begin{equation*}
 L_{p,\gamma,n} \leq \gamma 2^\gamma \frac{\Gamma\left(\frac\gamma{2}\right)\Gamma\left(p+\frac\gamma{2}+1\right)}{\Gamma(p+\gamma+1)}  Z\big(n,p,{\gamma}\big), 
\end{equation*}
 where we have set 
\begin{equation*}
 Z\big(n,p,{\gamma}\big):=  \sup_{\mu>0} \mu^{\frac\gamma2} \Tr\bigg[ \big(\dot{\Delta}^\frac{n}{2p}+\mu\big)^{p+\frac{\gamma}2}\bigg] =  \sup_{\mu>0} \mu^{\frac\gamma2} \sum_{k\in \Z^n\setminus 0} \big(|k|^{\frac{n}{p}}+\mu\big)^{p+\frac{\gamma}2}. 
\end{equation*}
For $n=2$ and $p=\gamma=1$ Ilyin~\cite{Il:JST12} found that $Z(2,1,1)<2\pi$. In general, by using Melin's transform arguments it can be shown that
\begin{equation*}
 Z\big(n,p,{\gamma}\big) \leq  \frac{\Gamma\left(\frac\gamma{2}\right) \Gamma(p+1)}{ \Gamma\left(p+\frac\gamma{2}\right)}\nu_0(n).  
\end{equation*}
This gives the upper bound, 
\begin{equation*}
 L_{p,\gamma,n} \leq \gamma 2^\gamma \left(p+\frac\gamma2\right)
  \frac{\Gamma\left(\frac\gamma{2}\right)^2  \Gamma(p+1)}{\Gamma(p+\gamma+1)} \nu_0(n). 
\end{equation*}
\end{remark}

\begin{remark}
 In~\cite{IL:SM16} Ilyin-Laptev used the dimension lifting approach of Laptev-Weidl~\cite{LW:AM00} to get a better bound for the LT constant 
 $L_{n}=L_{\frac{n}{2},1,n}$  on the ordinary torus $\T^n$ for 
 $2\leq n \leq 19$. It seems likely that similar approach can be used on NC tori, since the Laplacian $\Delta$ on $\T^n_\theta$ is isospectral to the Laplacian $\Delta$ on the ordinary torus $\T^n$. 
\end{remark}

As is well known the LT inequality for $\gamma=1$ and $p=n/2$ is equivalent to a Sobolev inequality (see~\cite[Theorem 4]{LT:SMP76}; see 
also~\cite{Fr:Survey20, Si:AMS15a}). Likewise, as a consequence of the above LT inequalities we shall obtain the following Sobolev's inequality.

\begin{theorem}[Sobolev Inequality on NC Tori]\label{thm:LT.Sobolev}  
Assume $n\geq 3$. There is a constant $K_n\geq 0$ such that, for every family $\{u_0, \ldots, u_N\}$ in $\dot{W}_2^{1}(\T^n_\theta)$ which is orthonormal in $L_2(\T^n_\theta)$, we have
\begin{equation}\label{eq:LT.Sobolev-ineq}
 \sum_{\ell=0}^N \tau\big[ |\nabla u_\ell|^2\big] \geq K_n \tau\left[ \bigg(\sum_{\ell=0}^N |u_\ell|^2\bigg)^{\frac{n+2}{n}}\right]. 
\end{equation}
 Moreover, the best constant $K_n$ is related to the best LT constant  $L_n=L_{\frac{n}{2},1,n}$ by
\begin{equation}\label{eq:LT.Ln-Kn}
 K_n= \frac{n}{n+2} \bigg( \frac{n+2}{2}L_n\bigg)^{-\frac{n}{2}}. 
\end{equation}
 \end{theorem}
\begin{proof}
The proof essentially follows the outline of the proof in~\cite[\S4.2]{LS:CUP10} for Schr\"odinger operators on $\R^n$ (see also~\cite[\S3.1]{Fr:Survey20}). 
Let $0\leq V\in L_{n/2}(\T^n_\theta)$, and set $A=\dot{\Delta}-\dot{\lambda}(V)$. Let $\{u_0, \ldots, u_N\}\subseteq \dot{W}_2^{1}(\T^n_\theta)$ be an orthonormal family in $L_2(\T^n_\theta)$. Note that the $u_\ell$ are in the domain of the quadratic form $Q_A$, since $Q_A$ is the restriction to $\dot{W}_2^{1}(\T^n_\theta)$ of $Q_\Delta+Q_{-V}$. Let $\Pi= \sum_\ell  \langle  u_\ell |u_\ell  \rangle$ be the orthogonal projection onto $\op{Span}\{u_0, \ldots, u_N\}$. We have
\begin{equation*}
 \sum_{\ell =0}^N  Q_A(u_\ell,u_\ell)\geq  \sum_{\ell =0}^N -Q_{A_-}(u_\ell,u_\ell)= \Tr\big[\Pi A_{-}\Pi\big]=-\sum_{j\geq 0}\mu_j(\Pi A_{-}\Pi). 
\end{equation*}
 As $\mu_j(\Pi A_{-}\Pi)\leq \mu_j(A_{-})=|\nlambda_j(A)|$, we get 
\begin{equation}\label{eq:LT.Sobolev-QA-lambda}
  \sum_{\ell =0}^N  Q_A(u_\ell,u_\ell)\geq - \sum _{j\geq 0} |\nlambda_j(A)|. 
\end{equation}
This inequality is an instance of the variational principle for sums of eigenvalues (see, e.g., \cite{Fan:1949}).

We also have 
\begin{align*}
 Q_A(u_\ell,u_\ell) & = \acou{\Delta u_\ell}{u_\ell} - \acou{\lambda(V)u_\ell}{u_\ell}\\ 
 & = -\sum_{j=1}^n \acou{\partial_j^2u_\ell}{u_\ell}   -\tau\big[ V|u_\ell|^2\big]\\
 &= \tau\big[|\nabla u_\ell|^2\big] -\tau\big[ V|u_\ell|^2\big].  
\end{align*}
Here $|\nabla u_\ell|^2=\sum |\partial_j u_\ell|^2\in L_1(\T^n_\theta)$. Moreover, as $W_2^1(\T^n_\theta)\subseteq L_{\frac{2n}{n-2}}(\T^n_\theta)$ by Proposition~\ref{prop:Sobolev-embeddingLp}, we see that $|u_\ell|^2\in L_{\frac{n}{n-2}}(\T^n_\theta)$, and hence $V|u_\ell|^2\in L_1(\T^n_\theta)$. Thus, if we set $\rho=\sum |u_\ell|^2\in  L_{\frac{n}{n-2}}(\T^n_\theta)$, then we obtain
\begin{equation}\label{eq:LT.Sobolev-QA-rho}
  \sum_{\ell =0}^N  Q_A(u_\ell,u_\ell) =  \sum_{\ell =0}^N \tau\big[|\nabla u_\ell|^2\big] - \tau\big[ \rho V\big].
\end{equation}

If we further assume that $V\in L_{\frac{n}{2}+1}(\T^n_\theta)$, then the LT inequality~(\ref{eq:LT-Inequality}) for $p=n/2$ and $\gamma=1$ delivers
\begin{equation*}
  \sum_{j\geq 0} |\nlambda_j(A)| \leq L_n \tau\big[V^{\frac{n}{2}+1}\big]. 
\end{equation*}
Combining this with~(\ref{eq:LT.Sobolev-QA-lambda}) and~(\ref{eq:LT.Sobolev-QA-rho}) shows that, for all $0\leq V\in L_{1+n/2}(\T^n_\theta)$, we have
\begin{equation}\label{eq:LT.Sobolev-nabla-V-rho}
 \sum_{\ell =0}^N \tau\big[|\nabla u_\ell|^2\big] \geq \tau\big[\rho V - L_n V^{\frac{n}{2}+1}\big].  
\end{equation}

The original argument of Lieb-Thirring relies on the following observation. Given $a,b,\alpha>0$, the function $\varphi(t)=at-bt^{\alpha +1}$, $t>0$, reaches its maximum at $t_0=[(\alpha+1)b]^{-1/\alpha}a^{1/\alpha}$ and the maximum value then  is equal to 
\begin{equation*}
\varphi(t_0)= t_0(a-bt_0^\alpha)= at_0\big(1 -\frac1{\alpha+1}\big)= \frac{\alpha}{\alpha+1}\big[(\alpha+1)b\big]^{-\frac1{\alpha}}a^{1+\frac1\alpha}>0 . 
\end{equation*}
Likewise, for $\alpha =n/2$, $a=\rho$ and $b=L_n$, by taking $V=(\frac12(n+2)L_n)^{-2/n}\rho^{2/n}\in L_{\frac{n^2}{2(n-2)}}(\T^n_\theta)\subseteq  L_{\frac{n}{2}+1}(\T^n_\theta)$, we have $\rho V - L_n V^{\frac{n}{2}+1}=\frac{n}{n+2}(\frac12 (n+2)L_n)^{-2/n}\rho^{(n+2)/n}$. In this case~(\ref{eq:LT.Sobolev-nabla-V-rho}) gives
\begin{equation*}
  \sum_{\ell =0}^N \tau\big[|\nabla u_\ell|^2\big] \geq \frac{n}{n+2}\bigg(\frac{n+2}2L_n\bigg)^{-\frac{2}{n}}\tau\big[\rho^{\frac{n+2}{n}}\big]. 
\end{equation*}
This shows that the Sobolev inequality~(\ref{eq:LT.Sobolev-ineq}) holds with a best constant $K_n$ such that
\begin{equation}\label{eq:LT.ineq-Kn-Ln}
 K_n\geq \frac{n}{n+2}\bigg(\frac{n+2}2L_n\bigg)^{-\frac{2}{n}}. 
\end{equation}

Conversely, let $V=V^*\in L_{\frac{n}{2}+1}(\T^n_\theta)$, and set $\dot{H}_{V}=\dot{\Delta}+\dot{\lambda}(V)$. Set $N=N^{-}(H_V)-1$, and let  $\{u_0, \ldots, u_N\}\subseteq \dot{W}_2^{1}(\T^n_\theta)$ be an orthonormal family such that $\dot{H}_Vu_\ell = \nlambda_{\ell}(H_V)u_\ell$ for $\ell=0, \ldots, N$. We have 
\begin{equation*}
 \sum_{\ell=0}^N \left|\nlambda\big(\dot{H}_V\big)\right|= -\sum_{\ell=0}^N Q_{\dot{H}_V}(u_\ell,u_\ell) \leq -\sum_{\ell=0}^N Q_{\dot{H}_{-V_{-}}}(u_\ell,u_\ell). 
\end{equation*}
As above, set $\rho = \sum |u_\ell|^2\in L_{\frac{n}{n-2}}(\T^n_\theta)$.  By using~(\ref{eq:LT.Sobolev-QA-rho}) and the Sobolev inequality~(\ref{eq:LT.Sobolev-ineq}) we get 
\begin{equation}\label{eq:LT.sum-V-rho} 
 \sum_{\ell=0}^N \left|\nlambda\big(\dot{H}_V\big)\right| \leq -\sum_{\ell =0}^N \tau\big[|\nabla u_\ell|^2\big] + \tau\big[ \rho V_{-}\big] \leq 
 -K_n \tau\big[\rho^{\frac{n+2}{n}}\big] +  \tau\big[ \rho V\big]. 
\end{equation}

Set $q=1+n/2$ and $r=1+2/n$, so that $q^{-1}+r^{-1}=1$. Recall that $V\in L_{q}(\T^n_\theta)$ and $\rho\in L_{\frac{n}{n-2}}(\T^n_\theta)\subseteq L_r(\T^n_\theta)$. Thus, by H\"older's inequality, 
\begin{equation*}
 \tau\big[ \rho V_{-}\big] \leq \tau\big[|V_{-}|^q\big]^{\frac{1}{q}}  \tau\big[\rho^r\big]^{\frac{1}{r}}. 
\end{equation*}
Recall the inequality $ab\leq q^{-1}a^q + r^{-1}b^r$ for $a,b>0$ (this is a convexity inequality for the function $t\rightarrow a^{1-t}b^t$). Applying it to $a=\epsilon \tau\big[|V_{-}|^q\big]^{\frac{1}{q}} $ and $b=\epsilon^{-1}\tau\big[\rho^r\big]^{\frac{1}{r}}$ with $\epsilon>0$ gives
\begin{equation*}
 \tau\big[ \rho V_{-}\big] \leq  \frac{1}{p}\epsilon^q\tau\big[|V_{-}|^q\big] + \frac{1}{r} \epsilon^{-r} \tau\big[\rho^r\big]. 
 \end{equation*}
Combining this with~(\ref{eq:LT.sum-V-rho}) gives 
\begin{equation*}
  \sum_{\ell=0}^N \left|\nlambda\big(\dot{H}_V\big)\right|\leq 
 \bigg(\frac{1}{r} \epsilon^{-r}-K_n\bigg) \tau\big[\rho^{r}\big] +  \frac{1}{q}\epsilon^q\tau\big[|V_{-}|^q\big]. 
\end{equation*}
If we choose $\epsilon$ so that $r^{-1} \epsilon^{-r}=K_n$, i.e., $\epsilon = (rK_n)^{-1/r}$, then we get 
\begin{equation*}
  \sum_{\ell=0}^N \left|\nlambda\big(\dot{H}_V\big)\right|\leq  \frac{1}{q}\epsilon^q\tau\big[|V_{-}|^q\big] \qquad \text{for all}\ V=V^*\in L_{\frac{n}{2}+1}(\T^n_\theta). 
\end{equation*}
This reproves the LT inequality~(\ref{eq:LT-Inequality}) for $p=n/2$ and $\gamma=1$ and shows that the best LT constant $L_n$ satisfies
\begin{equation*}
 L_n \leq  \frac{1}{q}\epsilon^q = \frac{1}{q} (rK_n)^{-\frac{q}{r}}=\frac{2}{n+2}\bigg( \frac{n}{n+2}K_n\bigg)^{-\frac{n}{2}}. 
\end{equation*}
Combining this with~(\ref{eq:LT.ineq-Kn-Ln}) gives~(\ref{eq:LT.Ln-Kn}). The proof is complete. 
\end{proof}


\begin{remark}
 For $n=2$ and $\theta=0$ Ilyin-Laptev-Zelik~\cite{ILZ:JFA20} got the lower bound $K_2\geq 32/(3\pi)$. It would be interesting to see if we could get a similar bound for NC 2-tori. 
\end{remark}


\appendix 

\section{Proof of the Birman-Schwinger Principle}\label{app:BSP}
In this appendix, for reader's convenience, we include proofs of  Proposition~\ref{prop:BSP.abstractBSP} and Corollary~\ref{cor:CLR.BKP-qinf}.

\begin{proof}[Proof of Proposition~\ref{prop:BSP.abstractBSP}]
Given $\lambda<0$, let $\xi\in \sH_+\setminus 0$ and set $\eta=(H-\lambda)^{1/2}\xi\in \sH$, i.e., $\xi=(H-\lambda)^{-1/2}\eta$. By using~(\ref{eq:BSP.H-extension}) we get
\begin{align*}
 Q_{H_V}(\xi,\xi)<\lambda\scal{\xi}{\xi} & \Longleftrightarrow  Q_{H}(\xi,\xi)-\lambda\scal{\xi}{\xi} <- Q_V(\xi,\xi),\\ 
 & \Longleftrightarrow  \acou{(H-\lambda)\xi}{\xi}<- \acou{V\xi}{\xi}. 
\end{align*}
In the same way as in~(\ref{eq:BSP.sqrt}) we have
\begin{equation*}
 \acou{(H-\lambda)\xi}{\xi}= \bigacou{(H-\lambda)^{\frac12}\eta}{(H-\lambda)^{-\frac12}\eta}=\scal{\eta}{\eta}. 
\end{equation*}
Moreover, by using~(\ref{eq:BSP.KV-QV}) we get 
\begin{equation*}
-\acou{V\xi}{\xi}= -\bigacou{V(H-\lambda)^{-\frac12}\eta}{(H-\lambda)^{-\frac12}\eta}=\scal{K_V(\lambda)\eta}{\eta}=Q_{K_V(\lambda)}(\eta,\eta). 
\end{equation*}
Thus,
\begin{equation}
 Q_{H_V}(\xi,\xi)<\lambda\scal{\xi}{\xi}  \Longleftrightarrow  Q_{K_V(\lambda)}(\eta,\eta)> \scal{\eta}{\eta}. 
 \label{eq:BSP.QHV-QKV}
\end{equation}

As $(H-\lambda)^{1/2}$ is a linear isomorphism from $\sH_+$ onto $\sH$, it follows from~(\ref{eq:BSP.QHV-QKV}) that it  induces a one-to-one correspondence between the subspaces in $\sF^-(H_V;\lambda)$ and those in $\sF^+(K_V(\lambda);1)$. This correspondence preserves the dimension. Thus, by using~(\ref{eq:BSP.Glazman})--(\ref{eq:BSP.Glazman0}) we get
\begin{align*}
 N^-(H_{V};\lambda)& = \max\{ \dim F; \ F\in \sF^-(H_V;\lambda)\},\\
 & = \max\{ \dim F^+; \ F^+\in \sF^+(K_V(\lambda);1)\}= N^+\left(K_V(\lambda); 1\right).
\end{align*}
This completes the proof of Proposition~\ref{prop:BSP.abstractBSP}. 
\end{proof}
 
\begin{remark}
 The above proof of the abstract Birman-Schwinger principle is due to Birman-Solomyak~\cite{BS:1989}. It is much simpler than various other known proofs of the Birman-Schwinger principle, even in the case of Schr\"odinger operators on $\R^n$.  In fact, we have merely recasted the proof of~\cite{BS:1989} into the framework of~\cite[\S7.9]{Si:AMS15}. 
\end{remark}

\begin{proof}[Proof of Corollary~\ref{cor:CLR.BKP-qinf}] 
Let $\lambda<0$ and set $A=(H+1)^{1/2}(H-\lambda)^{-1/2}$.  The fact that the operator  $(H+1)^{-1/2}V(H+1)^{-1/2}$ is in $\sL_{q,\infty}$ ensures us that 
$K_V(\lambda) \in \sL_{q,\infty}$, since 
  $A$ is a bounded operator on $\sH$, and we have 
\begin{equation}
 K_V(\lambda)=-A^*(H+1)^{-1/2}V(H+1)^{-1/2}A. 
 \label{eq:BSP.KV-Hl-H1} 
\end{equation}
 
Bearing this in  mind, set $N=N^-(H_V;\lambda)$. We may assume $N\geq 1$, since otherwise the inequality~(\ref{eq:CLR.BKP-qinf-lambda}) is trivially satisfied. The Birman-Schwinger principle~(\ref{eq:CLR.BKP}) asserts that $N^+(K_V(\lambda);1)=N$. As $V\leq 0$, and hence $K_V\geq 0$, we note that $N^+(K_V(\lambda);1)$ is the number of singular values of $K_V(\lambda)$ that are~$>1$. Thus, $K_V(\lambda)$ has exactly $N$ singular values~$> 1$. In particular, we have $\mu_{N-1} (K_V(\lambda))\geq 1$. Therefore, in view of the definition~(\ref{def:lp_infty_quasinorm}) of the quasi-norm of $\sL_{p,\infty}$, we get 
\begin{equation}\label{eq:CLR.BKP-qinf2}
 N \leq N\mu_{N-1}\big(K_V(\lambda)\big)^p \leq \biggl( \sup_{j\geq 0}  (j+1)^{\frac1p}\mu_{j}\big(K_V(\lambda)\big)\biggr)^p=\big\|K_V(\lambda)\big\|_{\sL_{p,\infty}}^p. 
\end{equation}
 This proves Corollary~\ref{cor:CLR.BKP-qinf}. 
\end{proof}

 \subsection*{Acknowledgements} 
The warmest thanks of the authors go to Fedor Sukochev for his constant support and encouragements throughout the preparation of this paper, and to Rupert Frank and Grigori Rozenblum for their careful reading of earlier versions of the manuscript and for sharing their insights on the CLR inequality. The authors also thank Galina Levitina and Dmitriy Zanin for various discussions related to the subject matter of the article. The 2nd named author further thanks University of New South Wales and University of Ottawa for their hospitality during the preparation of this manuscript. 

\subsection*{Data availability statement} Data sharing is not applicable to this article as no new data were created or analyzed in this study.

\end{document}